\documentclass{amsart}

\usepackage{graphicx}
\usepackage{amssymb, amsmath, amsthm}
\usepackage{geometry}
\usepackage{lipsum}
\usepackage{enumerate}
\usepackage{dsfont}
\usepackage{mathrsfs}
\usepackage{xcolor}

\input xy
\xyoption{all}

\usepackage{hyperref}

\newtheorem{theorem}{Theorem}[section]
\newtheorem{lemma}[theorem]{Lemma}
\newtheorem{corollary}[theorem]{Corollary}
\newtheorem{proposition}[theorem]{Proposition}
\newtheorem{lemma-definition}[theorem]{Lemma-Definition}

\theoremstyle{definition}
\newtheorem{definition}[theorem]{Definition}
\newtheorem{example}[theorem]{Example}

\newtheorem{parrafo}[theorem]{\unskip}

\theoremstyle{remark}
\newtheorem{remark}[theorem]{Remark}
\numberwithin{equation}{section}

\newcommand{\mb}{\mathbb}
\newcommand{\mbvtex}{\mathbb}

\newcommand{\xvtex}{\mathbf{x}}

\newcommand{\yvtex}{\mathbf{y}}
\newcommand{\z}{\mathbf{z}}

\newcommand{\e}{\mathbf{e}}

\newcommand{\Diff}{\mathit{Diff}}
\newcommand{\m}{\mathfrak{m}}
\newcommand{\mvtex}{\mathfrak{m}}

\newcommand{\pvtex}{\mathfrak{p}}
\newcommand{\Spec}{\operatorname{Spec}}

\makeatletter
\def\@seccntformat#1{\@ifundefined{#1@cntformat}%
	{\csname the#1\endcsname\quad}
	{\csname #1@cntformat\endcsname}}
\newcommand{\section@cntformat}{\S\thesection\quad}
\newcommand{\subsection@cntformat}{\S\thesubsection\quad}
\makeatletter
\begin{document}

	\title{Multiplicity along points of a radicial covering of a regular variety}
	
	\author{D. Sulca}
	\address{Facultad de Matem\'atica, Astronom\'ia, F\'isica y Computaci\'on, Universidad Nacional de C\'ordoba, Ciudad Universitaria, X5000HUA C\'ordoba, Argentina.}
	
	\email{sulca@famaf.unc.edu.ar}
	
	\author{O. E. Villamayor U.}
	\address{Dpto de Matem\'aticas, Universidad Aut\'onoma de Madrid and Instituto de Ciencias Matem\'aticas CSIC-UAM-UC3M-UCM, Ciudad Universitaria de Cantoblanco, 28049 Madrid, Spain.
	}
	\email{villamayor@uam.es}
	
	\thanks{Partially supported from the Spanish Ministry of Economy and Competitiveness, through the ``Severo Ochoa'' Program for Centres of Excellence in R\&D (SEV-2015-0554)}
	
	
	\date{\today}

	
	\begin{abstract}
		We study the maximal multiplicity locus of a variety $X$ over a field of
		characteristic $p>0$ that is provided with a finite surjective radicial
		morphism $\delta:X\rightarrow V$, where $V$ is regular, for example, when
		$X\subset\mathbb{A}^{n+1}$ is a hypersurface defined by an equation of
		the form $T^{q}-f(x_{1},\ldots,x_{n})=0$ and $\delta$ is the projection
		onto $V:=\operatorname{Spec}(k[x_{1},\ldots,x_{n}])$. The
		multiplicity along
		points of $X$ is bounded by the degree, say $d$, of the field extension
		$K(V)\subset K(X)$. We denote by $F_{d}(X)\subset X$ the set of points
		of multiplicity~$d$. Our guiding line is the search for invariants of
		singularities
		$x\in F_{d}(X)$ with a good behavior property under blowups
		$X'\rightarrow X$ along regular centers included in $F_{d}(X)$, which we
		call \emph{invariants with the pointwise inequality property}.
		
		A finite radicial morphism $\delta:X\to V$ as above will be expressed
		in terms of an $\mathcal{O}_{V}^{q}$-submodule
		$\mathscr{M}\subseteq\mathcal{O}_{V}$. A~blowup $X'\to X$ along a regular
		equimultiple center included in $F_{d}(X)$ induces a blowup
		$V'\to V$ along a regular center and a finite morphism
		$\delta':X'\to V'$. A~notion of transform of the
		$\mathcal{O}_{V}^{q}$-module $\mathscr{M}\subset\mathcal{O}_{V}$ to an
		$\mathcal{O}_{V'}^{q}$-module $\mathscr{M}'\subset\mathcal{O}_{V'}$ will
		be defined in such a way that $\delta':X'\to V'$ is the radicial morphism
		defined by $\mathscr{M}'$. Our search for invariants relies on techniques
		involving differential operators on regular varieties and also on logarithmic
		differential operators. Indeed, the different invariants we introduce and
		the stratification they define will be expressed in terms of ideals obtained
		by evaluating differential operators of~$V$ on $\mathcal{O}_{V}^{q}$-submodules
		$\mathscr{M}\subset\mathcal{O}_{V}$.
	\end{abstract}

	\maketitle
	
	\section{Introduction}
	
	\noindent Let $k$ be a field of characteristic $p>0$, and fix $q=p^{e}$, a power
	of $p$. Let $X\subset\mathbb{A}^{n+1}_{k}$ be a hypersurface defined by
	a polynomial equation of the form $T^{q}-f(x_{1},\ldots,x_{n})=0$. The
	maximal multiplicity along points of $X$ is $\leq q$, and we denote by
	$F_{q}(X)\subset X$ the set of points where the multiplicity is $q$. This
	subset is closed, as is indicated below, and when it is nonempty, it can
	be thought of as the set of the worst singularities. Our guideline is the
	search for invariants of singularities $x\in F_{q}(X)$. We focus our attention
	on the projection $X\to V:=\mathbb{A}_{k}^{n}$, which is a finite morphism
	of generic rank $q$, rather than on the immersion
	$X\subset\mathbb{A}_{k}^{n+1}$. The following theorem, which collects
	results taken from
	\cite[Sections~5 and 6]{villamayor2014equimultiplicity}, serves as starting
	point in our discussion.
	
	\begin{theorem}%
		\label{finite_morphisms_and_closed_subsets}
		Let $\delta:X\rightarrow V$ be a finite and surjective morphism of Noetherian
		integral schemes, where $V$ is excellent and regular, and
		$\dim\mathcal{O}_{V,x}$ is constant along closed points. We set
		$d:=[K(X):K(V)]$, the generic rank. Then the multiplicity along points
		of $X$ is at most $d$. Let $F_{d}(X)$ denote the set of points where the
		multiplicity is $d$ and assume that it is nonempty. Then the following
		holds.
		\begin{enumerate}[(2)]
			\item[(1)]$F_{d}(X)\subset X$ is closed and homeomorphic to its image
			$\delta(F_{q}(X))\subset V$ (via $\delta$); moreover,
			$F_{d}(X)=\delta^{-1}(\delta(F_{d}(X))$.
			\item[(2)] An integral subscheme $Y\subset X$ included in $F_{d}(X)$ is regular
			if and only if its schematic image $\delta(Y)\subset V$ is also regular.
			In that case the blowup of $X$ along $Y$ and the blowup of $V$ along
			$\delta(Y)$ fit into a commutative diagram
			%
			%
			\begin{equation}
				\label{2421} \xymatrix{X\ar[d]_{\delta} &\ar[l] X_{1}\ar[d]^{\delta_{1}}\\
					V&\ar[l] V_{1} }
			\end{equation}
			for a unique morphism $\delta_{1}:X_{1}\rightarrow V_{1}$. This morphism
			is again finite and surjective and has generic rank $d$. In particular,
			the maximal multiplicity of $X_{1}$ is $\leq d$.
		\end{enumerate}
	\end{theorem}

\begin{remark}%
	\label{remark_of_first_theorem}
	\begin{enumerate}[(iii)]
		\item[(i)] The closeness of $F_{d}(X)$ under these general hypothesis
		was proved
		by Dade \cite{dade1960multiplicity}; see
		\cite[Remark 6.13]{villamayor2014equimultiplicity}.
		\item[(ii)] The approach of viewing a variety, or a scheme, as a
		finite cover
		of a regular one, has been shown to be efficient when studying the multiplicity
		along singular points. For instance, Lipman
		\cite{lipman1982equimultiplicity} uses it to discuss the multiplicity of
		complex analytic and algebraic varieties and the effect of blowing up along
		equimultiple centers. This idea was taken further in
		\cite{villamayor2014equimultiplicity}, by the second author, to give an
		alternative proof of resolution of singularities in characteristic zero
		by using the multiplicity as the main invariant.
		\item[(iii)] Assume that $V=\Spec(S)$, where $S$ is of finite type
		over a field
		$k$, and that $X\subset V\times\mathbb{A}_{k}^{1}$ is defined by a polynomial
		of the form $T^{d}+f_{1} T^{d-1}+\cdots+f_{d}$, where $f_{i}\in S$, and
		$d$ is not divisible by the characteristic of $k$. By using elimination
		theory we can produce an $\mathcal{O}_{V}$-ideal $\mathcal{J}$ and a positive
		integer $b$ such that
		$\delta(F_{d}(X))=\operatorname{Sing}(\mathcal{J},b):=\{x\in V: \nu_{x}(
		\mathcal{J})\geq b\}$ and such that this description holds under any sequence
		of blowups that arises by applying Theorem \ref{finite_morphisms_and_closed_subsets}(2) successively. Here
		$\nu_{x}(\mathcal{J})$ denotes the order of $\mathcal{J}$ at the regular
		local ring $\mathcal{O}_{V,x}$. We refer to
		\cite[Section~3]{villamayor2014equimultiplicity} for a detailed discussion
		(see also \cite{sulca2018introduction}). We only stress the fact that
		$(\mathcal{J},b)$ is obtained quite explicitly: it is determined by the
		weighted homogeneous polynomials on the coefficients
		$f_{1},\ldots,f_{d}$ that are invariant under the change of variables
		$T\mapsto T+\lambda$, $\lambda\in S$, such as the discriminant
		\cite[Theorem 3.5]{villamayor2014equimultiplicity}.
	\end{enumerate}
\end{remark}

We return to our original situation, where $X$ is a hypersurface defined
by a purely inseparable equation $T^{q}-f(x_{1},\ldots,x_{n})=0$, and
where $\delta:X\to V$ is the projection. When applying the theorem to
this morphism, the resulting $\delta_{1}:X_{1}\to V_{1}$ is locally in
the same situation as $X\to V$, that is, defined by a purely inseparable
equation of degree $q$. This process can be repeated as long as the maximal
multiplicity along points of $X_{1}$ is $q$. The intention in this first
part of the introduction is to illustrate how the study of the singularities
along points $x\in F_{q}(X)$ and of blowups along regular centers included
in $F_{q}(X)$ naturally leads to the consideration of
$\mathcal{O}_{V}^{q}$-submodules $\mathscr{M}\subset\mathcal
{O}_{V}$ and
transformations of $\mathcal{O}_{V}^{q}$-submodules under suitable blowups
of the regular scheme $V$.

In contrast to the situation described in Remark \ref{remark_of_first_theorem}(iii), with $S=k[x_{1},\ldots,x_{n}]$, here
the only nonzero coefficient of the equation (apart from the principal
one) is the constant coefficient $f:=f(x_{1},\ldots,x_{n})$, and the change
of variables $T':=T+g$ with $g\in S$ produces the equation
${T'}^{q}-(f+g^{q})=0$, that is, $f$ is changed by $f+g^{q}$. This tells
us that we should not consider the ideal
$\langle f\rangle\subset S$ if we attempt to find a substitute for
$(\mathcal{J},b)$, since this ideal has nothing to do with
$\langle f+g^{q}\rangle$. Instead, we may consider the element $f$ only
up to equivalence, where $f\sim f'$ if $f-f'\in S^{q}$. In the case that
$S$ is a more general regular ring, we can also perform a change of variable
of the form $T':= u^{-1} T$ with $u\in S^{*}$, and the equation will take
the form $T^{q}-f u^{q}$. This suggests that it is a better idea to take
the full $S^{q}$-submodule $M\subset S$ generated by $f$ and to consider
$S^{q}$-submodules only up to equivalence, where $M\sim M'$ if
$M+S^{q}=M'+S^{q}$. Finally, if we are willing to completely forget the
immersion $X\subset\mathbb{A}_{k}^{n+1}$ and regard $X$ simply as a
$V$-scheme, then we have to look at the whole $S^{q}$-subalgebra of
$S$ generated by $f$, denoted by $S^{q}[f]$.

In general, we will attach to any $S^{q}$-submodule $M\subset S$ a finite
morphism $X\to V$ so that $M$ and $M'$ define the same morphism if and
only if $S^{q}[M]=S^{q}[M']$. One important step in our search for invariants
is the introduction of assignments of $S$-ideals $M\to I(M)$ so that
$I(M)=I(S^{q}[M])$. Our main device will be higher-order differential operators
that are $S^{q}$-linear, such as
${\partial^{\alpha}}/{\partial x^{\alpha}}$, with
$\alpha=(\alpha_{1},\ldots,\alpha_{n})\neq0$ and
$\sum_{i=1}^{n} \alpha_{i}<q$. To be more precise, for
$i=1,\ldots,q-1$, we denote by $\operatorname{Diff}_{S,+}^{i}$ the $S$-module
of differential operators $D:S\to S$ of order $\leq i$ such that
$D(1)=0$ (see \ref{Differential_operators}); for example,
$\operatorname{Diff}_{S,+}^{1}=\operatorname{Der}_{S}$, the set of
derivations on
$S$. All these operators are $S^{q}$-linear, and we have
$\operatorname{Diff}_{S,+}^{i}(S^{q}[M])=\operatorname
{Diff}_{S,+}^{i}(M)$ (Proposition \ref{Diff(M)=Diff(O^q[M])}).

We return to our hypersurface $X\subset\mathbb{A}_{k}^{n+1}$ defined by
the polynomial $\Theta(T):=T^{q}-f(x_{1},\ldots,x_{n})$. Recall that
$\delta:X\to V=\mathbb{A}_{k}^{n}$ denotes the projection. The following
proposition (reformulated and generalized in Proposition \ref{d(F(X))_is_closed}) offers a first example of the role played by
differential
operators.

\begin{proposition}
	$\delta(F_{q}(X))\subset V$ is the closed subset defined by the $S$-ideal
	$\operatorname{Diff}_{S,+}^{q-1}(f)=\operatorname
	{Diff}_{S,+}^{q-1}(S^{q}[f])$.
\end{proposition}
We remark that this is a direct consequence of the classical fact that
$F_{q}(X)\subset\mathbb{A}_{k}^{n+1}$ (the set of points where the order
of $\Theta(T)$ is $\geq q$) is the closed subset of
$\mathbb{A}_{k}^{n+1}$ defined by the ideal
$\operatorname{Diff}_{S[T]}^{q-1}(\Theta(T))$. Indeed, this ideal is generated
by $\Theta(T)$, $\operatorname{Diff}_{S,+}^{q-1}(\Theta(T))$, and
$({\partial^{i}}/{\partial T^{i}})(\Theta(T))$,
$i=1,\ldots,q-1$, where differential operators of $S$ are extended to
$S[T]$ by acting as $0$ on $T$. Note that
$\operatorname{Diff}_{S,+}^{q-1}(\Theta(T))=\operatorname
{Diff}_{S,+}^{q-1}(f)$ and
that $({\partial^{i}}/{\partial T^{i}})(\Theta(T))=0$ for
$i=1,\ldots,q-1$. The proposition follows from these observations.

The above proposition describes $\delta(F_{q}(X))$ with an ideal of
$S$, which is intrinsically attached to the $V$-scheme $X$. The next natural
step would be stratifying this closed set. This will be done by defining
upper-semicontinuous functions, which in general are defined in terms of
the order of suitable ideals along points of the regular scheme $V$. Recall
that a function $s:Y\to\mathbb{N}$, from a topological space $Y$, is
upper-semicontinuous
if the sets $\{x\in Y: s(x)\geq m\}$, $m\in\mathbb{N}$, are closed.

According to our previous discussion, the order of
$f\in\mathcal{O}_{V,x}$ along points $x\in\delta(F_{q}(X)))$ is not a
good function to look at. Instead, we might consider
%
%
\begin{equation}
	\nu^{(q)}_{\mathfrak{p}}(f):=\sup\{\nu_{\mathfrak{p}}(f+g^{q}):
	g\in S_{\mathfrak{p}}\}.
\end{equation} 
By using this function we obtain the following description.
%
%
\begin{proposition}
	$\delta(F_{q}(X))=\{x\in V:\nu_{x}^{(q)}(f)\geq q\}$.
\end{proposition}
See Proposition \ref{0309} for an equivalent formulation. A~drawback of
the above function along points $x\in\delta(F(X))$ is that it is not
upper-semicontinuous, as it is well known, and we recall it in the following:
%
%
\begin{example}%
	\label{the_q-order_is_not_upper-semicontinuous}
	Let $f:=x_{1}^{q}x_{2}\in k[x_{1},x_{2}]$ with algebraically closed
	$k$. We obtain $\operatorname{Diff}_{S,+}^{q-1}(f)=(x_{1}^{q})$, so that
	$Z:=\delta(F_{q}(X))\subset V$ is the $x_{2}$-axis. At $\xi\in Z$, the
	generic point, we have $\nu_{\xi}^{(q)}(f)=q$. However, at any closed point
	$x=(0,\lambda)$ of $Z$, we have
	$x_{1}^{q}x_{2}=x_{1}^{q}(x_{2}-\lambda)+x_{1}^{q}\lambda\sim x_{1}^{q}(x_{2}-
	\lambda)$ since $\lambda$ is a $q$-power; hence
	$\nu_{x}^{(q)}(f)\geq q+1$ (in fact, it is an equality). Were
	$x\mapsto\nu_{x}^{(q)}(f)$ upper-semicontinuous on $Z$, the set
	$\{x\in Z: \nu_{x}^{(q)}(f)=q\}$ would be open in $Z$, which is not the
	case.
\end{example}
We will use differential operators on the regular ring $S$ to define an
upper-semicontinuous function that coincides ``almost always'' with
$\nu_{x}^{(q)}(f)$. To motivate its definition, we recall that if
$D$ is a differential operator of order $i$ and $f$ has order $m$ at a
point $x\in V$, then $D(f)$ has order at least $m-i$ at $x$. In characteristic
zero, we can always select $D$ so that $D(f)$ has order exactly
$m-i$. This is not longer true in positive characteristic. If we apply
differential operators of order $<q$ to the polynomial
$x_{1}^{q}x_{2}$, we cannot get rid of the factor $x_{1}^{q}$. However,
${\partial}/{\partial x_{2}}$, which has degree 1, lowers the order
of $x_{1}^{q}x_{2}$ at the origin in exactly one unit. Similarly,
${\partial^{p}}/{\partial x_{2}^{p}}$ has degree $p$ and lowers the
order of $x_{1}^{q}x_{2}^{p}$ at the origin in exactly $p$ units.
%
%
\begin{definition}
	For $x\in V$, we define
	\begin{equation*}
		\eta_{x}(f)=\inf\{\nu_{x}(\operatorname{Diff}_{S,+}^{i}(f))+i
		| i=1, \ldots,q-1\}.
	\end{equation*}
\end{definition}
This defines an upper-semicontinuous function on $V$ with values in
$\mathbb{N}$: $\eta_{x}(f)\geq m$ if and only if $x$ belongs to the
intersection
of the closed subsets
$\{x\in V | \nu_{x}(\operatorname{Diff}_{S,+}^{i}(f)))\geq m-i\}$. We stress
the fact that this function is intrinsically attached to the $V$-scheme
$X$, that is, if $S^{q}[f]=S^{q}[g]$, then
$\eta_{x}(f)=\eta_{x}(g)$. The fact that $\eta_{x}(f)$ coincides with
$\nu_{x}^{(q)}(f)$ almost always on $\delta(F_{q}(X))$ will be made precise
in Lemma \ref{Alternative_Lemma}. The next proposition, reformulated and
generalized in Proposition \ref{Describing_Sing(M,a)_with_eta(M)}, shows
that $\delta(F_{q}(X))$ can be described by this function.
%
%
\begin{proposition}%
	\label{prop-intro-1}
	For $x\in V$, $x\in\delta(F_{q}(X))$ if and only if
	$\eta_{x}(f)\geq q$.
\end{proposition}
%
%
\begin{example}
	Set $f=x_{1}^{q}x_{2}$, as in the previous example. We obtain
	$\operatorname{Diff}_{S,+}^{i}(f)=(x_{1}^{q})$ for all $i=1,\ldots
	,q-1$. Therefore
	$\eta_{x}(f)$ is constant and equal to $q+1$ along
	$\delta(F_{q}(X))$, even at the generic point.
\end{example}
Our definition of $\eta_{x}(f)$ and the above proposition, which shows
that $\eta_{x}(f)$ can be used to stratify $\delta(F_{q}(X))$, offer
a first glimpse of the role played by the collection of ideals
%
%
\begin{equation}
	(\operatorname{Diff}_{S,+}^{1}(f),\ldots,
	\operatorname{Diff}_{S,+}^{q-1}(f)).
\end{equation}
We will return to this point later, where collections like this will be
studied under the name of $q$-differential collections. This is the subject
of Section \ref{Section_on_q-Diff}. We now turn the discussion to blowups.

Regular hypersurfaces $H\subset V$ included in $\delta(F_{q}(X))$ will
play a privileged role in our discussion. They will arise naturally when
blowing up along a regular center. For such a hypersurface, the blowup
of $X$ along $\delta^{-1}(H)\subset X$ can be interpreted in terms of
the $S^{q}$-modulo generated by $f$ as follows.
%
%
\begin{proposition}%
	\label{prop2-intro}
	Let $H\subset V$ be a regular hypersurface, say with defining equation
	$h=0$. Then the following are equivalent:
	\begin{enumerate}[(2)]
		\item[(1)] $H$ is included in $\delta(F_{q}(X))$.
		\item[(2)] $f=h^{q}f_{1}+g^{q}$ for some $f_{1},g\in S$.
	\end{enumerate}
	If these conditions are satisfied, then the blowup of $X$ along
	$\delta^{-1}(H)$ is defined by the equation $T^{q}-f_{1}=0$.
\end{proposition}
This will be reformulated in a more general form in Proposition\ref{blow-up_along_codimension_one_subscheme}. The outcome is that blowing
up $X$ along a regular 1-codimensional center included in $F_{q}(X)$ corresponds
to a notion of factorization of a $q$-power of the $S^{q}$-module generated
by $f$ or, more precisely, of the equivalence class of this module. We
will illustrate this in the following example.
%
%
\begin{example}
	Let $X\subset\mathbb{A}_{k}^{3}$ be defined by the equation
	$T^{p}-x_{1}^{3p-2}x_{2}(x_{2}-2x_{1}+x_{1}^{3})+x_{2}^{p}$ ($q=p>2$),
	and set $S=k[x_{1},x_{2}]$. We obtain the $S^{p}$-submodule of $S$ generated
	by $x_{1}^{3p-2}x_{2}(x_{2}-2x_{1}+x_{1}^{3})+x_{2}^{p}$, which is equivalent
	to that generated by $x_{1}^{3p-2}x_{2}(x_{2}-2x_{1}+x_{1}^{3})$. We apply
	$\operatorname{Diff}_{S,+}^{p-1}$ to any of these modules and obtain
	an ideal
	divisible by $x_{1}^{2p}$, which tells us that $\delta(F_{p}(X))$ includes
	the hypersurface $x_{1}=0$. Changing
	$x_{1}^{3p-2}x_{2}(x_{2}-2x_{1}+x_{1}^{3})$ by
	$f_{1}:=x_{1}^{2p-2}x_{2}(x_{2}-2x_{1}+x_{1}^{3})$ corresponds to a blowup
	$X_{1}\to X$ along a regular one-codimensional center. Changing again this
	polynomial by $f_{2}:=x_{1}^{p-2}x_{2}(x_{2}-2x_{1}+x_{1}^{3})$ corresponds
	to a blowup $X_{2}\to X_{1}$ of the same kind. Note that
	$\delta(F_{p}(X_{2}))$ has no one-dimensional component. The maximum
	of $\eta_{x}(f_{2})$ along points of $V$ is $p$, and this maximum is only
	attained at the origin $(0,0)$. By Proposition \ref{prop-intro-1},
	$\delta(F_{p}(X_{2}))=\{(0,0)\}$.
\end{example}

We now show with an example that a blowup of $X$ along a more general regular
center included in $F_{q}(X)$ leads to a notion of transformation of the
$S^{q}$-module assigned to $X$ under the induced blowup of $V$ (Theorem \ref{finite_morphisms_and_closed_subsets}(2)). We will give the details
in the second part of Section \ref{Section_on_modules}; see Corollary \ref{a-transform_as_sequence_of_blow-ups} for a generalized version.
%
%
\begin{example}
	Let us consider the hypersurface $X\subset\mathbb{A}_{k}^{3}$ of equation
	$T^{p}-x_{1}^{p-2}x_{2}(x_{2}-2x_{1}+x_{1}^{3})$. This has attached the
	(equivalence class of the) $S^{p}$-submodule of $S$ generated by
	$f:=x_{1}^{p-2}x_{2}(x_{2}-2x_{1}+x_{1}^{3})$. We obtain from our previous
	example that $\delta(F_{p}(X))=\{(0,0)\}$, and hence
	$F_{p}(X)=\{(0,0,0)\}$. Let $V_{1}\to V$ be the blowup at
	$\{(0,0)\}$, and let $X_{1}\to X$ be the blowup at $\{(0,0,0)\}$. We set
	$S_{1}:=k[x_{1},\frac{x_{2}}{x_{1}}]$ and
	$S_{2}:=k[\frac{x_{1}}{x_{2}},x_{2}]$. Then $X_{1}$ is defined by
	$T_{1}^{p}-\frac{x_{2}}{x_{1}}(\frac{x_{2}}{x_{1}}-2+x_{1}^{2})\in
	S_{1}[T_{1}]$
	at the $x_{1}$-chart and by
	$T_{2}^{p}-(\frac{x_{1}}{x_{2}})^{p-2}(1-2\frac{x_{1}}{x_{2}}+(
	\frac{x_{1}}{x_{2}})^{3}x_{2}^{2})$ at the $x_{2}$-chart. (The strict transform
	of $X$ is already included in the previous two charts.) We easily check
	that the $S_{1}^{p}$-submodule of $S_{1}$ and the $S_{2}^{p}$-submodule
	of $S_{2}$ generated respectively by
	%
	%
	\begin{align}
		\label{equations_of_the_transform} \frac{x_{2}}{x_{1}} \biggl(\frac{x_{2}}{x_{1}}-2+x_{1}^{2}
		\biggr) \quad \text{and}\quad  \biggl(\frac{x_{1}}{x_{2}} \biggr)^{p-2}
		\biggl(1-2 \frac{x_{1}}{x_{2}}+ \biggl(\frac{x_{1}}{x_{2}}
		\biggr)^{3}x_{2}^{2}\biggr)
	\end{align}
	glue to a sheaf of $\mathcal{O}_{V_{1}}^{p}$-submodule
	$\mathscr{M}\subset\mathcal{O}_{V_{1}}$. In addition, if
	$\mathcal{L}\subset\mathcal{O}_{V_{1}}$ denotes the exceptional
	ideal of
	the blowup $V_{1}\to V$, then we have the relation
	$F\mathcal{L}\cdot\mathscr{M}=\mathcal{O}_{V_{1}}^{p}\cdot f$, where
	$F$ stands for Frobenius.
\end{example}

So far we have illustrated how (or why) hypersurfaces defined by purely
inseparable equations and the notion of blowups along suitable equimultiple
regular centers relate to $S^{q}$-submodules and transforms of
$S^{q}$-submodules. This approach appears already in
\cite[Theorem 5]{abhyankar1956local} and also in
\cite{giraud1983condition}. There is yet a natural question that arises
when viewing (\ref{equations_of_the_transform}): there is a factorization
of the exceptional divisor to the power one, and we may ask if this power
is optimal. For example, let us look at the first expression. Can we have
an equivalent expression (in the sense that they both differ by an element
in $S_{1}^{p}$) in which $\frac{x_{2}}{x_{1}}$ factors to a higher
power? More generally, given a sequence of monoidal transformation, can
we characterize an ``optimal exceptional monomial'' in the previous sense?

These questions are studied within the frame ``jumping phenomenon'', concerning
the behavior of singularities in positive characteristic. The reader is
referred to \cite{moh1996newton} and \cite{HauPer2019} for more on this
concept. The characterization of the optimal exceptional monomial can be
achieved by using logarithmic differential operators with poles along such
exceptional monomial. Unfortunately, our work does not enlighten these
questions. Our results are related rather to the characterization of an
optimal exceptional monomial in which the exponents are multiple of a fixed
power $q$ of $p$; see, for instance, Theorem \ref{invariants_1}. As from
a technical point of view, we do use logarithmic differential operators
to search for invariants that do not get ``worst'' under suitable blowups;
see Theorem \ref{invariants_2}.

\subsection{Content of the Paper}%
\label{sec1.1}

We now turn to the more technical part of Introduction, where we present
the main results. First, we set the notation. Given a ring $B$ of characteristic
$p>0$, we denote by $F:B\to B$ the Frobenius endomorphism. Given
$q=p^{e}$ and an ideal $I\subset B$, we denote
$F^{e}I:=F^{e}(I)=\{x^{q}:x\in I\}$ and
$B^{q}:=F^{e}B=\{x^{q}:x\in B\}$, so that $F^{e} I$ is an ideal of
$B^{q}$. For a domain $B$, we set $B^{1/q}:=\{x\in L:x^{q}\in B\}$, where
$L$ is an algebraic closure of the fraction field of $B$. This notation
extends also to the setting of sheaves of rings and ideals on schemes of
characteristic $p$. Given a scheme $V$, when we use the expressions
$\mathcal{O}_{V}$-ideal, $\mathcal{O}_{V}$-module, and
$\mathcal{O}_{V}$-algebra, we assume that they are quasi-coherent. If
$\mathcal{I}$ is an $\mathcal{O}_{V}$-ideal on $V$, then
$\mathcal{V}(\mathcal{I})\subset V$ denotes the support of
$\mathcal{O}_{V}/\mathcal{I}$.

Let $V$ be a Noetherian irreducible regular scheme of characteristic
$p>0$, and assume that $V$ is $F$-finite, which means that
$\mathcal{O}_{V}$ is a finite $\mathcal{O}_{V}^{p}$-module. This ensures
that $V$ is excellent and that the sheaves of (absolute) differential operators
$\Diff_{V}^{i}$, $i=0,1,2,\ldots$, are locally free of finite rank. We
fix $q=p^{e}$, a power of~$p$. We consider $\mathcal{O}_{V}^{q}$-submodules
of $\mathcal{O}_{V}$. Such a module
$\mathscr{M}\subset\mathcal{O}_{V}$ has attached a finite surjective radicial
morphism $\delta:X\to V$, namely the one defined by the
$\mathcal{O}_{V}$-algebra $(\mathcal{O}_{V}^{q}[\mathscr
{M}])^{1/q}$. The
generic rank $d:=[K(X):K(V)]$ of this morphism is a power of $p$, which
is different from $q$ in general. We denote by $F_{d}(X)\subset X$ the
set of points of multiplicity $d$. Finally, we attach to
$\mathscr{M}$ the collection of $\mathcal{O}_{V}$-ideals
\begin{equation*}
	\Diff_{V,+}^{1}(\mathscr{M})\subseteq\cdots\subseteq
	\Diff_{V,+}^{q-1}( \mathscr{M}),
\end{equation*}
where $\Diff_{V,+}^{i}$ denotes the sheaf of differential operators of
order $i$ that annihilates 1. The next theorem contains results already
mentioned in the first part of the introduction in the particular case
that $X$ is a hypersurface defined by a purely inseparable equation.
\begin{theorem}%
	\label{main_theorem_2}
	In the above setting, the following holds.
	\begin{enumerate}[(3)]
		\item[(1)] There are equality of sets
		\begin{equation*}
			\delta(F_{d}(X))=\{x\in V: \mathscr{M}_{x}\subseteq
			\mathcal {O}_{V,x}^{q}+m_{V,x}^{q}
			\} = \mathcal{V}(\Diff_{V,+}^{q-1}(\mathscr{M})).
		\end{equation*}
		\item[(2)] Let $Y\subset X$ be a closed irreducible subscheme, and let
		$Z\subset V$ denote its image. The following statements are equivalent:
		\begin{enumerate}[(b)]
			\item[(a)]$Y$ is regular and included in $F_{d}(X)$.
			\item[(b)]$Z$ is regular, and
			$\mathscr{M}\subseteq\mathcal{O}_{V}^{q}+\mathcal{I}(Z)^{q}$.
		\end{enumerate}
		\item[(3)] Assume that the equivalent conditions in (2) are satisfied, and consider
		the commutative diagram
		\begin{equation}
			\label{eqdiac} \xymatrix{Y\subset\\ Z\subset} \xymatrix{X\ar[d]_{\delta} &\ar[l] X_{1}\ar[d]^{\delta_{1}}\\
				V&\ar[l] V_{1}}
		\end{equation}
		obtained in Theorem \ref{finite_morphisms_and_closed_subsets}(2). Then
		$\delta_{1}$ is the $V_{1}$-scheme attached to the
		$\mathcal{O}_{V_{1}}^{q}$-module
		$\mathscr{M}_{1}:=(\mathscr{M}\mathcal{O}_{V_{1}}^{q}+\mathcal
		{O}_{V_{1}}^{q}:
		F^{e}\mathcal{L}_{1})\subseteq\mathcal{O}_{V_{1}}$. Here
		$\mathscr{M}\mathcal{O}_{V_{1}}^{q}$ denotes the
		$\mathcal{O}_{V_{1}}^{q}$-submodule of $\mathcal{O}_{V_{1}}$
		generated by
		the sections of $\mathscr{M}$ when these are viewed as sections on the blowup
		$V_{1}$, and $\mathcal{L}_{1}\subset\mathcal{O}_{V_{1}}$ denotes the
		exceptional
		ideal of the blowup $V\leftarrow V_{1}$.
	\end{enumerate}
\end{theorem}
\begin{definition}%
	\label{permissible_center_for_pairs}
	Given an $\mathcal{O}_{V}^{q}$-submodule
	$\mathscr{M}\subset\mathcal{O}_{V}$, we define
	$\operatorname{Sing}(\mathscr{M},1):=\mathcal{V}(\Diff
	_{V,+}^{q-1}(\mathscr{M}))$,
	which is a closed subset of $V$. A~regular closed subscheme
	$Z\subset V$ included in $\operatorname{Sing}(\mathscr{M},1)$ is
	called a permissible
	center for $(\mathscr{M},1)$. If $V\leftarrow V_{1}$ is the blowup of
	$V$ along $Z$ and $\mathcal{L}_{1}$ denotes the exceptional ideal, then
	we call the $\mathcal{O}_{V_{1}}^{q}$-module
	$(\mathscr{M}\mathcal{O}_{V_{1}}^{q}+\mathcal{O}_{V_{1}}^{q}:F^{e}
	\mathcal{L}_{1})$ the 1-transform of $\mathscr{M}$ by the blowup.\looseness=1
\end{definition}
\begin{remark}%
	\label{nmls}
	Roughly speaking, the theorem enables us to replace diagrams like (\ref
	{eqdiac})
	by diagrams of the form
	\begin{equation}
		\label{eq3313bfrer} \xymatrix@R=0pt @C=30pt { V & \ar[l] V_{1} &
			\\
			\mathscr{M} & \mathscr{M}_{1} & }
	\end{equation}
	where $V\leftarrow V_{1}$ is a blowup along a permissible center for
	$(\mathscr{M},1)$, and $\mathscr{M}_{1}$ is the 1-transform of
	$\mathscr{M}$. Similarly, an iteration of diagrams like (\ref
	{eqdiac}) corresponds
	to a sequence
	\begin{equation}
		\label{eq3313b} \xymatrix@R=0pt @C=30pt { V & \ar[l]_-{\pi_{1}} V_{1} & \ar[l]_-{
				\pi_{2}} \ldots& \ar[l]_-{\pi_{r}} V_{r}%
			\\
			\mathscr{M} & \mathscr{M}_{1} & \ldots& \mathscr{M}_{r} }
	\end{equation}
	where $V_{i}\leftarrow V_{i+1}$ is the blowup along a permissible center
	for $(\mathscr{M}_{i},1)$, and $\mathscr{M}_{i+1}$ is the 1-transform of
	$\mathscr{M}_{i}$.
\end{remark}

We saw already the role played by the $\mathcal{O}_{V}$-ideal
$\Diff_{V,+}^{q-1}(\mathscr{M})$. We now consider the full collection
\begin{equation}
	\label{collection} \mathcal{G}(\mathscr{M}):=(\Diff_{V,+}^{1}(
	\mathscr{M}),\ldots ,\Diff_{V,+}^{q-1}( \mathscr{M})).
\end{equation}
This is an example of the following more general concept.
\begin{definition}%
	\label{q-diff}
	A \emph{$q$-differential collection of ideals} is a sequence of
	$\mathcal{O}_{V}$-ideals
	$\mathcal{G}=(\mathcal{I}_{1},\ldots,\mathcal{I}_{q-1})$ such that
	$\Diff_{V}^{i}(\mathcal{I}_{j})\subseteq\mathcal{I}_{i+j}$ whenever
	$i+j<q$. Given such a sequence and a point $x\in V$, we set
	\begin{equation*}
		\eta_{x}(\mathcal{G}):=\min\{\nu_{x}(
		\mathcal{I}_{i})+i:1\leq i\leq q-1 \},
	\end{equation*}
	where $\nu_{x}(\mathcal{I}_{i})$ denotes, as usual, the order of
	$\mathcal{I}_{i}$ at $x$.
\end{definition}
In the case where $\mathcal{G}=\mathcal{G}(\mathscr{M})$ for an
$\mathcal{O}_{V}^{q}$-submodule $\mathscr{M}\subseteq\mathcal
{O}_{V}$, we
simply write $\eta_{x}(\mathscr{M})$ instead of
$\eta_{x}(\mathcal{G}(\mathscr{M}))$. This number was already introduced
in the first part of Introduction in the case that $\mathscr{M}$ is principal
and $V$ is an affine space. The same argument given there shows that
$\eta_{x}(\mathcal{G})$ defines an upper-semicontinuous function on
$V$ with values on $\mathbb{N}$. It rests on the fact that for an
$\mathcal{O}_{V}$-ideal $\mathcal{J}$, the function
$x\mapsto\nu_{x}(\mathcal{J})$, is upper-semicontinuous (see, e.g.,
\cite[Chapter~2]{dietel2014refinement} or
\cite[Chapter~3]{giraud1972etude} for an affine space $V$). We include in
these notes a self-contained proof within our setting (see Proposition \ref{Jacobian_criterion:_global_version}). The following theorem establishes
that, under an appropriate definition of transformation of $q$-differential
collections, the function $x\mapsto\eta_{x}(\mathcal{G})$ satisfies the
so-called \emph{fundamental pointwise inequality}.
\begin{theorem}%
	\label{point-wise_inequality_for_q-diff}
	Let $\mathcal{G}$ be a $q$-differential collection on $V$, let
	$Z\subset V$ be a regular center included in the maximum locus of
	$\eta(\mathcal{G})$, say $\eta_{x}(\mathcal{G})=aq+b$ for all
	$x\in Z$ (with $a,b\in\mathbb{N}_{0}$ and $0\leq b<q$), let
	$V\xleftarrow{\pi} V_{1}$ be the blowup along $Z$, and let
	$\mathcal{L}$ denote the exceptional ideal. Then the collection of
	$\mathcal{O}_{V_{1}}$-ideals
	\begin{equation}
		\label{G_1} \mathcal{G}_{1}^{(a)}:=((
		\mathcal{I}_{1}\mathcal{O}_{V_{1}}:\mathcal
		{L}^{qa}), \ldots, (\mathcal{I}_{q-1}\mathcal{O}_{V_{1}}:
		\mathcal{L}^{qa}))
	\end{equation}
	is also $q$-differential, and the following pointwise inequality holds:
	\begin{equation}
		\eta_{\pi(x_{1})}(\mathcal{G})\geq\eta_{x_{1}}(
		\mathcal{G}_{1}^{(a)}), \quad \forall x_{1}\in
		V_{1}.
	\end{equation}
\end{theorem}
The collection of $\mathcal{O}_{V_{1}}$-ideals $\mathcal
{G}_{1}^{(a)}$ will
be called the $a$-transform of $\mathcal{G}$ by the blowup. The exponent
$qa$ of $\mathcal{L}$ appearing in its definition is somehow related with
a natural factorization in the definition of transformation of modules,
as we will further see. To get a flavor of this theorem, it is worth bearing
in mind the following classical example. Let $\mathcal{J}$ be a nonzero
$\mathcal{O}_{V}$-ideal, and let $Z\subset V$ be a closed regular subscheme
included in the maximum locus of the function
$x\mapsto\nu_{x}(\mathcal{J})$, say $\nu_{x}(\mathcal{J})=b$, for all
$x\in Z$. Let $V\xleftarrow{\pi} V_{1}$ be the blowup of $V$ along
$Z$, and let $H_{1}\subset V_{1}$ denote the exceptional hypersurface.
Then there is a factorization
$\mathcal{J}\mathcal{O}_{V_{1}}=\mathcal{I}(H_{1})^{b}\mathcal
{J}_{1}$ for
some $\mathcal{O}_{V_{1}}$-ideal $\mathcal{J}_{1}$ that does not
vanish along
$H_{1}$. Then we have the following pointwise inequality.
\begin{theorem}%
	\label{v_1(I)_is_at_most_v(I)}
	$ \nu_{\pi(x_{1})}(\mathcal{J}) \geq\nu_{x_{1}}(\mathcal{J}_{1})$ for
	all $x_{1}\in V_{1}$.
\end{theorem}
We include a self-contained proof of this theorem in Proposition \ref{the_order_function_of_the_transform_of_an_ideal}.

We return to the consideration of $\mathcal{O}_{V}^{q}$-submodules
$\mathscr{M}\subset\mathcal{O}_{V}$. Part (1) of the following
theorem was
already mentioned in Proposition \ref{prop-intro-1} in a particular case.
Part (2) is somehow related with the outcome of Proposition \ref{prop2-intro}.
\begin{theorem}%
	\label{invariants_1}
	For an $\mathcal{O}_{V}^{q}$-module
	$\mathscr{M}\subseteq\mathcal{O}_{V}$, the following holds.
	\begin{enumerate}[(3)]
		\item[(1)] $\operatorname{Sing}(\mathscr{M},1)=\{x\in V:\eta
		_{x}(\mathscr{M})\geq q\}$.
		\item[(2)] Assume that $\operatorname{Sing}(\mathscr{M},1)$ is
		nonempty, and write
		the maximum possible value of $\eta_{x}(\mathscr{M})$ in the form
		$aq+b$ with $0\leq b<q$. Then a regular center $Z\subset V$ included in
		the maximum locus of $x\mapsto\eta_{x}(\mathscr{M})$ defines a sequence
		\begin{equation}
			\xymatrix@R=0pt @C=30pt { V & \ar[l]_-{\pi} V_{1} & \ar[l]_-{\cong}
				\ldots& \ar[l]_-{\cong} V_{1}%
				\\
				\mathscr{M} & \mathscr{M}_{1}^{(1)} & \ldots& \mathscr{M}_{1}^{(a)} }
		\end{equation}
		where
		\begin{enumerate}[(c)]
			\item[(a)]$V\leftarrow V_{1}$ is the blowup along $Z$, and
			$\mathscr{M}_{1}^{(1)}$ is the 1-transform of $\mathscr{M}$ by $\pi$.
			\item[(b)] For each index $1\leq i<a$, the exceptional divisor of $\pi$, say
			$H_{1}\subset V_{1}$, is included in
			$\operatorname{Sing}(\mathscr{M}_{1}^{(i)},1)$, the isomorphism
			$V_{i}\xleftarrow{\cong} V_{i+1}$ is the blowup of $V_{i}$ along
			$H_{1}$, and $\mathscr{M}_{1}^{(i+1)}$ is the $1$-transform of $
			\mathscr{M}_{1}^{(i)}$.
			\item[(c)]$H_{1}\not\subset\operatorname{Sing}(V_{1},\mathscr{M}_{1}^{(a)})$.
		\end{enumerate}
		\item[(3)] If $(\mathcal{G}(\mathscr{M}))_{1}^{(a)}$ denotes the $a$-transform
		of $\mathcal{G}(\mathscr{M})$, then there is a componentwise inclusion
		\begin{equation}
			\label{incl} \mathcal{G}(\mathscr{M}_{1}^{(a)})
			\subseteq(\mathcal{G}(\mathscr {M}))_{1}^{(a)}.
		\end{equation}
		In particular,
		$\eta_{x_{1}}(\mathscr{M}_{1}^{(a)})\geq\eta_{x_{1}}((\mathcal{G}(
		\mathscr{M}))_{1}^{(a)})$ for all $x_{1}\in V_{1}$, whence
		\begin{equation}
			\label{inc} \{x_{1}\in V_{1}:
			\eta_{x_{1}}((\mathcal{G}(\mathscr{M}))_{1}^{(a)})
			\geq q\}\subseteq\operatorname{Sing}(\mathscr{M}_{1}^{(a)},1).
		\end{equation}
	\end{enumerate}
\end{theorem}
We call $\mathscr{M}_{1}^{(a)}$ \emph{the $a$-transform of
	$\mathscr{M}$} by the blowup.

\begin{remark}%
	\label{eta_does_not_satisfy_point-wise_inequality}
	\emph{The function $x\mapsto\eta_{x}(\mathscr{M})$ does not satisfy
		a pointwise
		inequality in general:} The inclusion (\ref{incl}) is in general strict,
	and therefore we cannot make use of Theorem \ref{point-wise_inequality_for_q-diff} to deduce a pointwise inequality
	$\eta_{\pi(x_{1})}(\mathscr{M})\geq\eta_{x_{1}}(\mathscr{M}_{1}^{(a)})$
	for $x_{1}\in V_{1}$, which would be desirable. For example, take
	$p=q=3$, $V=\Spec(k[x_{1},x_{2},x_{3},x_{4},x_{5}])$, and
	$\mathscr{M}=\mathcal{O}_{V}^{q}\cdot x_{1}x_{2}x_{3}x_{4}x_{5}$. The maximum
	of $\eta_{x}(\mathscr{M})$ is $5=3\cdot1+2$, whence $a=1$, and this maximum
	is attained only at the origin. However, after blowing up the origin, the
	restriction to the $x_{1}$-chart of the 1-transform
	$\mathscr{M}_{1}^{(1)}$ of $\mathscr{M}$ is generated by
	$x_{1}^{2}x_{2}'x_{3}'x_{4}'x_{5}'$, where
	$x_{i}'=\frac{x_{i}}{x_{1}}$. Thus the maximum of
	$\eta_{x}(\mathscr{M}_{1}^{(1)})$ is $\geq6=3\cdot2$. This example also
	shows that even $\lfloor\frac{\eta_{x}(\mathscr{M})}{q}\rfloor$ does
	not satisfy the pointwise inequality. Nevertheless, (\ref{inc}) implies
	that a sequence of permissible transformations of $q$-differential collections
	starting from $\mathcal{G}(\mathscr{M})$, in a way that we will specify,
	induces a sequence of permissible transformations of $\mathscr{M}$. This
	will be discussed in Remark \ref{rk429}.
\end{remark}

We end this presentation by illustrating how $q$-differential collections
define invariants of singularities satisfying the pointwise inequality
in a different manner. Given a sequence of monoidal transformations, it
is natural to define invariants that take into account the hypersurfaces
introduced in the previous blowups. Following this line, we consider 3-tuples
$(\mathscr{M},\Lambda,\mathcal{L})$, where
$\mathscr{M}\subset\mathcal{O}_{V}$ is an $\mathcal
{O}_{V}^{q}$-module on
an $F$-finite regular scheme $V$, $\Lambda$ is a finite collection of
hypersurfaces with only normal crossings, and $\mathcal{L}$ is an invertible
$\mathcal{O}_{V}$-ideal included in the ideal $\mathcal{I}(H)$ for each
$H\in\Lambda$. For example, we could take for $\Lambda$ the empty collection
and $\mathcal{L}=\mathcal{O}_{V}$. We attach to such 3-tuple the collection
of $\mathcal{O}_{V}$-ideals
\begin{align*}
	\mathcal{G}(\mathscr{M},\Lambda,\mathcal{L}):={}&((\Diff_{V,\Lambda,+}^{1}(
	\mathscr{M}):\mathcal{L}^{1}),(\Diff_{V,\Lambda,+}^{2}(
	\mathscr{M}): \mathcal{L}^{2}),\ldots,
	\\
	&{}(\Diff_{V,\Lambda,+}^{q-1}(
	\mathscr{M}): \mathcal{L}^{q-1})),
\end{align*}
where $\Diff_{V,\Lambda,+}^{i}\subseteq\Diff_{V,+}^{i}$ denotes the
subsheaf consisting of those differential operators that are logarithmic
with respect to $\mathcal{I}(H)$ for all $H\in\Lambda$; see
Section \ref{Section_on_Diff}.

\begin{theorem}%
	\label{invariants_2}
	Within the previous setting, the following properties hold:
	\begin{enumerate}[(2)]
		\item[(1)] $\mathcal{G}(\mathscr{M},\Lambda,\mathcal{L})$ is a
		$q$-differential
		collection, and there is a componentwise inclusion
		$\mathcal{G}(\mathscr{M})\subseteq\mathcal{G}(\mathscr{M},\Lambda,
		\mathcal{L})$. In particular,
		$\eta_{x}(\mathscr{M})\geq\eta_{x}(\mathcal{G}(\mathscr
		{M},\Lambda,
		\mathcal{L}))$ for $ x\in V$, whence
		\begin{equation*}
			\{x\in V: \eta_{x}(\mathcal{G}(\mathscr{M},\Lambda,\mathcal {L}))
			\geq q \}\subset\operatorname{Sing}(\mathscr{M},1).
		\end{equation*}
		\item[(2)] Let $Z\subset V$ be a regular center included in the
		maximum locus
		of the function
		$x\mapsto\eta_{x}(\mathcal{G}(\mathscr{M},\Lambda,\mathcal{L}))$, say
		$\eta_{x}(\mathcal{G}(\mathscr{M},\Lambda,\mathcal{L}))=aq+b$ for all
		$x\in Z$ ($a,b\in\mathbb{N}_{0}$, $0\leq b<q$), and suppose that
		$\Lambda$ has only normal crossings with $Z$. Let
		$V\xleftarrow{\pi}V_{1}$ be the blowup along $Z$, and let
		$H_{1}\subset V_{1}$ be the exceptional divisor. We set:
		\begin{enumerate}[(c)]
			\item[(a)] $\mathscr{M}_{1}^{(a)}:=$ the $a$-transform of $\mathscr
			{M}$ by the
			blowup.
			\item[(b)] $\Lambda_{1}:=$ the collection of the strict transforms of each
			$H\in\Lambda$, plus the exceptional divisor $H_{1}$.
			\item[(c)] $\mathcal{L}_{1}:=(\mathcal{L}\mathcal
			{O}_{V_{1}})\mathcal{I}(H_{1})$.
		\end{enumerate}
		Then $\Lambda_{1}$ is a collection of hypersurfaces with only normal crossings,
		$\mathcal{L}_{1}$ is included in $\mathcal{I}(H)$ for each
		$H\in\Lambda_{1}$, so that
		$\mathcal{G}(\mathscr{M}_{1},\Lambda_{1},\mathcal{L}_{1})$ is a
		$q$-differential
		collection on $V_{1}$, and the following pointwise inequality holds:
		\begin{align*}
			\eta_{\pi(x_{1})}(\mathcal{G}(\mathscr{M},\Lambda,\mathcal {L}))\geq
			\eta_{x_{1}}(\mathcal{G}(\mathscr{M}_{1},
			\Lambda_{1},\mathcal{L}_{1})),\quad  x_{1}\in
			V_{1}.
		\end{align*}
	\end{enumerate}
\end{theorem}
We refer to Remark \ref{setting_for_logarithmic} and Example \ref{final_example1} for a discussion on how this theorem can be used to
define sequences of permissible transformations over
$(V,\mathscr{M})$.

\subsection{On the Organization of the Paper}
\label{sec1.2}

In the first part of Section \ref{Section_on_modules}, we introduce definitions
and some basic constructions related to $\mathcal{O}_{V}^{q}$-modules and
their transformations under blowups. In the second part, we discuss the
interplay between $\mathcal{O}_{V}^{q}$-modules and their associated
$V$-schemes. In particular, we will prove Theorem \ref{main_theorem_2} except for the second equality in (1).

In Section \ref{Section_on_Diff}, we revise some properties about
$p$-basis and differential operators of $F$-finite regular local rings
and apply them to the introduction of local invariants associated with
$\mathcal{O}_{V}^{q}$-modules, such as the function
$x\to\eta_{x}(\mathscr{M})$, which is related to the notion of Hironaka's
slope discussed in \cite{benitovillamayor}. The application of this function
to the study of $\mathcal{O}_{V}^{q}$-modules and transformations of
$\mathcal{O}_{V}^{q}$-modules will be treated in Section \ref{Section_on_Diff2}. At this point, we can prove the second equality
in Theorem \ref{main_theorem_2}(1), and also part (1) and (2) of Theorem \ref{invariants_1}.

Finally, in Section \ref{Section_on_q-Diff}, we study $q$-differential
collections (Definition \ref{q-diff}). In the first part, we prove Theorem \ref{point-wise_inequality_for_q-diff} and part (3) of Theorem \ref{invariants_1}. In the final part, we delve into the study of invariants
of singularities by using logarithmic differential operators. The main
objective of the section is the proof of Theorem \ref{invariants_2}.

The presentation of the paper is not limited to the proof of the results
presented in this introduction. Instead, we choose a detailed exposition
that, we hope, will serve as a reference for future works. We also tried
to keep the paper as self-contained as possible. In particular, we include
proofs of the results stated in Theorem \ref{finite_morphisms_and_closed_subsets} adapted to our setting.

\subsection{Historical Remarks}
\label{sec1.3}

Our discussion makes systematic use of differential operators and logarithmic
differential operators. So it relates strongly with some published works.
Let us first mention \cite{giraud1983condition}, where Giraud studies Jung
conditions for finite radicial coverings of regular varieties in positive
characteristic. The outcome of \cite{cossart1987thesis} is a first breakthrough
in the resolution of singularities of radicial extensions in positive
characteristic.
Cossart and Piltant \cite{cossart2008resolution} proved resolution of
singularities
for arbitrary three-dimensional schemes of positive characteristic, and
more recently for arithmetical threefolds \cite{cossartpiltant2019}. These
are proofs that introduce suitable invariants satisfying the pointwise
inequality. Other invariants in positive characteristic arise in the work
of Kawanoue and Matsuki in \cite{kawanoue2007toward} and
\cite{kawanouematsuki}. Also in this line, and related with our exposition,
is the joint work of the second author with Benito
\cite{benitovillamayor} (see also \cite{benito2013monoidal}). All these
cited papers also make use either of differential operators or of logarithmic
differential operators.

\section{$\mathcal{O}_{V}^{q}$-Modules and Transformations by Blowups}
\label{Section_on_modules}

Throughout this section, $V$ denotes a fixed irreducible $F$-finite regular
scheme of characteristic $p>0$, and $q=p^{e}$ denotes a fixed power of
$p$. In Section \ref{Basic_constructions}, we present an exposition on
$\mathcal{O}_{V}^{q}$-submodules of $\mathcal{O}_{V}$ with a view toward
their associated finite radicial morphisms. We will introduce two notions
of equivalence of $\mathcal{O}_{V}^{q}$-modules and discuss different notions
of transformations of modules by blowups. In Section \ref{Relation_with_finite_radicial_morphisms}, we discuss the interplay
between $\mathcal{O}_{V}^q$-submodules of $\mathcal{O}_{V}$ and finite radicial
morphisms $\delta:X\to V$, in connection with blowups. In particular,
we prove Theorem \ref{main_theorem_2} except for the second equality in
(1).

\subsection{Basic Constructions}
\label{Basic_constructions}

An $\mathcal{O}_{V}^{q}$-submodule $\mathscr{M}\subset\mathcal
{O}_{V}$ has
attached a finite radicial morphism $\delta:X\to V$. This is the
$V$-scheme defined by the $\mathcal{O}_{V}$-algebra
$(\mathcal{O}_{V}^{q}[\mathscr{M}])^{1/q}\subset\mathcal
{O}_{V}^{1/q}$, where
$\mathcal{O}_{V}^{q}[\mathscr{M}]\subset\mathcal{O}_{V}$ denotes the
$\mathcal{O}_{V}^{q}$-subalgebra of $\mathcal{O}_{V}$ spanned by
$\mathscr{M}$. Two $\mathcal{O}_{V}^{q}$-submodules that span the same
$\mathcal{O}_{V}^{q}$-subalgebra of $\mathcal{O}_{V}$ have attached
the same
finite radicial morphism, so they should be considered as equivalent. We
call this relation of submodules weak equivalence. We also introduce a
stronger relation, which is more suitable for most of the constructions
that we will introduce.
\begin{definition}%
	\label{defequi}
	Let $\mathscr{M}_{1}$ and $\mathscr{M}_{2}$ be two
	$\mathcal{O}_{V}^{q}$-submodules of $\mathcal{O}_{V}$.
	\begin{enumerate}[(3)]
		\item[(1)] $\mathscr{M}_{1}$ and $\mathscr{M}_{2}$ are said to be
		equivalent if
		$\mathscr{M}_{1}+\mathcal{O}_{V}^{q}=\mathscr{M}_{2}+\mathcal{O}_{V}^{q}$.
		In notation, $\mathscr{M}_{1}\sim\mathscr{M}_{2}$.
		\item[(2)] $\mathscr{M}_{1}$ and $\mathscr{M}_{2}$ are said to be
		weakly equivalent
		if
		$\mathcal{O}_{V}^{q}[\mathscr{M}_{1}]=\mathcal{O}_{V}^{q}[\mathscr{M}_{2}]$.
		\item[(3)] $\mathscr{M}_{1}$ is said to be trivial if $\mathscr
		{M}\sim0$.
	\end{enumerate}
\end{definition}
Here the sum of two submodules of $\mathcal{O}_{V}$ is the smallest submodule
containing both. The relations introduced in (1) and (2) are clearly equivalence
relations. Note that $\mathcal{O}_{V}^{q}+\mathscr{M}$ is the largest
$\mathcal{O}_{V}^{q}$-submodule in its equivalence class; in particular,
each equivalence class has a canonical representative. Since
$\mathcal{O}_{V}^{q}[\mathscr{M}]=\mathcal{O}_{V}^{q}[\mathscr
{M}+\mathcal{O}_{V}^{q}]$,
equivalent modules are weakly equivalent. The converse is not true: take
$V=\Spec(\mathbb{F}_{p}[t])$ and $\mathscr{M}$ the
$\mathcal{O}_{V}^{q}$-module spanned by $t$. Note that
$t^{2}\in\Gamma(V,\mathcal{O}_{V}^{q}[\mathscr{M}])$; however, it
is not
a global section of $\mathscr{M}$. Finally, $\mathscr{M}\sim0$ if
and only
if $\mathscr{M}\subseteq\mathcal{O}_{V}^{q}$.

\begin{parrafo}
	One important operation on $\mathcal{O}_{V}^{q}$-submodules of
	$\mathcal{O}_{V}$ is the extension defined by taking conductors with respect
	to a hypersurface in a sense we specify now. Let $\mathcal{L}$ be an invertible
	ideal of $\mathcal{O}_{V}$. Then $F^{e}\mathcal{L}$ is an invertible ideal
	of $\mathcal{O}_{V}^{q}$, and we have
	$\mathcal{L}^{q}=F^{e}\mathcal{L}\cdot\mathcal{O}_{V}$. We deduce
	from this
	equality that all the $\mathcal{O}_{V}^{q}$-submodules of
	$\mathcal{L}^{q}=F^{e}\mathcal{L}\cdot\mathcal{O}_{V}$ are of the form
	$(F^{e}\mathcal{L})\mathscr{N}$ for some $\mathcal{O}_{V}^{q}$-module
	$\mathscr{N}\subseteq\mathcal{O}_{V}$.
	
	Given an $\mathcal{O}_{V}^{q}$-module
	$\mathscr{M}\subseteq\mathcal{O}_{V}$, the conductor
	$(\mathscr{M}:F^{e}\mathcal{L})\subseteq\mathcal{O}_{V}$ (defined
	in the
	class of $\mathcal{O}_{V}^{q}$-submodules of $\mathcal{O}_{V}$) is,
	by definition,
	the largest $\mathcal{O}_{V}^{q}$-submodule
	$\mathscr{N}\subseteq\mathcal{O}_{V}$ such that
	$(F^{e}\mathcal{L})\mathscr{N}\subseteq\mathscr{M}$. Since
	$(F^{e}\mathcal{L})\mathscr{N}\subseteq\mathcal{L}^{q}$, it is
	clear that
	$(\mathscr{M}:F^{e}\mathcal{L})=(\mathscr{M}\cap\mathcal{L}^{q}:F^{e}
	\mathcal{L})$. As $\mathscr{M}\cap\mathcal{L}^{q}$ is a submodule of
	$\mathcal{L}^{q}$, the discussion in the previous paragraph shows that
	$\mathscr{M}\cap\mathcal{L}^{q}=(F^{e}\mathcal{L})\mathscr{N}$ for an
	$\mathcal{O}_{V}^{q}$-submodule $\mathscr{N}$, and this submodule has to
	be the conductor $(\mathscr{M}:F^{e}\mathcal{L})$. Thus
	\begin{align}
		\label{M_int_L^q} \mathscr{M}\cap\mathcal{L}^{q}=(F^{e}
		\mathcal{L}) (\mathscr{M}:F^{e} \mathcal{L}).
	\end{align}
	
	We remark that even though the assignment
	$\mathscr{M}\mapsto(\mathscr{M}:F^{e}\mathcal{L})$ is well-defined
	in the
	class of $\mathcal{O}_{V}^{q}$-submodules of $\mathcal{O}_{V}$, it is not
	compatible with any of our notions of equivalence. For instance, take
	$V=\Spec(\mathbb{F}_{p}[s,t])$,
	$\mathscr{M}=\mathcal{O}_{V}^{q}\cdot st^{q}$,
	$\mathscr{M}'=\mathcal{O}_{V}^{q}\cdot(st^{q}+1)$ and
	$\mathcal{L}=\mathcal{O}_{V}\cdot t$. Note that
	$\mathscr{M}\sim\mathscr{M}'$, yet
	$(\mathscr{M}:F^{e}\mathcal{L})=\mathcal{O}_{V}^{q}\cdot s$ and
	$(\mathscr{M}:F^{e}\mathcal{L})=\mathcal{O}_{V}^{q}\cdot(st^{q}+1)$
	are neither
	equivalent nor weakly equivalent. This presents a difficulty as we will
	consider $\mathcal{O}_{V}^{q}$-modules only up to equivalence (and up to
	weak equivalence). One way to overcome this difficulty is by making use
	of the canonical representative in each equivalence class.
\end{parrafo}

\begin{definition}%
	\label{M_L}
	Let $\mathscr{M}\subseteq\mathcal{O}_{V}$ be an $\mathcal
	{O}_{V}^{q}$-submodule,
	and let $\mathcal{L}\subseteq\mathcal{O}_{V}$ be an invertible
	ideal. The
	$\mathcal{L}$-conductor of $\mathscr{M}$ is the $\mathcal{O}_{V}^{q}$-module
	\begin{align*}
		\mathscr{M}_{\mathcal{L}}:=((\mathcal{O}_{V}^{q}+
		\mathscr{M}):F^{e} \mathcal{L}) \subseteq\mathcal{O}_{V}.
	\end{align*}
\end{definition}
It follows from the definition that
$\mathscr{N}_{\mathcal{L}}=\mathscr{M}_{\mathcal{L}}$ if
$\mathscr{N}\sim\mathscr{M}$, that is, the definition of $\mathcal
{L}$-conductor
is compatible with the notion of equivalence. We gather more properties
in the next lemma.
\begin{lemma}%
	\label{L^qM_L}
	Let $\mathscr{M}\subseteq\mathcal{O}_{V}$ be an $\mathcal
	{O}_{V}^{q}$-module,
	and let $\mathcal{L}\subset\mathcal{O}_{V}$ be an invertible ideal. Then
	the following statements hold:
	\begin{enumerate}[(3)]
		\item[(1)] $ (\mathcal{O}_{V}^{q}+\mathscr{M})\cap\mathcal
		{L}^{q}=(F^{e}\mathcal{L})
		\mathscr{M}_{\mathcal{L}}$.
		\item[(2)] $\mathscr{M}_\mathcal{L}=\mathscr{M}_\mathcal
		{L}+\mathcal{O}_{V}^{q}$.
		\item[(3)] If $\mathcal{L}'\subset\mathcal{O}_{V}$ is a second
		invertible ideal,
		then
		$\mathscr{M}_{\mathcal{L}\mathcal{L}'}=(\mathscr{M}_{\mathcal{L}})_{
			\mathcal{L}'}$.
	\end{enumerate}
\end{lemma}
\begin{proof}
	(1) follows from (\ref{M_int_L^q}) when applied to the submodule
	$\mathcal{O}_{V}^{q}+\mathscr{M}$. As for (2), we have
	$(F^{e}\mathcal{L})\mathcal{O}_{V}^{q}\subseteq\mathcal{O}_{V}^{q}
	\subseteq\mathcal{O}_{V}^{q}+\mathscr{M}$, so that
	$\mathcal{O}_{V}^{q}\subseteq\mathscr{M}_{\mathcal{L}}$, and hence
	$\mathscr{M}_{\mathcal{L}}=\mathcal{O}_{V}^{q}+\mathscr
	{M}_{\mathcal{L}}$. Note
	finally that
	\begin{eqnarray*}
		\mathscr{M}_{\mathcal{L}\mathcal{L}'}&=&(\mathcal{O}_{V}^{q}+
		\mathscr {M}:(F^{e} \mathcal{L}) (F^{e}
		\mathcal{L'}))
		\\
		&=&((\mathcal{O}_{V}^{q}+
		\mathscr{M}:F^{e} \mathcal{L}):F^{e}\mathcal{L'})
		\\
		&=&(
		\mathscr{M}_{\mathcal
			{L}}:F^{e}\mathcal{L'})=(
		\mathcal{O}_{V}^{q}+\mathscr{M}_{\mathcal{L}}:F^{e}
		\mathcal {L'})=(\mathscr{M}_{
			\mathcal{L}})_{\mathcal{L}'},
	\end{eqnarray*}
	which proves (3).
\end{proof}

Our definition of $\mathcal{L}$-conductor is not compatible with the notion
of weak equivalence in general. For instance, in the following example,
we have a strict inclusion
$\mathcal{O}_{V}^{q}[\mathscr{M}_\mathcal{L}]\subset(\mathcal{O}_{V}^{q}[
\mathscr{M}])_{\mathcal{L}}$, which implies that the weakly
equivalent modules
$\mathscr{M}$ and $\mathcal{O}_{V}^{q}[\mathscr{M}]$ have $\mathcal
{L}$-conductors
that are not weakly equivalent. Take
$V=\operatorname{Spec}(\mathbb{F}_{p}[t])$,
$\mathscr{M}=\mathcal{O}_{V}^{q}\cdot t^{q-1}$ and
$\mathcal{L}=\mathcal{O}_{V}\cdot t$. Note that
$\mathscr{M}_{\mathcal{L}}=\mathscr{M}$ and
$\Gamma(V,\mathcal{O}_{V}^{q}[\mathscr{M}])=\mathbb{F}_{p}[t^{q-1},t^{q}]$.
On the other hand,
$\Gamma(V,(\mathcal{O}_{V}^{q}[\mathscr{M}])_\mathcal{L})$ contains
$t^{q-2}$ since $t^{q}t^{q-2}=(t^{q-1})^{2}$.

In Lemma \ref{Well_definition_of_the_weak_equivalence_class_of_M_L}, we
introduce conditions under which the compatibility with the weak equivalence
holds. We first study, in Lemma \ref{When_M_is_generated_by_elements_in_I}, the condition
$\mathscr{M}\subseteq\mathcal{O}_{V}^{q}+\mathcal{I}$ for a general
$\mathcal{O}_{V}$-ideal $\mathcal{I}$. Note that
$\mathcal{O}_{V}^{q}+\mathcal{I}$ is an $\mathcal{O}_{V}^{q}$-subalgebra,
namely,
$\mathcal{O}_{V}^{q}+\mathcal{I}=\mathcal{O}_{V}^{q}[\mathcal{I}]$.
\begin{lemma}%
	\label{When_M_is_generated_by_elements_in_I}
	Given an $\mathcal{O}_{V}^{q}$-module
	$\mathscr{M}\subseteq\mathcal{O}_{V}$ and an $\mathcal{O}_{V}$-ideal
	$\mathcal{I}$, the following conditions are equivalent.
	\begin{enumerate}[(3)]
		\item[(1)] $\mathscr{M}\subseteq\mathcal{O}_{V}^{q}+\mathcal{I}$.
		\item[(2)] $\mathcal{O}_{V}^{q}[\mathscr{M}]\subseteq\mathcal
		{O}_{V}^{q}+\mathcal{I}$.
		\item[(3)] $\mathscr{M}\sim((\mathcal{O}_{V}^{q}+\mathscr{M})\cap
		\mathcal{I})$.
		\item[(4)] $\mathscr{M} \sim\mathscr{M'}$ for some $\mathcal
		{O}_{V}^{q}$-module
		$\mathscr{M'}$ included in $\mathcal{I}$.
	\end{enumerate}
\end{lemma}
\begin{proof}
	The equivalence (1) $\Leftrightarrow$ (2) is trivial. For the implication
	(1) $\Rightarrow$ (3), observe that (1) implies that
	$\mathcal{O}_{V}^{q}+\mathscr{M}\subseteq\mathcal
	{O}_{V}^{q}+\mathcal{I}$,
	and hence
	\begin{eqnarray*}
		\mathscr{M}\sim\mathcal{O}_{V}^{q}+\mathscr{M}&=&(
		\mathcal{O}_{V}^{q}+ \mathscr{M})\cap(
		\mathcal{O}_{V}^{q}+\mathcal{I})
		\\
		&=&\mathcal{O}_{V}^{q}+((
		\mathcal{O}_{V}^{q}+\mathscr{M})\cap\mathcal{I})\sim((
		\mathcal{O}_{V}^{q}+ \mathscr{M})\cap\mathcal{I}).
	\end{eqnarray*}
	Finally, the implications (3) $\Rightarrow$ (4) and (4) $\Rightarrow$ (1)
	are clear.
\end{proof}

\begin{lemma}%
	\label{Well_definition_of_the_weak_equivalence_class_of_M_L}
	Let $\mathscr{M}\subseteq\mathcal{O}_{V}$ be an $\mathcal
	{O}_{V}^{q}$-submodule,
	and let $\mathcal{L}\subset\mathcal{O}_{V}$ be an invertible ideal. Assume
	that $\mathscr{M}\subseteq\mathcal{O}_{V}^{q}+\mathcal{L}^{q}$. Then
	\begin{enumerate}[(2)]
		\item[(1)] $\mathscr{M}\sim(F^{e}\mathcal{L})\mathscr{M}_{\mathcal{L}}$.
		\item[(2)] Let $\tilde{\mathscr{M}}\subseteq\mathcal{O}_{V}$ be another
		$\mathcal{O}_{V}^{q}$-module, and assume that $\mathscr{M}$ and
		$\tilde{\mathscr{M}}$ are weakly equivalent. Then
		$\tilde{\mathscr{M}}\subseteq\mathcal{O}_{V}^{q}+\mathcal{L}^{q}$,
		and the
		$\mathcal{O}_{V}^{q}$-modules $\mathscr{M}_{\mathcal{L}}$ and
		$\tilde{\mathscr{M}}_{\mathcal{L}}$ are weakly equivalent.
	\end{enumerate}
\end{lemma}
\begin{proof}
	By Lemma \ref{When_M_is_generated_by_elements_in_I} the condition
	$\mathscr{M}\subseteq\mathcal{O}_{V}^{q}+\mathcal{L}^{q}$ implies that
	$\mathscr{M}\sim((\mathcal{O}_{V}^{q}+\mathscr{M})\cap\mathcal{L}^{q})$,
	and the latter is $ (F^{e}\mathcal{L})\mathscr{M}_{\mathcal{L}}$ according
	to Lemma \ref{L^qM_L}(1). This proves (1).
	
	Before proving (2), as a preliminary step, we show that the condition
	$\mathscr{M}\subseteq\mathcal{O}_{V}^{q}+\mathcal{L}^{q}$ implies that
	\begin{align}
		\label{O[M_L]} \mathcal{O}_{V}^{q}[
		\mathscr{M}_{\mathcal{L}}]=\mathcal{O}_{V}^{q}[(
		\mathcal{O}_{V}^{q}[\mathscr{M}])_{\mathcal{L}}],
	\end{align}
	and then we will show that (2) easily follows from this equality.
	
	Clearly,
	$\mathscr{M}_{\mathcal{L}}\subseteq(\mathcal{O}_{V}^{q}[\mathscr{M}])_{
		\mathcal{L}}$, and hence
	$\mathcal{O}_{V}^{q}[\mathscr{M}_{\mathcal{L}}]\subseteq\mathcal
	{O}_{V}^{q}[(
	\mathcal{O}_{V}^{q}[\mathscr{M}])_{\mathcal{L}}]$. As for the
	reverse inclusion,
	we only need to prove that
	\begin{equation*}
		(\mathcal{O}_{V}^{q}[\mathscr{M}])_{\mathcal{L}}
		\subseteq\mathcal {O}_{V}^{q}[ \mathscr{M}_{\mathcal{L}}].
	\end{equation*}
	Since $F^{e}\mathcal{L}\subset\mathcal{O}_{V}^{q}$ is an invertible ideal,
	the previous inclusion is equivalent to
	\begin{equation*}
		(F^{e}\mathcal{L}) (\mathcal{O}_{V}^{q}[
		\mathscr{M}])_{\mathcal{L}} \subseteq(F^{e}\mathcal{L})
		\mathcal{O}_{V}^{q}[\mathscr {M}_{\mathcal{L}}].
	\end{equation*}
	To prove this last inclusion, note that, on the one hand, Lemma \ref{L^qM_L}(1), applied to
	$\mathcal{O}_{V}^{q}[\mathscr{M}]=\mathcal{O}_{V}^{q}+\mathcal{O}_{V}^{q}[
	\mathscr{M}]$ and $\mathcal{L}$, says that
	\begin{align}
		\label{eq1:O^q[M_L]} \mathcal{O}_{V}^{q}[\mathscr{M}]\cap
		\mathcal{L}^{q}= (F^{e}\mathcal{L}) (
		\mathcal{O}_{V}^{q}[\mathscr{M}])_{\mathcal{L}}.
	\end{align}
	On the other hand, by statement (1),
	$\mathscr{M}\sim(F^{e}\mathcal{L})\mathscr{M}_{\mathcal{L}}$,
	whence they
	are also weakly equivalent, and
	\begin{align}
		\label{eq2:O^q[M_L]} \mathcal{O}_{V}^{q}[\mathscr{M}]&=
		\mathcal{O}_{V}^{q}[(F^{e}\mathcal{L})
		\mathscr{M}_{\mathcal{L}}]\nonumber
		\\
		&= \mathcal{O}_{V}^{q}+\sum
		_{n=1}^{\infty}((F^{e}
		\mathcal{L})\mathscr{M}_{\mathcal{L}})^{n}\nonumber
		\\
		&\subseteq
		\mathcal{O}_{V}^{q}+ \sum_{n=1}^{\infty}(F^{e}
		\mathcal{L}) (\mathscr{M}_{\mathcal{L}})^{n}.
	\end{align}
	We finally claim that
	\begin{eqnarray*}
		(F^{e}\mathcal{L}) (\mathcal{O}_{V}^{q}[
		\mathscr{M}])_{\mathcal{L}} &=& \mathcal{O}_{V}^{q}[
		\mathscr{M}]\cap\mathcal{L}^{q}
		\\
		&\subseteq& \biggl(\mathcal{O}_{V}^{q}+
		\sum_{n=1}^{\infty}(F^{e}
		\mathcal{L}) (\mathscr{M}_{\mathcal{L}})^{n}\biggr) \cap
		\mathcal{L}^{q}
		\\
		&=&(\mathcal{O}_{V}^{q}\cap
		\mathcal{L}^{q})+\sum_{n=1}^{\infty}(F^{e}
		\mathcal{L}) (\mathscr{M}_{\mathcal{L}})^{n}
		\\
		&=& (F^{e}
		\mathcal{L}) \mathcal{O}_{V}^{q}[\mathscr{M}_{\mathcal{L}}].
	\end{eqnarray*}
	In fact, the first equality is (\ref{eq1:O^q[M_L]}), and the first inclusion
	follows from (\ref{eq2:O^q[M_L]}). As for the second equality, we just
	need to note that
	$\sum_{n=1}^{\infty}(F^{e}\mathcal{L})(\mathscr{M}_{\mathcal
		{L}})^{n}$ is
	already included in $\mathcal{L}^{q}$. Finally, the last equality is clear
	since $\mathcal{O}_{V}^{q}\cap\mathcal{L}^{q}=F^{e}\mathcal{L}$.
	This completes
	the proof of (\ref{O[M_L]}). We now turn to the proof of (2). Notice first
	that
	$\tilde{\mathscr{M}}\subseteq\mathcal{O}_{V}^{q}[\tilde{\mathscr{M}}]=
	\mathcal{O}_{V}^{q}[\mathscr{M}]\subseteq\mathcal
	{O}_{V}^{q}+\mathcal{L}^{q}$,
	whence equality (\ref{O[M_L]}) holds for $\tilde{\mathscr{M}}$ as well.
	Therefore
	\begin{equation*}
		\mathcal{O}_{V}^{q}[\tilde{\mathscr{M}}_{\mathcal{L}}]=
		\mathcal {O}_{V}^{q}[( \mathcal{O}_{V}^{q}[
		\tilde{\mathscr{M}}])_{\mathcal{L}}]=\mathcal {O}_{V}^{q}[(
		\mathcal{O}_{V}^{q}[\mathscr{M}])_{\mathcal{L}}]=
		\mathcal{O}_{V}^{q}[ \mathscr{M}_{\mathcal{L}}],
	\end{equation*}
	which completes the proof of (2).
\end{proof}

\begin{parrafo}%
	\label{pull_back_of_O^q-modules}
	We now discuss the notion of pull-back of $\mathcal{O}_{V}^{q}$-submodules
	of $\mathcal{O}_{V}$ by morphisms. Let $V_{1}$ be a second irreducible
	$F$-finite regular scheme, and let $\pi:V_{1}\rightarrow V$ be a morphism.
	As usual, we write
	$\pi^{\sharp}:\pi^{-1}\mathcal{O}_{V}\rightarrow\mathcal
	{O}_{V_{1}}$ for
	the underlying morphism of sheaves. Clearly,
	\begin{equation*}
		\pi^{-1}\mathcal{O}_{V}^{q}=(
		\pi^{-1}\mathcal{O}_{V})^{q},
	\end{equation*}
	and hence by restriction of $\pi^{\sharp}$ we obtain a commutative diagram
	\begin{equation*}
		\xymatrix{\pi^{-1}\mathcal{O}_{V}\ar[r]^{\pi^\sharp}&\mathcal
			{O}_{V_{1}}\\
			\pi^{-1}\mathcal{O}_{V}^{q}\ar[u] \ar[r]^{\pi^\sharp}&\mathcal
			{O}_{V_{1}}^{q}\ar[u],
		}
	\end{equation*}
	where the vertical arrows are inclusions. This commutative diagram yields
	a morphism of $\mathcal{O}_{V_{1}}^{q}$-algebras
	\begin{align}
		\label{theta_tensor_1} \Phi:\pi^{-1}\mathcal{O}_{V}
		\otimes_{\pi^{-1}\mathcal{O}_{V}^{q}} \mathcal{O}_{V_{1}}^{q}\rightarrow
		\mathcal{O}_{V_{1}}.
	\end{align}
\end{parrafo}
\begin{definition}%
	\label{total_transform}
	Within the previous setting, if $\mathscr{M}$ is an
	$\mathcal{O}_{V}^{q}$-submodule of $\mathcal{O}_{V}$, then we denote by
	$\mathscr{M}\mathcal{O}_{V_{1}}^{q}$ the image of the composition
	\begin{equation*}
		\pi^{-1}\mathscr{M}\otimes_{\pi^{-1}\mathcal{O}_{V}^{q}}\mathcal
		{O}_{V_{1}}^{q} \xrightarrow{\iota\otimes\operatorname{id}}
		\pi^{-1}\mathcal{O}_{V} \otimes_{\pi^{-1}\mathcal{O}_{V}^{q}}
		\mathcal{O}_{V_{1}}^{q} \xrightarrow{\Phi}
		\mathcal{O}_{V_{1}},
	\end{equation*}
	where $\iota$ denotes the inclusion
	$\pi^{-1}\mathscr{M}\subseteq\pi^{-1}\mathcal{O}_{V}$.
\end{definition}
We readily check that for any $x_{1}\in V$,
\begin{align}
	\label{stalk_of_MO_V1^q} (\mathscr{M}\mathcal{O}_{V_{1}}^{q})_{x_{1}}=
	\mathscr{M}_{\pi(x_{1})} \mathcal{O}_{V_{1},x_{1}}^{q},
\end{align}
where the term on the right is the $\mathcal{O}_{V_{1},x_{1}}^{q}$-submodule
of $\mathcal{O}_{V_{1},x_{1}}$ spanned by the image of the composition map
$\mathscr{M}_{\pi(x_{1})}\hookrightarrow\mathcal{O}_{V,\pi(x_{1})}
\xrightarrow{\pi_{x_{1}}^{\sharp}} \mathcal{O}_{V_{1},x_{1}}$. The following
proposition is a direct consequence of this equality.
\begin{proposition}%
	\label{the_equivalence_class_of_the_transform_of_a_module}
	Fix a morphism $\pi:V_{1}\rightarrow V$ of $F$-finite regular schemes.
	\begin{enumerate}[(2)]
		\item[(1)] If $\mathscr{B}\subseteq\mathcal{O}_{V}$ is an
		$\mathcal{O}_{V}^{q}$-subalgebra, then
		$\mathscr{B}\mathcal{O}_{V_{1}}^{q}\subseteq\mathcal{O}_{V_{1}}$ is an
		$\mathcal{O}_{V_{1}}^{q}$-subalgebra.
		\item[(2)] If $\mathscr{M},\mathscr{M}'\subseteq\mathcal{O}_{V}$
		are equivalent
		(resp., weakly equivalent) $\mathcal{O}_{V}^{q}$-modules, then
		$\mathscr{M}\mathcal{O}_{V_{1}}^{q}$ and
		$\mathscr{M}'\mathcal{O}_{V_{1}}^{q}$ are equivalent (resp., weakly
		equivalent)
		$\mathcal{O}_{V_{1}}^{q}$-modules.
	\end{enumerate}
\end{proposition}

We now formulate different notions of transformation of modules when we
blow up along a regular center.
\begin{definition}%
	\label{the_a-transform_of_an_O^q-module}
	Let $Z\subset V$ be an irreducible regular subscheme, and let
	$V\xleftarrow{\pi} V_{1}\supset H_{1}$ be the blowup of $V$ along
	$Z$, where $H_{1}$ denotes the exceptional hypersurface. Let
	$\mathscr{M}\subseteq\mathcal{O}_{V}$ be an $\mathcal{O}_{V}^{q}$-submodule.
	\begin{enumerate}[(2)]
		\item[(1)] The \emph{total transform} of $\mathscr{M}$ by the blowup
		is the
		$\mathcal{O}_{V_{1}}^{q}$-submodule
		$\mathscr{M}\mathcal{O}_{V_{1}}^{q}\subseteq\mathcal{O}_{V_{1}}$.%
		\item[(2)] For each positive integer $a$, the \emph{$a$-transform} of
		$\mathscr{M}$ by the blowup is the $\mathcal{I}(H_{1})^{a}$-conductor of
		$\mathscr{M}\mathcal{O}_{V_{1}}^{q}\subset\mathcal{O}_{V_{1}}$ (see
		Definition \ref{M_L}(1)). We denote it $\mathscr{M}_{1}^{(a)}$.
	\end{enumerate}
\end{definition}
According to Proposition \ref{the_equivalence_class_of_the_transform_of_a_module}, if
$\mathscr{M}$ and $\mathscr{N}$ are equivalent $\mathcal{O}_{V}^{q}$-modules,
then $\mathscr{M}\mathcal{O}_{V_{1}}^{q}$ and
$\mathscr{N}\mathcal{O}_{V_{1}}^{q}$ are equivalent
$\mathcal{O}_{V_{1}}^{q}$-modules; therefore $\mathscr{M}$ and
$\mathscr{N}$ have equivalent $a$-transforms for all $a\geq1$.

The notion of $a$-transform of an $\mathcal{O}_{V}^{q}$-module
$\mathscr{M}\subseteq\mathcal{O}_{V}$ will be of particular interest when
$\mathscr{M}$ is equivalent to an $\mathcal{O}_{V}^{q}$-module
included in
$\mathcal{I}(Z)^{qa}$. In that case, if $X\rightarrow V$ denotes the
$V$-scheme attached to $\mathscr{M}$ (i.e., the one defined by
$\mathcal{O}_{V}\subset(\mathcal{O}_{V}[\mathscr{M}])^{1/q}$), then
we will
further see that the $V_{1}$-scheme attached to the $a$-transform of
$\mathscr{M}$ will be obtained from $\delta$ by an iteration of diagrams
like (\ref{eqdiac}).
\begin{definition}%
	\label{permissible_center_for_(M,a)}
	For an $\mathcal{O}_{V}^{q}$-module
	$\mathscr{M}\subseteq\mathcal{O}_{V}$ and a positive integer $a$, we set
	\begin{align*}
		\operatorname{Sing}(\mathscr{M},a):=\{x\in V: \mathscr {M}_{x}
		\subseteq \mathcal{O}_{V,x}^{q}+m_{V,x}^{qa}
		\}.
	\end{align*}
	An irreducible regular closed subscheme $Z\subset V$ is called a \emph
	{permissible
		center} for $(\mathscr{M},a)$ if the following equivalent conditions hold:
	\begin{enumerate}[(3)]
		\item[(1)] $\mathscr{M}\subseteq\mathcal{O}_{V}^{q}+\mathcal{I}(Z)^{qa}$.
		\item[(2)] $\mathscr{M}\sim\mathscr{M}'$ for some
		$\mathscr{M}'\subseteq\mathcal{I}(Z)^{qa}$.
		\item[(3)] $Z\subset\operatorname{Sing}(\mathscr{M},a)$.
	\end{enumerate}
	The equivalence of (1) and (2) follows from Lemma \ref{When_M_is_generated_by_elements_in_I}, and they clearly imply (3).
	The converse will be proved in Proposition \ref{characterization_of_permissible_centers_for_a_pair_(M,a)} using
	$p$-basis.
\end{definition}
\begin{remark}
	We will see in Proposition \ref{d(F(X))_is_closed} that each
	$\operatorname{Sing}(\mathscr{M},a)\subset V$ is a closed subset.
\end{remark}

\begin{proposition}%
	\label{Well_definition_of_the_transforms}
	Assume that $Z\subset V$ is permissible for $(\mathscr{M},a)$, and let
	$V\leftarrow V_{1}\supset H_{1}$ be the blowup along $Z$, where
	$H_{1}$ denotes the exceptional hypersurface. Then $H_{1}$ is permissible
	for $(\mathscr{M}\mathcal{O}_{V_{1}}^{q},a)$, and we have
	$\mathscr{M}\mathcal{O}_{V_{1}}^{q}\sim F^{e}(\mathcal{I}(H_{1})^{a})
	\mathscr{M}_{1}^{(a)}$, where $\mathscr{M}_{1}^{(a)}$ is the $a$-transform
	of $\mathscr{M}$ (Definition \ref{the_a-transform_of_an_O^q-module}(2)).
	
	Let $\mathscr{N}\subset\mathcal{O}_{V}$ be another
	$\mathcal{O}_{V}^{q}$-submodule of $\mathcal{O}_{V}$ that is weakly equivalent
	to $\mathscr{M}$ (i.e.,
	$\mathcal{O}_{V}^{q}[\mathscr{M}]=\mathcal{O}_{V}^{q}[\mathscr
	{N}]$). Then
	$Z$ is also permissible for $(\mathscr{N},a)$, and we have
	$\mathcal{O}_{V_{1}}^{q}[\mathscr{M}_{1}^{(a)}]=\mathcal{O}_{V_{1}}^{q}[
	\mathscr{N}_{1}^{(a)}]$.
\end{proposition}
\begin{proof}
	By the definition of $Z$ to be permissible for $(\mathscr{M},a)$, we have
	$\mathscr{M}\sim\mathscr{M}'$ for some
	$\mathscr{M}'\subseteq\mathcal{I}(Z)^{qa}$. The last inclusion
	implies that
	$\mathscr{M}'\mathcal{O}_{V_{1}}^{q}\subseteq\mathcal
	{I}(H_{1})^{qa}$. Since
	$\mathscr{M}\mathcal{O}_{V_{1}}^{q}\sim\mathscr{M}'\mathcal{O}_{V_{1}}^{q}$
	(Proposition \ref{the_equivalence_class_of_the_transform_of_a_module}),
	we deduce that $H_{1}$ is permissible for
	$(\mathscr{M}\mathcal{O}_{V_{1}}^{q},a)$. This enables use to apply Lemma \ref{Well_definition_of_the_weak_equivalence_class_of_M_L} to
	$\mathscr{M}\mathcal{O}_{V_{1}}^{q}$ and $\mathcal{I}(H_{1})^{a}$
	and conclude
	that
	$\mathscr{M}\mathcal{O}_{V_{1}}^{q}\sim F^{e}(\mathcal{I}(H_{1})^{a})
	\mathscr{M}_{1}^{(a)}$. This proves the first part of the proposition.
	
	Assume now that
	$\mathcal{O}_{V}^{q}[\mathscr{M}]=\mathcal{O}_{V}^{q}[\mathscr{N}]$
	for a second
	$\mathcal{O}_{V}^{q}$-module $\mathscr{N}\subset\mathcal{O}_{V}$. Then
	$\mathscr{N}\subset\mathcal{O}_{V}^{q}[\mathscr{N}]=\mathcal{O}_{V}^{q}[
	\mathscr{M}]=\mathcal{O}_{V}^{q}[\mathscr{M}']\subset\mathcal{O}_{V}^{q}+
	\mathcal{I}(Z)^{qa}$, which implies that $Z$ is also permissible for
	$(\mathscr{N},a)$. On the other hand,
	$\mathcal{O}_{V_{1}}^{q}[\mathscr{M}\mathcal
	{O}_{V_{1}}^{q}]=\mathcal{O}_{V_{1}}^{q}[
	\mathscr{N}\mathcal{O}_{V_{1}}^{q}]$ by Proposition \ref{Well_definition_of_the_transforms}. Now the first part tells us that
	$H_{1}$ is permissible for $(\mathscr{M}\mathcal{O}_{V_{1}}^{q},a)$, which
	means that
	$\mathscr{M}\mathcal{O}_{V_{1}}^{q}\subseteq\mathcal
	{O}_{V}^{q}+\mathcal{I}(H_{1})^{qa}$.
	Therefore we can apply Lemma \ref{Well_definition_of_the_weak_equivalence_class_of_M_L}(2) to
	$\mathscr{M}:=\mathscr{M}\mathcal{O}_{V_{1}}^{q}$,
	$\mathcal{L}:=\mathcal{I}(H_{1})^{a}$ and
	$\tilde{\mathscr{M}}:=\mathscr{N}\mathcal{O}_{V_{1}}^{q}$. The conclusion
	is that
	$\mathcal{O}_{V_{1}}^{q}[\mathscr{M}_{1}^{(a)}]=\mathcal{O}_{V_{1}}^{q}[
	\mathscr{N}_{1}^{(a)}]$. This completes the proof of the proposition.
\end{proof}

\subsection{Relation with Finite Radicial Morphisms}
\label{Relation_with_finite_radicial_morphisms}

In this second part of the section, we discuss the interplay between
$\mathcal{O}_{V}^{q}$-submodules $\mathscr{M}\subset\mathcal
{O}_{V}$ and
certain class of finite radicial morphisms $\delta:X\to V$. We begin with
some remarks about this class of morphisms.
\begin{parrafo}%
	\label{equivalence_of_categories-first_step}
	We denote by $\mathscr{C}_{q}(V)$ the class of all finite surjective morphisms
	$\delta:X\rightarrow V$, where $X$ is an integral scheme, and there is
	an inclusion
	$\delta_{*}(\mathcal{O}_{X}^{q})\subseteq\mathcal{O}_{V}$. We can view
	$\mathscr{C}_{q}(V)$ as a full subcategory of the category of $V$-schemes.
	In this regard, the assignment
	$(X\xrightarrow{\delta} V)\mapsto\delta_{*}(\mathcal{O}_{X})$ defines
	an antiequivalence between $\mathscr{C}_{q}(V)$ and the category of
	$\mathcal{O}_{V}$-subalgebras of $\mathcal{O}_{V}^{1/q}$. Here
	$\mathcal{O}_{V}$ is seen as embedded in the function field $K(V)$ of
	$V$, and $\mathcal{O}_{V}^{1/q}$ is the $\mathcal{O}_{V}$-subalgebra of
	$\overline{K(V)}$ (a fixed algebraic closure of $K(V)$) whose sections on
	an open subset $U\subset V$ are given by
	$\Gamma(U,\mathcal{O}_{V}^{1/q})=\{s\in\overline{K(V)}: s^{q}\in
	\Gamma(U,\mathcal{O}_{V})\}$. The following lemma expresses a characteristic
	property of radicial morphisms.
\end{parrafo}
\begin{lemma}%
	\label{lemma_on_purely_inseparable_morphisms}
	Given $\delta:X\rightarrow V$ in $\mathscr{C}_{q}(V)$ and a reduced
	$V$-scheme $Y$, there is at most one $V$-morphism $f:Y\rightarrow X$.
\end{lemma}
\begin{proof}
	We may assume that both $V$ and $X$ are affine, say
	$X=\operatorname{Spec}(B)$ and $V=\operatorname{Spec}(S)$, so that
	$B^{q}\subseteq S\subseteq B$. We have to prove that there is at most one
	$S$-homomorphism $\varphi:B\rightarrow\Gamma(Y,\mathcal{O}_{Y})$. Now
	for any two such homomorphisms, say $\varphi$ and $\varphi'$, we have
	that for any $b\in B$,
	$(\varphi(b))^{q}=\varphi(b^{q})=\varphi'(b^{q})=(\varphi'(b))^{q}$,
	whence $\varphi(b)=\varphi'(b)$ since $\Gamma(Y,\mathcal{O}_{Y})$ is
	reduced.
\end{proof}

\begin{corollary}%
	\label{coro:equivalence_of_categories}
	The assignment
	$(X\xrightarrow{\delta} V)\mapsto\delta_{*}(\mathcal{O}_{X}^{q})$ defines
	an antiequivalence between $\mathscr{C}_{q}(V)$ and the category of
	$\mathcal{O}_{V}^{q}$-subalgebras of $\mathcal{O}_{V}$. The only possible
	morphisms in the latter category are inclusions. Therefore, given
	$\delta:X\rightarrow V$ with associated $\mathcal{O}_{V}^{q}$-subalgebra
	$\mathscr{B}\subseteq\mathcal{O}_{V}$ and $\delta':X'\rightarrow V$ with
	$\mathcal{O}_{V}^{q}$-subalgebra
	$\mathscr{B}'\subseteq\mathcal{O}_{V}$, there exists a $V$-morphism
	$X'\rightarrow X$ if and only if $\mathscr{B}\subseteq\mathscr{B}'$, and
	in that case the $V$-morphism is unique. In particular, $X$ and $X'$ are
	$V$-isomorphic if and only if $\mathscr{B}=\mathscr{B}'$.
\end{corollary}
\begin{proof}
	The first statement (on the antiequivalence of categories) follows from
	the discussion in \ref{equivalence_of_categories-first_step} and from the
	fact that the category of $\mathcal{O}_{V}$-subalgebras of
	$\mathcal{O}_{V}^{1/q}$ is equivalent, via Frobenius, to the category of
	$\mathcal{O}_{V}^{q}$-subalgebras of $\mathcal{O}_{V}$. Next, given two
	$\mathcal{O}_{V}^{q}$-subalgebras
	$\mathscr{B},\mathscr{B}'\subseteq\mathcal{O}_{V}$, assume that
	there exists
	a morphism of $\mathcal{O}_{V}^{q}$-algebras
	$\mathscr{B}\rightarrow\mathscr{B}'$. It follows from Lemma \ref{lemma_on_purely_inseparable_morphisms} that the inclusion
	$\mathscr{B}\subset\mathcal{O}_{V}$ and the composition
	$\mathscr{B}\rightarrow\mathscr{B}'\subset\mathcal{O}_{V}$
	coincide, whence
	$\mathscr{B}\rightarrow\mathscr{B}'$ has to be an inclusion. The
	last two
	conclusions in the corollary are now immediate.
\end{proof}

Before continuing with our discussion, we need to review some properties
of integral closure of ideals.
\begin{parrafo}%
	\label{integral_closure_of_ideals}
	Given a Noetherian ring $R$ and an ideal $I\subset R$, we denote by
	$\overline{I}$ the integral closure of $I$. This is an ideal of $R$ and
	consists of all $r\in R$ satisfying a relation of the form
	$r^{n}+a_{1}r^{n-1}+\cdots+a_{n}r^{0}=0$ for some $n\geq1$ and
	$a_{i}\in I^{i}$, $i=1,\ldots,n$. An ideal $I$ is said to be integrally
	closed if $\overline{I}=I$. Given an inclusion of ideals
	$J\supset I$, $J$ is said to be \emph{integral} over $I$ or that $I$ is
	a \emph{reduction} of $J$ if $J\subseteq\overline{I}$. In this case,
	$\overline{J}=\overline{I}$. We list some properties that will be used
	along this section. Properties (b), (c), and (d) can be found in the table
	of basic properties of \cite{hunekeswanson}, and (g) can be deduced from
	(23) of that table. In what follows, $I\subset J$ denotes ideals in the
	Noetherian ring $R$.
	\begin{enumerate}[(d)]
		\item[(a)] $J$ is integral over $I$ if and only if the Rees ring
		$R_{J}:=R\oplus J\oplus J^{2}\oplus\ldots$ is finite over
		$R_{I}:=R\oplus I\oplus I^{2}\oplus\ldots$. This follows from
		\cite[Proposition 5.2.1]{hunekeswanson}
		\item[(b)] If $J$ is integral over $I$, then for any $R$-algebra $S$,
		$JS$ is integral over $IS$ (as $S$-ideals).
		\item[(c)] Given an integral extension $R\subset S$, we have
		$\overline{I}=R\cap\overline{IS}$.
		\item[(d)] Given a multiplicative subset $T\subset R$, we have
		$\overline{T^{-1}I}=T^{-1}\overline{I}$ (in $T^{-1} R$).
		\item[(e)] If $J=(f)\subset R$ for a nonzero divisor $f\in R$, then
		$J$ has
		no proper reductions, that is, if $I\subset J$ is a reduction of $J$, then
		$I=J$.
		
		\emph{Proof:} Assume that $I$ is a reduction of $J$. This means that there
		is a relation of the form $f^{n}+a_{1}f^{n-1}+\cdots+a_{n}=0$ with
		$a_{i}\in I^{i}$ $n>0$. Assume that this relation is of minimal degree. We obtain that
		$a_{n}=a_{n}'f$ for some $a_{n}'\in R$. If $n>1$, then the relation of
		integral dependence becomes
		$f(f^{n-1}+a_{1}f^{n-2}+\cdots+(a_{n-1}+a_{n}'))=0$, which implies that
		$f^{n-1}+a_{1}f^{n-2}+\cdots+(a_{n-1}+a_{n}')=0$ since $f$ is not a
		zero-divisor.
		This contradicts the minimality of our initial relation. Therefore $n=1$ and
		$f=-a_{n}\in I$.
		\item[(f)] If $(R,\mathfrak{m})$ is a regular local ring, then $\m$
		has no proper
		reductions (see \cite[Section~6, Cor. 2]{northcott1954rees}).
		\item[(g)] If $R$ is a regular local ring, and $P\subset R$ is a prime ideal
		such that $R/P$ is regular, then $P^{n}=\overline{P^{[n]}}$, where
		$P^{[n]}$ is the ideal generated by $\{x^{n}:x\in P\}$; in particular,
		$P^{n}$ is integrally closed.
		\item[(h)] Corollary of (c) and (g): if $R$ is an $F$-finite regular
		local ring,
		and if $O\subset R$ is an $R^{q}$-subalgebra ($q=p^{e}$), then for any
		prime $P\subset R$ such that $R/P$ is regular, the contraction
		$O\cap P^{q}\subset O$ is the integral closure of
		$F^{e}P\cdot O \subset O$, the ideal of $O$ generated by
		$\{x^{q}: x\in P\}$; in particular, $O\cap P^{q}$ is an integrally closed
		ideal of $O$.%
		
		\emph{Proof:} By (c)
		$\overline{F^{e}P\cdot O}=O\cap\overline{F^{e}P\cdot R}=O\cap
		\overline{P^{[q]}}$, and by (g) this is equal to $O\cap P^{q}$.
	\end{enumerate}
	Property (d) enables us to extend the notion of integral closure to coherent
	ideals on schemes. Namely, given a scheme $W$ and an
	$\mathcal{O}_{W}$-ideal $\mathcal{I}$, there is an $\mathcal{O}_{W}$-ideal
	$\overline{\mathcal{I}}$, called the integral closure of $\mathcal
	{I}$, such
	that if $U\subset W$ is an affine open subset, then
	$\overline{\mathcal{I}}(U)$ is the integral closure of
	$\mathcal{I}(U)\subseteq\mathcal{O}_{V}(U)$. Given an inclusion of
	$\mathcal{O}_{W}$-ideals $\mathcal{I}\subset\mathcal{J}$, $\mathcal
	{J}$ is
	said to be integral over $\mathcal{I}$ if
	$\overline{\mathcal{J}}=\overline{\mathcal{I}}$. In this case,
	property (b)
	implies that $\mathcal{J}\cdot\mathcal{O}_{W'}$ is integral over
	$\mathcal{I}\cdot\mathcal{O}_{W'}$ for any morphism of schemes 
	$W'\rightarrow W$. Property (e) implies that if $\mathcal{L}$ is an invertible
	ideal of $\mathcal{O}_{W}$, then $\mathcal{L}$ is not integral over any
	$\mathcal{O}_{W}$-ideal properly included in $\mathcal{L}$. Finally, if
	$V$ is an irreducible $F$-finite regular scheme and $Z\subset V$ is an
	irreducible regular closed subscheme, then (g) implies that
	$\mathcal{I}(Z)^{q}\subseteq\mathcal{O}_{V}$ is integrally closed.
	In addition,
	if $\mathscr{B}\subset\mathcal{O}_{V}$ is an $\mathcal
	{O}_{V}^{q}$-subalgebra,
	then (h) implies that the contraction
	$\mathscr{B}\cap\mathcal{I}(Z)^{q}\subset\mathscr{B}$ is an
	integrally closed
	$\mathscr{B}$-ideal; indeed, it is the integral closure of the
	$\mathscr{B}$-ideal $F^{e}(\mathcal{I}(Z))\cdot\mathscr{B}$.
\end{parrafo}

We return to our general discussion and delve into the proof of the properties
stated in Theorem \ref{main_theorem_2}. We also include proofs of the
statements
given in Theorem \ref{finite_morphisms_and_closed_subsets} (in our particular
setting) to make the presentation self-contained.
\begin{proposition}%
	\label{0309}
	Given $\delta:X\to V$ in $\mathscr{C}_{q}(V)$, the maximal multiplicity
	along points of $X$ is $d:=[K(X):K(V)]$, the generic rank. Let
	$F_{d}(X)\subset X$ denote the set of points where the multiplicity is
	$d$, and let $\mathscr{M}\subset\mathcal{O}_{V}$ be any
	$\mathcal{O}_{V}^{q}$-submodule of $\mathcal{O}_{V}$ that defines
	$\delta$. Then
	\begin{equation*}
		\delta(F_{d}(X))=\operatorname{Sing}(\mathscr{M},1).
	\end{equation*}
\end{proposition}
(By Proposition \ref{d(F(X))_is_closed} this implies that
$F_{d}(X)\subset X$ is a closed subset.)
\begin{proof}
	Since
	$\operatorname{Sing}(\mathscr{M},1)=\operatorname{Sing}(\mathcal
	{O}_{V}^{q}[\mathscr{M}],1)$,
	the proposition follows from the following local version.
\end{proof}
\begin{lemma}%
	\label{lem1}
	Let $(S,\m)\subset(B,M)$ be a finite extension of local domains of
	characteristic $p$ such that $(S,\m)$ is regular and $F$-finite, and
	assume that $B^{q}\subset S$. Then the multiplicity of $B$ is at most
	$d:=[\operatorname{Frac}(B):\operatorname{Frac}(S)]$. In addition,
	the following are
	equivalent.
	\begin{enumerate}[(3)]
		\item[(1)] The multiplicity of $B$ is $d$.
		\item[(2)] $B=S+M$ and $M=\overline{\m B} (\subset B)$.
		\item[(3)] $B^{q}\subseteq S^{q}+\m^{q} (\subset S)$; in
		particular, there
		is an inclusion $F^{e}M\subseteq\m^{q}$ (as subsets of $S$).
	\end{enumerate}
\end{lemma}
\begin{proof}
	The first part of the lemma and the equivalence (1) $\Leftrightarrow$ (2)
	follows from Zariski's formula
	\cite[VIII, Cor. 1 of Thm. 24]{zariski1960commutative} and a result by
	Rees relating multiplicity with integral closure of ideals (\cite[Thm.
	3.2]{rees1961transforms}).
	We refer to \cite[Section~4]{villamayor2014equimultiplicity} for details.
	Since we are dealing only with local domains, and the extension
	$S\subset B$ is purely inseparable of finite exponent, all the required
	hypothesis in the references are satisfied.
	
	As for the equivalence (2) $\Leftrightarrow$ (3), clearly (2) is equivalent
	to the equality $B=S+\overline{\m B}$, which in turn is equivalent
	to the equality $B^{q}=S^{q}+\overline{F^{e}\m\cdot B^{q}}$ (the integral
	closure being taken inside $B^{q}$). By \ref{integral_closure_of_ideals}(h) the latter is translated as
	$B^{q}=S^{q}+(\m^{q}\cap B^{q})$, which is clearly equivalent to the
	inclusion $B^{q}\subseteq S^{q}+\m^{q}$.
\end{proof}

Part (2) and formula (\ref{identity_with_Frobenius}) of the next proposition
are of technical nature, and they are there for later use.
\begin{proposition}%
	\label{0310}
	Given $\delta:X\to V$ in $\mathscr{C}_{q}(V)$, let $d:=[K(X):K(V)]$ and
	assume that $F_{d}(X)\neq\emptyset$. Let $Y\subset X$ be an integral
	subscheme of $X$ and set $Z:=(\delta(Y))_{\mathrm{red}}\subset V$. If
	$\mathscr{M}\subset\mathcal{O}_{V}$ is any $\mathcal{O}_{V}^{q}$-submodule
	of $\mathcal{O}_{V}$ defining $\delta$, then the following conditions are
	equivalent.
	\begin{enumerate}[(3)]
		\item[(1)] $Y$ is regular and included in $F_{d}(X)$.
		\item[(2)] $Z$ is regular, $\delta$ induces an isomorphism $Y\cong
		Z$, and
		$\mathcal{I}(Y)=\overline{\mathcal{I}(Z)\cdot\mathcal{O}_{X}}$.
		\item[(3)] $Z\subset V$ is regular, and
		$\mathscr{M}\subseteq\mathcal{O}_{V}^{q}+\mathcal{I}(Z)^{q}$, or say
		$Z$ is a permissible center for $(\mathscr{M},1)$ (Definition \ref{permissible_center_for_(M,a)}).
	\end{enumerate}
	In addition, if these equivalent conditions hold, then
	\begin{align}
		\label{identity_with_Frobenius} F^{e}(\mathcal{I}(Y))&=\mathcal{I}(Z)^{q}
		\cap\mathcal{O}_{V}^{q}[ \mathscr{M}].
	\end{align}
\end{proposition}
\begin{proof}
	Note that we can replace $\mathscr{M}$ by
	$\mathcal{O}_{V}^{q}[\mathscr{M}]$, and the outcome of the
	proposition does
	not change. It is clear therefore that everything follows from the following
	local version.
\end{proof}

\begin{lemma}
	Let the hypothesis and notation be as in Lemma \ref{lem1}, and assume that
	the multiplicity of $B$ is $d=[\operatorname{Frac}(B):\operatorname
	{Frac}(S)]$. Let
	$Q\subset B$ be a prime ideal and $\mathfrak{q}:=S\cap Q \subset S$. Then
	the following conditions are equivalent.
	\begin{enumerate}[(3)]
		\item[(1)] $B/Q$ is regular, and $B_{P}$ has multiplicity $d$ for
		every prime
		ideal $P\subset B$ that includes $Q$.
		\item[(2)] $B/Q$ is regular, and $B_{Q}$ has also multiplicity $d$.
		\item[(3)] $S/\mathfrak{q}$ is regular, $B=S+Q$, and
		$Q=\overline{\mathfrak{q}B}$.
		\item[(4)] $S/\mathfrak{q}$ is regular, and
		$B^{q}\subset S^{q}+\mathfrak{q}^{q}$.
	\end{enumerate}
	If these equivalent conditions hold, then
	$F^{e} Q=\mathfrak{q}^{q}\cap B^{q}$.
\end{lemma}
\begin{proof}
	The implication (1) $\Rightarrow$ (2) is trivial. We now prove (2)
	$\Rightarrow$ (3) $\Rightarrow$ (4) $\Rightarrow$ (1).
	
	By Lemma \ref{lem1} we have $B=S+M$ and $M=\overline{\m B}$. By \ref{integral_closure_of_ideals}(b) the latter equality implies that
	$M/Q$ is integral over $(\m B+Q)/Q$ (as an ideal of $B/Q$). Assume
	(2). As $B/Q$ is regular, its maximal ideal $M/Q$ has no proper reductions
	(we are using \ref{integral_closure_of_ideals}(f)); therefore
	$M/Q=(\m B+Q)/Q$, which means that $M=\m B+Q$. It follows that
	$B=S+M=S+Q+\m B$, and hence $B=S+Q$ by Nakayama's lemma. This implies
	that $S/\mathfrak{q}=B/Q$, and hence $S/\mathfrak{q}$ is regular. We now prove
	that $Q=\overline{\mathfrak{q}B}$. By applying the implication
	$(1)\Rightarrow(3)$ of \ref[Lemma]{lem1} to the inclusion
	$S_\mathfrak{q}\subseteq B_{Q}$ we obtain that for all $x\in Q$,
	$x^{q}\in\mathfrak{q}^{q}S_\mathfrak{q}\cap S=\mathfrak{q}^{q}$, where
	the last
	equality holds since usual and symbolic powers of a regular prime in a
	regular ring coincide. We have shown that any $x\in Q$ is integral over
	$\mathfrak{q}B$ since such $x$ satisfies the relation of integral dependence
	$T^{q}-x^{q}=0$; thus $Q=\overline{\mathfrak{q}B}$. This completes the proof
	of (2) $\Rightarrow$ (3).
	
	Assume (3). It follows that $B^{q}=S^{q}+F^{e}Q$ and that
	$F^{e}Q=\overline{F^{e}\mathfrak{q}\cdot B^{q}}$, and the latter is equal
	to $\mathfrak{q}^{q}\cap B^{q}$ by \ref{integral_closure_of_ideals}(h) when
	applied to $O:=B^{q}$ and $P:=\mathfrak{q}$. We have shown that (3) implies
	(4) and the last sentence of the lemma.
	
	Assume finally (4). On the one hand, we have that
	$B^{q}=S^{q}+\mathfrak{q}^{q}\cap B^{q}$, and by applying
	\ref{integral_closure_of_ideals}(h), as we did before, we obtain that
	$B^{q} =S^{q}+\overline{F^{e}\mathfrak{q}\cdot B^{q}}$. By applying the inverse
	of Frobenius we obtain $B=S+\overline{\mathfrak{q}B}$. This equality implies
	that $B/\overline{\mathfrak{q}B}\cong S/\mathfrak{q}$; in particular,
	$\overline{\mathfrak{q}B}$ is prime. Since $S/\mathfrak{q}\subset B/Q$
	is finite,
	by dimension reasons we have that $Q=\overline{\mathfrak{q}B}$. In particular,
	$B/Q\cong S/\mathfrak{q}$ is regular. On the other hand, let
	$P\subset B$ be any prime ideal including $Q$, and set
	$\mathfrak{p}:=P\cap S$. The inclusion
	$B^{q}\subseteq S^{q}+\mathfrak{q}^{q}$ implies clearly that
	$(B_{P})^{q}\subseteq S_\mathfrak{p}^{q}+(\mathfrak{p}S_\mathfrak{p})^{q}$. According
	to Lemma \ref{lem1}, the latter inclusion implies that $B_{P}$ has multiplicity
	$d$. This completes the proof of (4) $\Rightarrow$ (1) and hence also the
	proof of the lemma.
\end{proof}

Propositions \ref{0309} and \ref{0310} express $F_{d}(X)$ and the condition
of being a regular center included in $F_{d}(X)$ completely in terms of
the $\mathcal{O}_{V}^{q}$-submodule
$\mathscr{M}\subset\mathcal{O}_{V}$ which defines $\delta:X\to V$.
We now
express certain transformations of $X$ in terms of transformations of the
module $\mathscr{M}$. For instance, the next proposition shows that the
definition of pull-back of $\mathcal{O}_{V}^{q}$-modules given in Definition \ref{total_transform} corresponds to a notion of base change of finite
morphisms.
\begin{proposition}%
	\label{B_pi_is_the_total_transform_of_B}
	Let $\mathscr{M}\subset\mathcal{O}_{V}$ be an $\mathcal
	{O}_{V}^{q}$-submodule,
	and let $\delta:X\to V$ be the associated morphism in
	$\mathscr{C}_{q}(V)$. Given a morphism $V\xleftarrow{\pi}V_{1}$ of
	irreducible
	$F$-finite regular schemes, the reduced fiber product
	$(X\times_{V} V_{1})_{\textrm{red}}$ is integral, the projection
	$\delta_{\pi}:(X\times_{V} V_{1})_{\textrm{red}}\rightarrow V_{1}$ belongs
	to $\mathscr{C}_{q}(V_{1})$, and this is the $V_{1}$-scheme defined by the
	$\mathcal{O}_{V_{1}}^{q}$-submodule
	$\mathscr{M}\mathcal{O}_{V_{1}}^{q}\subseteq\mathcal{O}_{V_{1}}$.
	Diagrammatically,
	\begin{align*}
		&\xymatrix@C=50pt{ X\ar[d]_\delta&\ar[l]_{\pi_\delta} (X\times_{V}
			V_{1})_{\textrm{red}}\ar[d]^{\delta_\pi}%
			\\
			V&\ar[l]^\pi V_{1}}
		\\
		&\xymatrix@C=60pt{\mathscr{M}&\mathscr{M}\mathcal{O}_{V_{1}}^{q}}
	\end{align*}
\end{proposition}
\begin{proof}
	Since $\delta:X\rightarrow V$ is finite and
	$\mathcal{O}_{X}^{q}\subset\mathcal{O}_{V}$, the base change
	$X\times_{V} V_{1}\rightarrow V_{1}$ is also finite and satisfies
	$\mathcal{O}_{X\times_{V} V_{1}}^{q}\subset\mathcal{O}_{V_{1}}$. Therefore
	$\mathcal{O}_{(X\times_{V} V_{1})_{\textrm{red}}}=(\mathcal{O}_{X
		\times_{V} V_{1}})_{\textrm{red}}$ is the natural image of
	$\mathcal{O}_{X\times_{V} V_{1}}=\pi^{-1}\mathcal{O}_{X}\otimes_{
		\pi^{-1}\mathcal{O}_{V}}\mathcal{O}_{V_{1}}$ in
	$\mathcal{O}_{V_{1}}^{1/q}$; in particular,
	$(X\times_{V} V_{1})_{\textrm{red}}$ is integral. Equivalently, by applying
	$F^{e}$ we obtain that
	$(\mathcal{O}_{(X\times_{V} V_{1})_{\textrm{red}}})^{q}$ is the natural
	image of
	$\pi^{-1}\mathcal{O}_{X}^{q}\otimes_{\pi^{-1}\mathcal{O}_{V}^{q}}
	\mathcal{O}_{V_{1}}^{q}$ in $\mathcal{O}_{V_{1}}$, which is
	$\mathcal{O}_{X}^{q}\mathcal{O}_{V_{1}}^{q}=\mathcal
	{O}_{V}^{q}[\mathscr{M}]
	\mathcal{O}_{V_{1}}^{q}$; see Definition \ref{total_transform}. The latter
	is equal to
	$\mathcal{O}_{V_{1}}^{q}[\mathscr{M}\mathcal{O}_{V_{1}}^{q}]$ by Proposition \ref{the_equivalence_class_of_the_transform_of_a_module}(2). This shows
	that the projection
	$\delta_{\pi}:(X\times_{V} V_{1})_{\textrm{red}}\to V_{1}$ is the
	$V_{1}$-scheme in $\mathscr{C}_{q}(V_{1})$ defined by
	$\mathscr{M}\mathcal{O}_{V_{1}}^{q}$.
\end{proof}

For completeness, we include the proof of Theorem \ref{finite_morphisms_and_closed_subsets}(2). We formulate it in our particular
setting of purely inseparable morphisms, though the same proof we give
here applies to the general case.
\begin{proposition}%
	\label{blow-ups_of_radicial_morphisms}
	Let the setting be as in \ref[Proposition]{0310}, and assume that the equivalent
	conditions given there are satisfied. Let $f:X_{1}\to X$ denote the blowup
	of $X$ along $Y$, and let $\pi:V_{1}\to V$ denote the blowup of $V$ along
	$Z$.
	\begin{enumerate}[(2)]
		\item[(1)] There exists a unique morphism $\delta_{1}:X_{1}\to V_{1}$ making
		the diagram
		\begin{align*}
			\xymatrix{X\ar[d]_\delta& X_{1}\ar[l]_f\ar[d]^{\delta_{1}}\\ V&
				V_{1} \ar[l]^\pi}
		\end{align*}
		commutative. In addition, $\delta_{1}$ belongs to
		$\mathscr{C}_{q}(V_{1})$, and
		$\mathcal{I}(Y)\cdot\mathcal{O}_{X_{1}}=\mathcal{I}(H_{1})\cdot
		\mathcal{O}_{X_{1}}$,
		where $H_{1}\subset V_{1}$ is the exceptional divisor of $\pi$.
		\item[(2)] We now consider the commutative diagram
		\begin{align*}
			\xymatrix{X\ar[d]_{\delta}& X\times_{V} V_{1}\ar[l]_{\pi_\delta
				}\ar[d]^{\delta_\pi}& X_{1}\ar[l]_{f_{1}}\ar[ld]^{\delta_{1}}\\
				V& V_{1} \ar[l]^\pi}
		\end{align*}
		deduced from the previous one and the universal property of fiber products.
		Then $\delta_{\pi}$ is in $\mathscr{C}_{q}(V_{1})$, has generic rank
		$d$, and is defined by the $\mathcal{O}_{V_{1}}^{q}$-submodule
		$\mathscr{M}\mathcal{O}_{V_{1}}^{q}\subset\mathcal{O}_{V_{1}}$. In addition,
		the equivalent conditions of \ref[Proposition]{0310} hold for
		$\delta_{\pi}$ and the inverse image subscheme
		$Y_{1}:=\pi_{\delta}^{-1}(Y)\subset X\times_{V} V_{1}$, which in turn
		satisfies $(\delta_{\pi}(Y_{1}))_{\textrm{red}}=H_{1}$. Finally,
		$f_{1}$ is the blowup of $X\times_{V} V_{1}$ along $Y_{1}$.
	\end{enumerate}
\end{proposition}
\begin{proof}
	(1) Let $T$ denote a variable. In the following chain of extensions of
	sheaves over $V$:
	\begin{equation*}
		\mathcal{O}_{V}[(\mathcal{I}(Z)) T]\subset\mathcal{O}_{X}[(
		\mathcal{I}(Z)) T]= \mathcal{O}_{X}[(\mathcal{I}(Z)\cdot
		\mathcal{O}_{X})T]\subset \mathcal{O}_{X}[(
		\mathcal{I}(Y))T]
	\end{equation*}
	the first one is finite since $\mathcal{O}_{V}\subset\mathcal
	{O}_{X}$ is
	finite, and the last one is also finite by
	\ref{integral_closure_of_ideals}(a). Indeed, the $\mathcal{O}_{X}$-ideal
	$\mathcal{I}(Y)$ is integral over $\mathcal{I}(Z)\cdot\mathcal
	{O}_{X}$ by
	Proposition \ref{0310}. It follows that the extension of sheaves of graded
	rings
	$\mathcal{O}_{V}[(\mathcal{I}(Z)) T]\subset\mathcal
	{O}_{X}[(\mathcal{I}(Y))T]$
	over $V$ is finite, and hence this extension induces a morphism
	$X_{1}:=\operatorname{Proj}(\mathcal{O}_{X}[(\mathcal{I}(Y))T])
	\xrightarrow{\delta_{1}} V_{1}:=\operatorname{Proj}(\mathcal{O}_{V}[(
	\mathcal{I}(Z)) T])$, which is compatible with the blowups
	$X\xleftarrow{f} X_{1}$ and $V\xleftarrow{\pi} V_{1}$. Note that
	$V_{1}$ is irreducible, $F$-finite, and regular, and by construction
	$\delta_{1}$ belongs to $\mathscr{C}_{q}(V_{1})$. The uniqueness of
	$\delta_{1}$ making the diagram commutative follows from the universal
	property of the blow-up $V_{1}\to V$. Finally,
	$\mathcal{I}(Y)\cdot\mathcal{O}_{X_{1}}$ is invertible and integral over
	$(\mathcal{I}(Z)\cdot\mathcal{O}_{X})\cdot\mathcal
	{O}_{X_{1}}=\mathcal{I}(H_{1})
	\cdot\mathcal{O}_{X_{1}}$ by \ref{integral_closure_of_ideals}(b), whence
	$\mathcal{I}(Y)\cdot\mathcal{O}_{X_{1}}=\mathcal{I}(H_{1})\cdot
	\mathcal{O}_{X_{1}}$;
	see \ref{integral_closure_of_ideals}(d).
	
	(2) Note that $X\times_{V}V_{1}$ is the blowup of $X$ along
	$\mathcal{I}(Z)\mathcal{O}_{X}$, and in particular it is integral. It clearly
	follows that $X\times_{V} V_{1}\to V_{1}$ belongs to
	$\mathscr{C}_{q}(V_{1})$ and has generic rank $d$. Proposition \ref{B_pi_is_the_total_transform_of_B} tells us that $\delta_{\pi}$ is
	defined by the $\mathcal{O}_{V_{1}}^{q}$-submodule
	$\mathscr{M}\mathcal{O}_{V_{1}}^{q}\subset\mathcal{O}_{V}$. Next,
	$\pi_{\delta}^{-1}(Y)\subset X\times_{V} V_{1}$ is naturally isomorphic
	to $Y\times_{V} V_{1}\cong Y\times_{Z} H_{1}$, which in turn is isomorphic
	to $H_{1}$ via the projection $\delta_{\pi}$ since $Y\cong Z$ (recall that
	we assume the equivalent conditions of Proposition \ref{0310}). By Proposition \ref{Well_definition_of_the_transforms}, $H_{1}\subset V_{1}$ is permissible
	for $(\mathscr{M}\mathcal{O}_{V_{1}}^{q},1)$. Therefore the
	equivalent conditions
	of Proposition \ref{0310} hold for $\delta_{\pi}$ and
	$Y_{1}\subset X\times_{V} V_{1}$.
	
	We finally check that $f_{1}$ is the blowup of $X\times_{V} V_{1}$ along
	$Y_{1}$. Let $f_{1}':X_{1}'\to X\times_{V}V_{1}$ be the blowup along $Y_{1}$.
	We aim to prove that $X_{1}$ and $X_{1}'$ are $(X\times_{V} V_{1})$-isomorphic.
	This will be done by proving that there exist a unique
	$(X\times_{V} V_{1})$-morphism $h:X_{1}\to X_{1}'$ and a unique
	$(X\times_{V} V_{1})$-morphism $u:X_{1}'\to X_{1}$. On the one hand,
	$(\mathcal{I}(Y)\cdot\mathcal{O}_{X\times_{V} V_{1}})\cdot\mathcal
	{O}_{X_{1}}=
	\mathcal{I}(Y)\cdot\mathcal{O}_{X_{1}}$ is an invertible
	$\mathcal{O}_{X_{1}}$-ideal since $f=\pi_{\delta}f_{1}$ is the
	blowup along $Y$;
	thus by the universal property of $f_{1}'$ there exists a unique
	$(X\times_{V} V_{1})$-morphism $X_{1}\xrightarrow{h} X_{1}'$. On the other
	hand,
	$\mathcal{I}(Y)\cdot\mathcal{O}_{X_{1}'}=(\mathcal{I}(Y)\cdot
	\mathcal{O}_{X
		\times_{V} V_{1}})\cdot\mathcal{O}_{X_{1}'}$ is invertible since
	$f_{1}'$ is the blowup along $Y_{1}$; thus by the universal property of
	$f$ there exists a unique $X$-morphism $X_{1}'\xrightarrow{u} X_{1}$. It
	is only left to prove that $u$ is an $(X\times_{V} V_{1})$-morphism, that
	is, $f_{1} u=f_{1}'$. Since $u$ is an $X$-morphism, we have that
	$\pi\delta_{\pi}f_{1} u=\delta\pi_{\delta}f_{1} u=\delta\pi
	_{\delta}f_{1}'=\pi\delta_{\pi}f_{1}'$. Looking at the extremes of this
	equality and using the universal property of $\pi$ and the fact that
	$X_{1}'$ is integral, we conclude that
	$\delta_{\pi}f_{1} u=\delta_{\pi}f_{1}'$, that is, $u$ is also a
	$V_{1}$-morphism. Therefore $u$ is an $(X\times_{V}V_{1})$-morphism, as
	was aimed to prove. This completes the proof of (2).
\end{proof}
\begin{proposition}%
	\label{blow-up_along_codimension_one_subscheme}
	Let the setting be as in Proposition \ref{0310}, and assume that the equivalent
	conditions given there are satisfied. Assume, in addition, that
	$H:=Z\subset V$ is one-codimensional and set
	$\mathscr{M}_{1}:=\mathscr{M}_{\mathcal{I}(H)}$. Then the $V$-morphism
	$X_{1}\to X$ defined by the inclusion
	$\mathcal{O}_{V}^{q}[\mathscr{M}]\subseteq\mathcal
	{O}_{V}^{q}[\mathscr{M}_{1}]$
	is the blowup of $X$ along $Y$.
\end{proposition}
\begin{proof}
	We first show that $X\xleftarrow{f} X_{1}$ factorizes as
	$X\xleftarrow{f'} X_{1}'\xleftarrow{u} X_{1}$, where
	$X\xleftarrow{f'}X_{1}'$ is the blowup of $X$ along $Y$, by showing that
	$\mathcal{I}(Y)\mathcal{O}_{X_{1}}=\mathcal{I}(H)\mathcal
	{O}_{X_{1}}$, an invertible
	$\mathcal{O}_{X_{1}}$-ideal. Since one inclusion is clear, this is equivalent
	to showing that
	$F^{e}(\mathcal{I}(Y))\cdot\mathcal{O}_{V}^{q}[\mathscr{M}_{1}]
	\subseteq F^{e}(\mathcal{I}(H))\cdot\mathcal{O}_{V}^{q}[\mathscr{M}_{1}]$,
	or simply
	$F^{e}(\mathcal{I}(Y))\subseteq F^{e}(\mathcal{I}(H))\cdot\mathcal
	{O}_{V}^{q}[
	\mathscr{M}_{1}]$.
	
	We now prove this inclusion of subsheaves of $\mathcal{O}_{V}$. By Proposition \ref{0310},
	$F^{e}(\mathcal{I}(Y))=\mathcal{O}_{V}^{q}[\mathscr{M}]\cap\mathcal
	{I}(H)^{q}$,
	and this is equal to
	$\mathcal{O}_{V}^{q}[F^{e}(\mathcal{I}(H))\cdot\mathscr{M}_{1}]\cap
	\mathcal{I}(H)^{q}$ by Lemma \ref{Well_definition_of_the_weak_equivalence_class_of_M_L} when applied
	to $\mathcal{L}=\mathcal{I}(H)$. Finally, this intersection is
	included in
	$\mathcal{O}_{V}^{q}\cap\mathcal{I}(H)^{q}+F^{e}(\mathcal
	{I}(H))\cdot
	\mathcal{O}_{V}^{q}[\mathscr{M}_{1}]=F^{e}(\mathcal
	{I}(H))+F^{e}(\mathcal{I}(H))
	\cdot\mathcal{O}_{V}^{q}[\mathscr{M}_{1}]=F^{e}(\mathcal
	{I}(H))\cdot
	\mathcal{O}_{V}^{q}[\mathscr{M}_{1}]$.
	
	We now show that $u$ is an isomorphism. By Proposition \ref{blow-ups_of_radicial_morphisms}(1)
	$\delta f'\in\mathscr{C}_{q}(V)$ and
	$\mathcal{I}(Y)\cdot\mathcal{O}_{X_{1}'}=\mathcal{I}(H)\cdot
	\mathcal{O}_{X_{1}'}$
	(since the blowup of $V$ along $H$ is an isomorphism). Since $u$ is a morphism
	of $V$-scheme, it induces an inclusion
	$\mathcal{O}_{X_{1}'}^{q}\subseteq\mathcal{O}_{V}^{q}[\mathscr{M}_{1}]$,
	and it is only left to show the reverse inclusion.
	
	Now from
	$\mathcal{I}(Y)\cdot\mathcal{O}_{X_{1}'}=\mathcal{I}(H)\cdot
	\mathcal{O}_{X_{1}'}$
	we obtain
	$F^{e}(\mathcal{I}(Y))\mathcal{O}_{X_{1}'}^{q}=F^{e}(\mathcal{I}(H))
	\cdot\mathcal{O}_{X_{1}'}$. On the other hand,
	$F^{e}(\mathcal{I}(Y))=\mathcal{O}_{V}^{q}[\mathscr{M}]\cap\mathcal
	{I}(H)^{q}$,
	and this clearly includes $F^{e}(\mathcal{I}(H))\cdot\mathscr
	{M}_{1}$ by
	the definition of $\mathscr{M}_{1}$. We deduce that
	$F^{e}(\mathcal{I}(H))\cdot\mathscr{M}_{1}\subseteq F^{e}(\mathcal{I}(H))
	\cdot\mathcal{O}_{X_{1}'}^{q}$, whence
	$\mathscr{M}_{1}\subseteq\mathcal{O}_{X_{1}'}^{q}$. Thus
	$\mathcal{O}_{V}^{q}[\mathscr{M}_{1}]\subseteq\mathcal{O}_{X_{1}'}^{q}$,
	as was to be proved.
\end{proof}

\begin{corollary}%
	\label{a-transform_as_sequence_of_blow-ups}
	Let the setting be as in Proposition \ref{0310}, and assume that the equivalent
	conditions given there are satisfied. Assume in addition that $d>1$
	(equivalently,
	$\mathscr{M}\nsubseteq\mathcal{O}_{V}^{q}$). Let $f:X_{1}\to X$ and
	$\pi:V_{1}\to V$ be the blowups of $X$ along $Y$ and of $V$ along
	$Z$, respectively, and let $\delta_{1}:X_{1}\to V_{1}$ be the morphism
	in $\mathscr{C}_{q}(V_{1})$ given in Proposition \ref{blow-ups_of_radicial_morphisms}. Let $H_{1}\subset V_{1}$ be the
	exceptional
	divisor of $\pi$. Then the following holds.
	\begin{enumerate}[(2)]
		\item[(1)] $\delta_{1}:X_{1}\to V_{1}$ is defined by the 1-transform
		$\mathscr{M}_{1}^{(1)}\subseteq\mathcal{O}_{V_{1}}$ of $\mathscr
		{M}$ under
		$\pi$ (Definition \ref{the_a-transform_of_an_O^q-module}).
		\item[(2)] More generally, let $a$ denote the largest integer such that
		$H_{1}\subset V_{1}$ is permissible for
		$(\mathscr{M}\mathcal{O}_{V_{1}}^{q},a)$, and for each $1\leq i\leq
		a$, we
		consider the $i$-transform
		$\mathscr{M}_{1}^{(i)}\subseteq\mathcal{O}_{V_{1}}$ of $\mathscr
		{M}$ under
		$\pi$ and its associated morphism $\delta_{i}:X_{i}\to V_{1}$ in
		$\mathscr{C}_{q}(V_{1})$. Then $a\geq1$, and there is a commutative diagram
		\begin{align*}
			& \xymatrix@C=57pt {X\ar[d]^{\delta}& X_{0}
				\ar[l]_{\pi_\delta}\ar[d]^{\delta_{0}}& X_{1}
				\ar[l]_{f_{1}}\ar[d]^{\delta_{1}}& \cdots \ar[l]_{f_{2}}&X_{a}
				\ar[l]_{f_{a}}\ar[d]^{\delta_{a}}%
				\\
				V& V_{1}\ar[l]_\pi& V_{1}\ar[l]_{=}& \cdots \ar[l]_{=} & V_{a}
				\ar[l]_{=}}
			\\
			& \xymatrix@C=49pt { \mathscr{M}&\mathscr{M}\mathcal
				{O}_{V_{1}}^{q}&\mathscr{M}_{1}^{(1)}&
				& \mathscr{M}_{1}^{(a)}}
		\end{align*}
		where the first square is Cartesian (so that $\delta_{0}$ is defined by
		the $\mathcal{O}_{V_{1}}^{q}$-module
		$\mathscr{M}\mathcal{O}_{V_{1}}\subseteq\mathcal{O}_{V_{1}}$ by Proposition \ref{B_pi_is_the_total_transform_of_B}), and $f_{i}$ is the blowup of
		$X_{i-1}$ along
		$(\delta_{i-1}^{-1}(H_{1}))_{\textrm{red}}\subset X_{i-1}$. In addition,
		$H_{1}$ is permissible for $(\mathscr{M}^{(i)},1)$, $i =1,\ldots,a-1$, and
		not for $(\mathscr{M}_{1}^{(a)},1)$.
	\end{enumerate}
\end{corollary}
\begin{proof}
	By Proposition \ref{blow-ups_of_radicial_morphisms}, $\delta_{1}$ can
	be expressed as $\delta_{0} f_{1}$, where $f_{1}$ is the blowup along
	$(\delta_{0}^{-1}(H_{1}))_{\textrm{red}}\subset X_{0}$. Applying Proposition \ref{blow-up_along_codimension_one_subscheme} to $\delta_{0}$ and
	$H_{1}$, we see that $\delta_{1}$ is defined by
	$(\mathscr{M}\mathcal{O}_{V_{1}}^{q})_{\mathcal{I}(H_{1})}$, which
	is the
	definition of $\mathscr{M}_{1}^{(1)}$. This proves (1).
	
	The proof of (2) follows by iteration of (1) or of Proposition \ref{blow-up_along_codimension_one_subscheme}. We only need to check that
	for $1\leq i<a$,
	\begin{enumerate}[(c)]
		\item[(a)] $H_{1}$ is permissible for $(\mathscr{M}_{1}^{(i)},1)$,
		\item[(b)] $\mathscr{M}_{1}^{(i+1)}=(\mathscr
		{M}_{1}^{(i)})_{\mathcal{I}(H_{1})}$, and
		\item[(c)] $H_{1}$ is not permissible for $(\mathscr{M}_{1}^{(a)},1)$.
	\end{enumerate}
	
	(b) follows from Lemma \ref{L^qM_L}(3) when applied to
	$\mathscr{M}\mathcal{O}_{V_{1}}^{q}\subseteq\mathcal{O}_{V_{1}}$,
	$\mathcal{L}=\mathcal{I}(H_{1})^{i}$, and
	$\mathcal{L}'=\mathcal{I}(H_{1})$. By the definition of $a$ and Lemma \ref{Well_definition_of_the_weak_equivalence_class_of_M_L}(1),
	$\mathscr{M}{\mathcal{O}_{V_{1}}^{q}}\subseteq\mathcal{O}_{V}^{q}+F^{e}(
	\mathcal{I}(H_{1})^{a})\mathscr{M}_{1}^{(a)}$, but the same does not hold
	with $a$ replaced by $a+1$. This implies (c), and also that for
	$i<a$,
	$\mathscr{M}_{1}^{(i)}=(\mathcal{O}_{V_{1}}^{q}+\mathscr{M}\mathcal
	{O}_{V_{1}}^{q}:F^{e}(
	\mathcal{I}(H_{1})^{i}))=\mathcal{O}_{V_{1}}^{q}+F^{e}(\mathcal
	{I}(H_{1})^{a-i})
	\mathscr{M}_{1}^{(a)}$. Note that (a) follows from this.
\end{proof}
\begin{remark}
	Proposition \ref{the_invariant_a_in_the_blow-up} shows that the number
	$a$ in the above corollary coincides with the largest integer such that
	$Z$ is permissible for $(\mathscr{M},a)$.
\end{remark}

\begin{remark}%
	\label{proof_of_the_first_version_of_main_theorem2}
	\emph{On the proof of Theorem \ref{finite_morphisms_and_closed_subsets} (for
		$\delta\in\mathscr{C}_{q}(V)$) and Theorem \ref{main_theorem_2}:}
	Proposition \ref{0309} covers the proof of the first part and (1) of Theorem \ref{finite_morphisms_and_closed_subsets}, except for the closeness of
	$F_{d}(X)$, whose proof will be postponed until Proposition \ref{d(F(X))_is_closed}. It also covers the proof of the first equality
	in (1) of Theorem \ref{main_theorem_2}, the second one being also included
	in Proposition \ref{d(F(X))_is_closed}. Propositions \ref{0310} and 
	\ref{blow-ups_of_radicial_morphisms} cover the proof of (2) of Theorem \ref{finite_morphisms_and_closed_subsets} and of (2) of Theorem \ref{main_theorem_2}. Finally, Corollary \ref{a-transform_as_sequence_of_blow-ups}(1) covers the proof of (3) of
	Theorem \ref{main_theorem_2}.
\end{remark}

\section{Differential Operators and the Definition of Local Invariants
	for Points $x\in F_{d}(X)$}
\label{Section_on_Diff}

Let $V$ denote an irreducible $F$-finite regular scheme, let
$\mathscr{M}\subset\mathcal{O}_{V}$ be an $\mathcal
{O}_{V}^{q}$-module, and
let $\delta:X\to V$ be the finite morphism attached to $\mathscr{M}$. We
set $d:=[K(X):K(V)]$. One of the main purposes of this paper is defining
invariants of singularities for points $x\in F_{d}(X)$. Proposition \ref{0309} provides the description
$\delta(F_{d}(X))=\{x\in V:\mathscr{M}_{x}\subseteq\mathcal
{O}_{V,x}^{q}+m_{V,x}^{q}
\}$. Our definition of invariants will arise of course in studying the
inclusion $\mathscr{M}_{x}\subseteq\mathcal
{O}_{V,x}^{q}+m_{V,x}^{q}$ and
similar conditions. Differential operators will play a relevant role in
this vein.

We formulate our discussion in this section at the local level. We fix
an $F$-finite regular local ring $(R,\m)$ and an $R^{q}$-submodule
$M\subset R$. We shall assign to $M$ two different numerical invariants.
Both invariants are compatible with our notions of equivalence of Definition \ref{defequi}. The first one, denoted by $\nu_\m^{(q)}(M)$, is the
highest integer $n$ such that
$M\subset R^{q}+\m^{n} (\subset R)$. This will be, to some extend,
an analog to the notion of order of an ideal but applied here to
$R^{q}$-submodules of $R$. The second invariant, denoted by
$\eta_\m(M)$, will be defined in terms of differential operators;
see \ref{Definition_of_eta}.

We will make use of $p$-basis to show that these invariants are closely
related. For instance, $\nu_\m^{(q)}(M)=\eta_\m(M)$ if
$R/\m$ is perfect. A~$p$-basis of $R$ provides, for each element
$f\in R$ and each power $q=p^{e}$, a natural notion of $q$-expansion, which,
in some way to be clarified, can be interpreted as a Taylor expansion of~$f$. This approach will lead us, for instance, to a simple proof of a classical
result expressing the order of an element at a local regular ring in terms
of differential operators (see Corollary \ref{local_Jacobian_criterion}).

Given $V$ and an $\mathcal{O}_{V}^{q}$-module
$\mathscr{M}\subset\mathcal{O}_{V}$ as above, we get for each point
$x\in V$ an $F$-finite regular local ring $\mathcal{O}_{V,x}$ and an
$\mathcal{O}_{V,x}^{q}$-module
$\mathscr{M}_{x}\subset\mathcal{O}_{V,x}$. Hence we obtain two functions
$x\mapsto\nu_{x}^{(q)}(\mathscr{M}):=\nu_{m_{V,x}}^{(q)}(\mathscr{M}_{x})$
and
$x\mapsto\eta_{x}(\mathscr{M}):=\eta_{m_{V,x}}(\mathscr{M}_{x})$.

Both functions appear as an attempt to extend to modules properties known
for ideals. If we fix an ideal $\mathcal{J}$ over $V$, the order function:
$x\mapsto\nu_{x}(\mathcal{J})$ is upper-semicontinuous. However, of the
two functions defined by $\mathcal{O}_{V}^{q}$-modules, only the second
one has this property. In fact, we introduce the function
$x\mapsto\eta_{x}(\mathscr{M})$ to remedy the fact that
$\nu_{x}^{(q)}(\mathscr{M})$ is not upper-semicontinuous. We will discuss
this in the next section (see \ref{par416}).

\subsection{Differential Operators}
\label{Diff_operators}

We summarize some well-known facts about differential operators. We refer
the reader to \cite[Chapter~3]{giraud1972etude} and \cite
[16.8]{EGAIV} for
details of the definitions and facts that will be stated here without proofs.
\begin{parrafo}%
	\label{Differential_operators}
	Given a ring $R$, we denote
	$T_{R}:=\langle1\otimes r-r\otimes1:r\in R\rangle\subset R\otimes R$.
	For an integer $i\geq0$, we denote by $\operatorname{Diff}_{R}^{i}$
	the (left)
	$R$-submodule of $\operatorname{End}_{\mathbb{Z}}(R,+)$ consisting of those
	$\mathbb{Z}$-linear maps $D:R\to R$ whose (left) $R$-linear extension
	$\varphi_{D}:R\otimes R\to R$ ($\varphi_{D}(r_{1}\otimes r_{2})=r_{1}
	D(r_{2})$ annihilates the ideal $T_{R}^{i+1}\subset R\otimes R$ and hence
	factors through the quotient
	$\operatorname{P}_{R}^{i}:=(R\otimes R)/{T_{R}^{i+1}}$. This quotient
	is called
	the $R$-module of principal parts of order $i$.
	
	There is a chain of $R$-modules
	$\operatorname{Diff}_{R}^{0}\subseteq\operatorname
	{Diff}_{R}^{1}\subseteq
	\operatorname{Diff}_{R}^{2}\subseteq\cdots$. An element $D:R\to R$
	in any of
	these modules is called an \emph{absolute differential operator} or simply
	a \emph{differential operator} of the ring $R$. The order of $D$ is the
	first integer $i$ such that $D\in\operatorname{Diff}_{R}^{i}$, and it
	will be
	denoted $|D|$. So $\operatorname{Diff}_{R}^{i}$ consists of the differential
	operators of order at most $i$.
	
	We denote by $\operatorname{Diff}_{R,+}^{i}$ the $R$-submodule of
	$\operatorname{Diff}_{R}^{i}$ consisting of the operators $D$ such that
	$D(1)=0$. For $i=0$, we have
	$\operatorname{Diff}_{R}^{0}=\operatorname{End}_{R}(R,+)\cong R$.
	This identification
	leads to a decomposition
	\begin{align}
		\label{Decomposition_of_Diff=R+Diff_+} \operatorname{Diff}_{R}^{i}=R\oplus
		\operatorname{Diff}_{R,+}^{i}.
	\end{align}
	
	The composite of two differential operators is again a differential operator.
	More precisely, if $D\in\operatorname{Diff}_{R}^{i}$ and
	$D'\in\operatorname{Diff}_{R}^{j}$, then
	$D\circ D'\in\operatorname{Diff}_{R}^{i+j}$. In addition, if
	$D'\in\operatorname{Diff}_{R,+}^{j}$, then
	$D\circ D'\in\operatorname{Diff}_{R,+}^{i+j}$.
	
	Given an $R$-submodule $\mathscr{D}\subset\operatorname
	{Diff}_{R}^{i}$ and a
	nonempty subset $\Phi\subset R$, we denote by $\mathscr{D}(\Phi)$ the
	abelian subgroup of $R$ generated by the evaluations $D(r)$ with
	$D\in\mathscr{D}$ and $r\in\Phi$, which is clearly an ideal of $R$. Note
	that for two nonempty subsets $\Phi,\Phi'\subseteq R$, we have
	\begin{align}
		\label{Diff(A+B)=Diff(A)+Diff(B)} \mathscr{D}(\Phi\cup\Phi')=\mathscr{D}(\Phi)+
		\mathscr{D}(\Phi').
	\end{align}
	
	Given an ideal $I\subset R$ and $D\in\operatorname{Diff}_{R}^{i}$, we
	will often
	use the following property:
	\begin{align}
		\label{evaluating_differential_operators_at_powers_of_ideals} D(I^{n})\subseteq I^{n-i}
		\quad \text{for all integers }i: 0\leq i\leq n.
	\end{align}
\end{parrafo}
\begin{parrafo}%
	\label{Differential_operators_of_order_<q_are_R^q-linear}
	If the ring $R$ has characteristic $p>0$, and if $q$ denotes a power of
	$p$, then any $D\in\operatorname{Diff}_{R}^{q-1}$ is $R^{q}$-linear.
	In fact,
	if $\varphi:R\otimes R\to R$ denotes the morphism of left $R$-modules
	extending $D$, then by definition $\varphi(T_{R}^{q})=0$. Therefore
	$D(x^{q} r)=\varphi(1\otimes x^{q} r)= \varphi(x^{q} \otimes r)+
	\varphi(1\otimes x^{q}r-x^{q}\otimes r)=x^{q}\varphi(1\otimes r)+
	\varphi((1\otimes x^{q}-x^{q}\otimes1)(1\otimes r))= x^{q}D(r)+
	\varphi((1\otimes x-x\otimes1)^{q}(1\otimes r))=x^{q} D(r)$. In particular,
	\begin{align}
		\label{Diff(R^q)=0} \operatorname{Diff}_{R,+}^{q-1}(R^{q})=0.
	\end{align}
\end{parrafo}
\begin{parrafo}%
	\label{logarithmic_differential_operators}%
	\emph{Logarithmic Differential Operators.} Given an ideal
	$I\subset R$, a differential operator $D:R\to R$ is said to be $I$-logarithmic
	if $D(I^{k})\subseteq I^{k}$ for all $k\geq0$. We denote by
	$\operatorname{Diff}_{R,I}^{i}$ the $R$-submodule of
	$\operatorname{Diff}_{R}^{i}$ consisting of those differential
	operators that
	are $I$-logarithmic. We also denote
	$\operatorname{Diff}_{R,I,+}^{i}:=\operatorname{Diff}_{R,I}^{i}\cap
	\operatorname{Diff}_{R,+}^{i}$.
	According to (\ref{evaluating_differential_operators_at_powers_of_ideals}),
	there are inclusions
	\begin{align}
		\label{IDiff_is_logarithmic} I^{i}\operatorname{Diff}_{R}^{i}
		\subset\operatorname{Diff}_{R,I}^{i} \quad \text{and}\quad
		I^{i}\operatorname{Diff}_{R,+}^{i}\subset
		\operatorname {Diff}_{R,I,+}^{i}.
	\end{align}
	More generally, if $\Lambda=\{I_{1},\ldots,I_{r}\}$ is a finite collection
	of ideals of $R$, then we denote
	$\operatorname{Diff}_{R,\Lambda}^{i}:=\operatorname
	{Diff}_{R,I_{1}}^{i}\mathrel{\cap}
	\cdots\cap\operatorname{Diff}_{R,I_{r}}^{i}$. The elements of this $R$-module
	are called $\Lambda$-logarithmic differential operators. We also define
	$\operatorname{Diff}_{R,\Lambda,+}^{i}:=\operatorname
	{Diff}_{R,\Lambda}^{i}\cap
	\operatorname{Diff}_{R,+}^{i}$.
	
	If $D\in\operatorname{Diff}_{R,\Lambda}^{i}$ and
	$D'\in\operatorname{Diff}_{R,\Lambda}^{j}$, then, as we already mentioned,
	$D\circ D'$ is a differential operator of order $\leq i+j$, and clearly
	$D\circ D'\in\operatorname{Diff}_{R,\Lambda}^{i+j}$.
	
	As a consequence of (\ref{IDiff_is_logarithmic}), if $L\subset R$ is an
	ideal included in $I_{1}\mathrel{\cap}\cdots\cap I_{r}$, then
	\begin{align}
		\label{LDiff_is_Lambda_logarithmic} L^{i}\operatorname{Diff}_{R}^{i}
		\subseteq\operatorname {Diff}_{R,\Lambda}^{i} \quad \text{and}\quad
		L^{i}\operatorname {Diff}_{R,+}^{i}
		\subseteq\operatorname{Diff}_{R,\Lambda,+}^{i} .
	\end{align}
	
	Note that $\operatorname{Diff}_{R,\Lambda}^{i}=\operatorname
	{Diff}_{R}^{i}$ if
	$\Lambda$ is either the empty family or if $\Lambda=\{R\}$, so everything
	we say in the logarithmic setting applies also to the original case.
\end{parrafo}
\begin{parrafo}%
	\label{compatibility_with_localization}%
	\emph{Compatibility with Localizations.} If $S$ is a multiplicative subset
	of $R$ such that $R\to S^{-1}R$ is injective, then a differential operator
	$D:R\rightarrow R$ can be extended uniquely to a differential operator
	$D_{S}:S^{-1}R\rightarrow S^{-1}R$ (of the same order as $D$). This follows
	easily from the good localization property of the module of principal parts,
	namely, $S^{-1}\operatorname{P}_{R}^{i}=\operatorname
	{P}_{S^{-1}R}^{i}$. In the case
	that $R$ is a ring of characteristic $p>0$ (as we will always assume in
	these notes), there is a handy description of the action of $D_{S}$ on
	fractions. Let $q$ be a power of $p$ strictly greater than the order of
	$D$. By \ref{Differential_operators_of_order_<q_are_R^q-linear}, $D$ is
	$R^{q}$-linear, and similarly $D_{S}$ is $(S^{-1}R)^{q}$-linear. Therefore,
	given $a/f\in S^{-1}R$ with $f\in S$, we have
	\begin{align}
		\label{handy_description_of_D_S} D_{S}(a/f)=D_{S}(af^{q-1}/f^{q})=(1/f^{q})D_{S}(af^{q-1}/1)=D(af^{q-1})/f^{q}.
	\end{align}
	
	If $\Lambda$ is a finite collection of ideals of $R$, then we denote
	$S^{-1}\Lambda=\{S^{-1}I:I\in\Lambda\}$. The above description of
	$D_{S}$ shows easily that if $D$ is $\Lambda$-logarithmic, then
	$D_{S}$ is $S^{-1}\Lambda$-logarithmic.
	
	The assignment $D\mapsto D_{S}$ extends naturally to a morphism of
	$S^{-1}R$-modules
	\begin{align}
		\label{localization_of_Diff} S^{-1}(\operatorname{Diff}_{R,\Lambda}^{i})
		\rightarrow\operatorname {Diff}_{S^{-1}R,S^{-1}
			\Lambda}^{i}
	\end{align}
	When $R$ is a domain, then this map is clearly injective. We now show that
	if $R$ is an $F$-finite Noetherian domain, then the map (\ref
	{localization_of_Diff})
	is bijective, giving a natural isomorphism
	\begin{align}
		\label{S^-1Diff(R)=Diff(S^-1R)} S^{-1}(\operatorname{Diff}_{R,\Lambda}^{i})
		\cong\operatorname {Diff}_{S^{-1}R,S^{-1}
			\Lambda}^{i}.
	\end{align}
	In fact, it is only left to prove that given
	$\tilde{D}\in\operatorname{Diff}_{S^{-1}R,S^{-1}\Lambda}^{i}$,
	there exist
	$D\in\operatorname{Diff}_{R,\Lambda}^{i}$ and $f\in S$ such that
	$\tilde{D}=\frac{1}{f}D_{S}$. Let $q$ be a power of $p$ such that
	$q>i$. Then the restriction of $\tilde{D}$ to $R$ is $R^{q}$-linear, and
	we have $\tilde{D}(I^{k})\subset S^{-1}I^{k}$ for all $k\geq0$ and
	$I\in\Lambda$. Now, since $R$ is an $F$-finite Noetherian domain,
	$R$ and each of its ideals are finite $R^{q}$-modules. This observation
	and the fact that $\Lambda$ is finite shows that there is an element
	$f\in S$ such that $ f\tilde{D}(R)\subseteq R$, which implies that
	$f\tilde{D}\in\operatorname{Diff}_{R}^{i}$, and such that
	$f\tilde{D}(I^{k})\subseteq I^{k}$ for all $I\in\Lambda$ and
	$k\leq i$. Now it is easy to show that this already implies that
	$f\tilde{D}(I^{k})\subseteq I^{k}$ for all $k\geq0$. Therefore
	$f\tilde{D}\in\operatorname{Diff}_{R,\Lambda}^{i}$. Note finally that
	$\tilde{D}=\frac{1}{f}D_{S}$ for $D:=f\tilde{D}$. This completes the proof
	of our isomorphism.
	
	An analogous discussions also holds for the module
	$\operatorname{Diff}^{i}_{R,\Lambda,+}$.
	
	Finally, let $R$ be an $F$-finite Noetherian domain, let $\Lambda$ be
	a finite collection of ideals, and let $q=p^{e}$ be a power of $p$. Let
	$M\subset R$ be an $R^{q}$-submodule, and let $J\subset R$ be an ideal.
	We deduce from (\ref{S^-1Diff(R)=Diff(S^-1R)}) and from the explicit
	description
	of this isomorphism given in (\ref{handy_description_of_D_S}) that the
	following equalities hold:
	\begin{align}
		\label{localization} S^{-1}(\operatorname{Diff}_{R,\Lambda}^{i}(J))={}&
		\operatorname {Diff}_{S^{-1}R,S^{-1}
			\Lambda}^{i}(J)\nonumber
		\\
		={}&\operatorname{Diff}_{S^{-1} R,S^{-1}\Lambda}^{i}(S^{-1}J),\quad
		\forall i\geq0,\nonumber
		\\
		S^{-1}(\operatorname{Diff}_{R,\Lambda,+}^{i}(M))={}&
		\operatorname {Diff}_{S^{-1}R,S^{-1}
			\Lambda,+}^{i}(M)
		\\
		={}&\operatorname{Diff}_{S^{-1}R,S^{-1}\Lambda
			,+}^{i}((F^{e}(S))^{-1}M),\nonumber\\
		&{}
		\forall i=1,\ldots,q-1.\nonumber
	\end{align}
\end{parrafo}

\subsection{$p$-Basis and Taylor Operators}
\label{sec3.2}

We recall the general notion of $p$-basis and the associated Taylor operators
for $F$-finite regular local rings. Later we will focus only on those
$p$-basis that extend a regular system of parameters. These are the
$p$-basis that we will use to study the invariants that will be assigned
to $R^{q}$-modules.
\begin{parrafo}%
	\label{p-basis}%
	\emph{$p$-Basis.} Let $R$ be an $F$-finite domain of characteristic
	$p$. An ordered set $\{z_{1},\ldots,z_{n}\}\subset R$ such that the monomials
	$\z^{\alpha}:=z_{1}^{\alpha_{1}}\ldots z_{n}^{\alpha_{n}}$, with
	$0\leq\alpha_{i}<p$, form a basis for $R$ over $R^{p}$ is called a
	\emph{$p$-basis}. Note that if $q$ is a power of $p$ and
	\begin{equation*}
		\mathcal{A}_{q}:=\{\alpha\in\mb{N}_{0}^{n}:
		0\leq\alpha_{i}<q \},
	\end{equation*}
	then $R$ is an $R^{q}$-free module with basis
	$\{\z^{\alpha}: \alpha\in\mathcal{A}_{q}\}$. Therefore any
	$f\in R$ has a unique expression of the form
	\begin{equation*}
		f=\sum_{\mathcal{A}_{q}}c_{\alpha}^{q}
		\z^{\alpha} \quad \text{with
		} c_{\alpha}=c_{\alpha}(f)
		\in R.
	\end{equation*}
	This expression will be referred to as the \emph{$q$-expansion of $f$ with
		respect to the $p$-basis $(z_{1},\ldots,z_{n})$} or simply as \emph{the
		$q$-expansion of $f$} if there is no risk of confusion.
	
	A well-known result by Kunz states that if a local ring admits a $p$-basis,
	then it must be regular. Conversely, Proposition \ref{special_p-basis} recalls the fact that $F$-finite regular local rings
	admit $p$-basis.
\end{parrafo}
\begin{parrafo}%
	\emph{Taylor Differential Operators.} A $p$-basis
	$\{z_{1},\ldots,z_{n}\}$ is also a differential basis for $R$, that is,
	$\Omega_{R}^{1}$ is free with basis $\{dz_{1},\ldots,dz_{n}\}$. In fact,
	in the context of $F$-finite rings, these two concepts are the same; see
	\cite{kimuraniitsuma1984}. Thus $R$ is differentially smooth over the prime
	field in the sense of Grothendieck \cite[16.10]{EGAIV}, and for each
	$\gamma\in\mb{N}_{0}^{n}$, there exists a unique differential operator
	$D_{\gamma}:R\to R$ such that
	\begin{align}
		\label{evaluation_of_Taylor_operators_in_monomials} D_{\gamma}(\z^{\alpha})=
		\binom{\alpha}{\gamma}\z^{\alpha-
			\gamma},\quad  \forall\alpha\in
		\mb{N}_{0}^{n}.
	\end{align}
	(Here $\binom{\alpha}{\gamma}$ denotes
	$\prod_{k=1}^{n}\binom{\alpha_{k}}{\gamma_{k}}$.) In addition,
	$D_{\gamma}$ has order $|\gamma|$, and for each $i\geq0$,
	$\operatorname{Diff}_{R}^{i}$ is a free $R$-module with basis
	$\{D_{\gamma}:|\gamma|\leq i\}$, and $\operatorname{Diff}_{R,+}^{i}$
	is free
	with basis $\{D_{\gamma}:0<|\gamma|\leq i\}$ (observe that $D_{0}$ is the
	identity).
	
	The operators $D_{\gamma}$ are called Taylor operators (with respect to
	the given $p$-basis). Note that from (\ref
	{evaluation_of_Taylor_operators_in_monomials})
	we deduce that
	\begin{align}
		\label{composition_of_Taylor_operators} D_{\gamma}\circ D_{\gamma'}=
		\binom{\gamma+\gamma'}{\gamma}D_{
			\gamma+\gamma'}, \quad \gamma,\gamma'\in
		\mb{N}_{0}^{n}.
	\end{align}
	In particular, a Taylor operator decomposes as
	\begin{align}
		\label{decomposition_of_a_Taylor_operator} D_{\gamma}=D_{\gamma_{1}\e_{1}}\circ\cdots\circ
		D_{\gamma_{n}
			\e_{n}},
	\end{align}
	where $\e_{1},\ldots,\e_{n}\in\mb{N}_{0}^{n}$ are the canonical
	vectors.
	
	Since the characteristic is $p>0$, it is worth mentioning that a binomial
	coefficient of integers might be zero when viewed as an element of
	$R$. In fact, $\binom{a}{b}$ is zero in $R$ if and only if
	$\binom{a}{b}\equiv0 \mod p$. Lucas's theorem states that if
	$a=\sum a_{i}p^{i}$ and $b=\sum b_{i}p^{i}$ are the $p$-adic expansions
	of nonnegative integers $a$ and $b$, then
	$\binom{a}{b}\equiv\prod\binom{a_{i}}{b_{i}}\mod p$; in particular,
	\begin{align}
		\label{Lucas_theorem} \binom{a}{b}\not\equiv0 \mod p \quad \text{if and only if}\quad
		b_{i} \leq a_{i},\quad  \forall i.
	\end{align}
\end{parrafo}
\begin{parrafo}%
	\label{Taylor_operators_are_respectfull_with_the_p-basis}%
	\emph{Compatibility of the action of Taylor operators with the $q$-expansions.}
	Let $q$ be a power of $p$. There is a compatibility of the action of Taylor
	operators $D_{\gamma}$, when $\gamma\in\mathcal{A}_{q}$, with $q$-expansions
	of elements of $R$. To express this, note first that $D_{\gamma}$ is
	$R^{q}$-linear, for it is the composition of differential operators of
	order $\leq q-1$ by (\ref{decomposition_of_a_Taylor_operator}), and these
	are $R^{q}$-linear by
	\ref{Differential_operators_of_order_<q_are_R^q-linear}. Therefore, if
	$f=\sum_{\mathcal{A}_{q}} c_{\alpha}^{q}\z^{\alpha}$ is the
	$q$-expansion
	of an element $f\in R$, then we have
	\begin{equation*}
		D_{\gamma}(f)=\sum_{\mathcal{A}_{q}}
		c_{\alpha}^{q} \binom{\alpha}{\gamma}\z^{\alpha-\gamma}.
	\end{equation*}
	The important observation is that this is the $q$-expansion of
	$D_{\gamma}(f)$ with respect to the given $p$-basis. Indeed,
	$\binom{\alpha}{\gamma}=\binom{\alpha}{\gamma}^{q}$ in $R$, and if
	$\binom{\alpha}{\gamma}\z^{\alpha-\gamma}=
	\binom{\alpha'}{\gamma}\z^{\alpha'-\gamma}\neq 0$, then we have
	$\alpha=\alpha'$.
\end{parrafo}

The following proposition recalls the fact that $F$-finite regular local
rings admit $p$-basis.
\begin{proposition}%
	\label{special_p-basis}
	Let $(R,\m)$ be an $F$-finite regular local ring, and let
	$\{x_{1},\ldots,x_{r}\}$ be a regular system of parameters. Then
	$\{x_{1},\ldots,x_{r}\}$ can be extended to a $p$-basis for $R$. More
	precisely, given a subset $\{y_{1},\ldots,y_{s}\}\subset R$, the following
	conditions are equivalent.
	\begin{enumerate}[(2)]
		\item[(1)] The classes of $y_{1},\ldots,y_{s}$ modulo $\m$ form a
		$p$-basis for the residue field $R/\m$.
		\item[(2)] $\{x_{1},\ldots,x_{r},y_{1},\ldots,y_{s}\}$ is a
		$p$-basis for
		$R$.
	\end{enumerate}
\end{proposition}
\begin{proof}
	The existence of a $p$-basis for $(R,\m)$ follows from the main result
	of \cite{kimuraniitsuma1980}. The result of \cite{kimuraniitsuma1984} ensures
	that a $p$-basis for $R$ is the same thing as a differential basis. Thus
	the proposition follows from the exact sequence
	$0\rightarrow\m/\m^{2}\rightarrow\Omega_{R}^{1}\otimes R/
	\mvtex\rightarrow\Omega_{R/\mvtex}^{1}\rightarrow0$ proved in
	\cite[Satz 1, (b)]{berger1961struktur}.
\end{proof}
\begin{definition}
	An ordered $p$-basis $(x_{1},\ldots,x_{r},y_{1},\ldots,y_{s})$ for
	$(R,\mvtex)$ as in Proposition \ref{special_p-basis} is called an
	\emph{adapted
		$p$-basis}.
\end{definition}
We denote by $D_{\alpha,\beta}$,
$(\alpha,\beta)\in\mathbb{N}_{0}^{r+s}$, the Taylor operators associated
with an adapted $p$-basis
$(x_{1},\ldots,x_{r},y_{1},\ldots,y_{s})$. The set
$\mathcal{A}_{q}$ takes the form
\begin{equation*}
	\mathcal{A}_{q}=\{(\alpha,\beta)\in\mbvtex{N}_{0}^{r+s}:
	0\leq \alpha_{i}, \beta_{j}<q, 1\leq i\leq r, 1\leq j
	\leq s\}.
\end{equation*}
We also denote
$\mathcal{A}_{q}^{+}:=\mathcal{A}_{q}\setminus\{(0,0)\}$. Given
$f\in R$, its $q$-expansion has the form
\begin{equation*}
	f=\sum_{\mathcal{A}_{q}} c_{\alpha,\beta}^{q}
	\xvtex^{\alpha
	}\yvtex^{
		\beta}=c_{0,0}^{q}+
	\sum_{\mathcal{A}_{q}^{+}} c_{\alpha,\beta}^{q}
	\xvtex^{\alpha}\yvtex^{\beta}.
\end{equation*}

The first part of the following proposition expresses the order of an element
$f$ in $(R,m)$ in terms of the $q$-expansion with respect to an adapted
$p$-basis. The second part concerns the condition
$f\in R^{q}+\mvtex^{n}$, which will be particularly useful in the study
of $R^{q}$ submodules of $R$.
\begin{proposition}%
	\label{order_in_terms_of_q-atoms}
	Let $(R,\mvtex)$ be an $F$-finite regular local ring, and let
	$(x_{1},\ldots,x_{r},y_{1},\ldots,y_{s})$ be an adapted $p$-basis. Given
	$f\in R$ with $q$-expansion
	$f=\sum_{(\alpha,\beta)\in\mathcal{A}_{q}}c_{\alpha,\beta}^{q}
	\xvtex^{\alpha}\yvtex^{\beta}$, for any integer $n$, the following holds.
	\begin{enumerate}[(2)]
		\item[(1)] $f\in\mvtex^{n}$ if and only if
		$c_{\alpha,\beta}^{q}\xvtex^{\alpha}\yvtex^{\beta}\in\mvtex
		^{n}$ for
		all $(\alpha,\beta)\in\mathcal{A}_{q}$.
		\item[(2)] $f\in R^{q}+\mvtex^{n}$ if and only if
		$c_{\alpha,\beta}^{q}\xvtex^{\alpha}\yvtex^{\beta}\in\mvtex
		^{n}$ for
		all $(\alpha,\beta)\in\mathcal{A}_{q}^{+}$.
	\end{enumerate}
\end{proposition}
\begin{proof}
	(1) For each integer $t\geq0$ we denote by $\mathcal{M}_{t}$ the Abelian
	subgroup of $R$ generated by the elements of $\mvtex^{t}$ that are of
	the form $c^{q}\xvtex^{\alpha}\yvtex^{\beta}$ with $c\in R$ and
	$(\alpha,\beta)\in\mathcal{A}_{q}$. By definition we have that
	$\mathcal{M}_{t}\subseteq\mvtex^{t}$ for all $t\geq0$. We claim that
	these are all equalities. This will imply (1) by the uniqueness of the
	$q$-expansion. To show the claim, we first observe that
	$\mathcal{M}_{t}\mathcal{M}_{u}\subseteq\mathcal{M}_{t+u}\subseteq
	\mvtex^{t+u}$. In fact, given
	$c^{q}\xvtex^{\alpha}\yvtex^{\beta}\in\mvtex^{t}$ and
	$c^{\prime q}\xvtex^{\alpha'}\yvtex^{\beta'}\in\mvtex^{u}$, with
	$(\alpha,\beta), (\alpha',\beta')\in\mathcal{A}_{q}$, the product
	$(cc')^{q}\xvtex^{\alpha+\alpha'}\yvtex^{\beta+\beta'}$ belongs to
	$\mvtex^{t+u}$, and it can be clearly expressed in the form
	$c^{\prime\prime q}\xvtex^{\alpha''}\yvtex^{\beta''}$ with
	$(\alpha'',\beta'')\in\mathcal{A}_{q}$. In view of this observation,
	our claim will follow if we show that $\mathcal{M}_{1}=\mvtex$, or say
	$\mvtex\subseteq\mathcal{M}_{1}$. Fix $g\in\mvtex$, say with $q$-expansion
	$g=\sum_{\mathcal{A}_{q}}d_{\alpha,\beta}^{q}\xvtex^{\alpha
	}\yvtex^{\beta}$. If we reduce this equality modulo $\mvtex$, we get
	$0=\sum_{(0,\beta)\in\mathcal{A}_{q}}\bar{d}_{0,\beta}^{q}
	\bar{y_{1}}^{\beta_{1}}\cdots\bar{y_{s}}^{\beta_{s}}$, where the bar
	stands for reduction modulo $\mvtex$. Given that
	$\bar{y_{1}},\ldots,\bar{y_{s}}$ is a $p$-basis of $R/\mvtex$, the
	coefficients
	$\bar{d}_{0,\beta}^{q}$ have to be zero, that is,
	$d_{0,\beta}\in\mvtex$. Therefore $g\in\mathcal{M}_{1}$. This proves
	the inclusion $\mvtex\subset\mathcal{M}_{1}$, and hence the proof of our
	claim is complete.
	
	(2) For the nontrivial direction, given $f\in R^{q}+\mvtex^{n}$, we have
	$f=\lambda^{q}+g$ with $g\in\mvtex^{m}$. We write $g$ in its $q$-expansion,
	say
	$g=\sum_{\mathcal{A}_{q}}d_{\alpha,\beta}^{q}\xvtex^{\alpha
	}\yvtex^{\beta}$. By (1)
	$d_{\alpha,\beta}^{q}\xvtex^{\alpha}\yvtex^{\beta}\in\mvtex
	^{n}$ for
	all $(\alpha,\beta)\in\mathcal{A}_{q}$. The conclusion follows as clearly
	$c_{\alpha,\beta}=d_{\alpha,\beta}$ for all
	$(\alpha,\beta)\in\mathcal{A}_{q}^{+}$.
\end{proof}

For the following corollary, we recall that if $R$ is a regular local ring,
and if $\pvtex\subset R$ is a regular prime, that is, if
$R/\pvtex$ is regular, then the symbolic powers and the usual powers of
$\pvtex$ coincide. In other words,
$\pvtex^{n}=(\pvtex^{n} R_\pvtex)\cap R$.
\begin{corollary}%
	\label{R^q+p^n_and_localization}
	Let $(R,\mvtex)$ be an $F$-finite regular local ring, and let
	$\pvtex\subset R$ be a regular prime. Then
	$R^{q}+\pvtex^{n}=(R_\pvtex^{q}+\pvtex^{n} R_\pvtex)\cap R$ for all
	$n\geq0$.
\end{corollary}
\begin{proof}
	The inclusion $\subseteq$ in the equality is clear. We now prove the reverse
	inclusion. Choose an adapted $p$-basis
	$(x_{1},\ldots,x_{r},y_{1},\ldots,y_{s})$ for $R$ so that the first
	$x_{1},\ldots,x_{r'}$ (say) generate $\pvtex$. Fix
	$f\in(R_\pvtex^{q}+\pvtex^{n} R_\pvtex)\cap R$, say with $q$-expansion
	$f=\sum_{\mathcal{A}_{q}}c_{\alpha,\beta}^{q} \xvtex^{\alpha
	}\yvtex^{\beta}$. It follows from the definition of $p$-basis that
	$(x_{1},\ldots,x_{r},y_{1},\ldots,y_{s})$ is also a $p$-basis for
	$R_\pvtex$, and it is in fact an adapted $p$-basis for
	$(R_\pvtex,\pvtex R_\pvtex)$. The $q$-expansion of $f$ (viewed in
	$R_\pvtex$) remains the same. Now by Proposition \ref{order_in_terms_of_q-atoms}(2)
	$c_{\alpha,\beta}^{q}\xvtex^{\alpha}\yvtex^{\beta}\in\pvtex^{n}
	R_\pvtex$ for all $(\alpha,\beta)\in\mathcal{A}_{q}^{+}$. Thus
	$c_{\alpha,\beta}^{q}\xvtex^{\alpha}\yvtex^{\beta}\in(\pvtex
	^{n} R_\pvtex)\cap R=\pvtex^{n}$ if
	$(\alpha,\beta)\in\mathcal{A}_{q}^{+}$, and therefore
	$f\in R^{q}+\pvtex^{n}$. This proves the inclusion $\supseteq$ and ends
	the proof of the corollary.
\end{proof}

\subsection{On the Definition of Invariants for Points $x\in\delta
	(F_{d}(X))$}
\label{On_the_definition_of_invariants_for_points_in_F(X)}

We now consider the two numerical values $\nu_\mvtex^{(q)}(f)$ and
$\eta_\mvtex(f)$ mentioned at the introduction of the paper (and of this
section). These numerical values can be interpreted, to some extent, as
analogues of the order of an ideal, but applied to modules. Lemma \ref{Alternative_Lemma} will establish the connection between these two
invariants. The proof of this lemma relies strongly on adapted $p$-basis
and Taylor differential operators. %
\begin{parrafo}%
	\label{Definition_of_q-order}
	\emph{The $q$-Order of an $R^{q}$-Module}. Let $(R,\mvtex)$ be an
	$F$-finite regular local ring and fix $q=p^{e}$, a power of $p$. For
	$f\in R$, we define the $q$-order of $f$ as
	\begin{align*}
		\nu_\mvtex^{(q)}(f):=\sup\{n\in\mbvtex{N}_{0}:
		f\in R^{q}+\mvtex^{n} \}.
	\end{align*}
	Note that
	$\nu_\mvtex^{(q)}(f)=\sup\{\nu_\mvtex(g): g-f\in R^{q}\}$. In particular,
	$\nu_\mvtex(f)\leq\nu_\mvtex^{(q)}(f)$. As a matter of fact, the equality
	holds if $q\nmid\nu_\mvtex(f)$: Suppose that
	$q\nmid\nu_\mvtex(f)$ and $\nu_\mvtex(f)<\nu_\mvtex^{(q)}(f)$; then
	$f=\lambda^{q}+g$ for some $\lambda, g\in R$ with
	$\nu_\mvtex(g)>\nu_\mvtex(f)$, so
	$\nu_\mvtex(f)=\nu_\mvtex(\lambda^{q})=q\nu_\mvtex(\lambda)$, which
	is a contradiction.
	
	We extend the above definition to any subset $S\subset R$ by setting
	$\nu_\mvtex^{(q)}(S):=\sup\{n\in\mbvtex{N}_{0}: S\subseteq R^{q}+
	\mvtex^{n}\}$. Note that
	$\nu_\mvtex^{(q)}(S)=\min\{\nu_\mvtex^{(q)}(f): f\in S\}$. Since
	$R^{q}+\mvtex^{n}$ is an $R^{q}$-subalgebra of $R$, we obtain the following:
\end{parrafo}
\begin{lemma}%
	\label{compatibility_of_the_q-order_with_the_weak_equivalence}
	$\nu_\mvtex^{(q)}(S)=\nu_\mvtex^{(q)}(R^{q}[S])$, where
	$R^{q}[S]$ is the $R^{q}$-subalgebra of $R$ generated by $S$.
\end{lemma}
\begin{parrafo}%
	\label{formula_for_the_q-order}
	Fix now an adapted $p$-basis
	$(x_{1},\ldots,x_{r},y_{1},\ldots,y_{s})$ for $(R,\mvtex)$. Given
	$f\in R$, we consider its $q$-expansion, say
	$f=\sum_{\mathcal{A}_{q}}c_{\alpha,\beta}^{q}\xvtex^{\alpha
	}\yvtex^{\beta}$. Note that by Proposition \ref{order_in_terms_of_q-atoms}(2)
	\begin{align}
		\label{formula_for_the_q-order-equation} \nu_\mvtex^{(q)}(f)=\min\{
		\nu_\mvtex(c_{\alpha,\beta}^{q} \xvtex^{\alpha}
		\yvtex^{\beta}): (\alpha,\beta)\in\mathcal{A}_{q}^{+}
		\}.
	\end{align}
	This implies that $\nu_\mvtex^{(q)}(f)=\infty$ if and only if
	$f\in R^{q}$. It also shows that
	\begin{equation*}
		f^{+}:=\sum_{\mathcal{A}_{q}^{+}}c_{\alpha,\beta}^{q}
		\xvtex ^{\alpha
		}\yvtex^{\beta},
	\end{equation*}
	which is an element whose definition depends on the choice of the adapted
	$p$-basis, has intrinsically defined order. In fact, by Proposition \ref{order_in_terms_of_q-atoms}(1)
	\begin{align*}
		\nu_\mvtex(f^{+})=\min\{\nu_\mvtex(c_{\alpha,\beta}^{q}
		\xvtex ^{\alpha}\yvtex^{\beta}: (\alpha,\beta)\in
		\mathcal{A}_{q}^{+}\} =\nu_\mvtex^{(q)}(f).
	\end{align*}
\end{parrafo}
\begin{parrafo}%
	\label{Definition_of_eta}%
	\emph{Definition of $\eta_\mvtex(M)$}. We have introduced the $q$-order
	$\nu_\mvtex^{(q)}(f)$ of an element $f\in R$. This definition has some
	similarities with the usual order of $f$ at the local ring, although both
	notions are different. For example, as noted previously,
	$\nu_\mvtex^{(q)}(f)=\nu_\mvtex(f)$ whenever
	$q\nmid\nu_\mvtex(f)$. However, $\nu_\mvtex^{(q)}(1)=\infty$.
	
	The notion of $q$-order was extended to $R^{q}$-submodules
	$M\subset R$. Lemma \ref{compatibility_of_the_q-order_with_the_weak_equivalence} expresses
	the compatibility of this definition with the two notion of equivalence
	of $R^{q}$-modules introduced in Definition \ref{defequi}. As was indicated
	in Example  \ref{the_q-order_is_not_upper-semicontinuous}, the drawback
	of the $q$-order is that it does not define an upper-semicontinuous function
	along a scheme. To deal with this issue, we introduce now a new notion
	of order for $R^{q}$-modules, also compatible with equivalences, which
	will define an upper-semicontinuos function. We will discuss this latter
	fact in the next section.
	
	Given $f\in R$, we define
	\begin{align*}
		\eta_\mvtex(f):=\min\{\nu_\mvtex(\operatorname{Diff}_{R,+}^{i}(f))+i:i=1,
		\ldots,q-1\}.
	\end{align*}
	Note that by (\ref{Diff(R^q)=0}) $\eta_\mvtex(f)=\eta_\mvtex(g)$ whenever
	$f-g\in R^{q}$. We extend the above definition to any subset
	$S\subset R$ by setting
	$\eta_\mvtex(S):= \min\{\nu_\mvtex(\operatorname
	{Diff}_{R,+}^{i}(S))+i:i=1,
	\ldots,q-1\}$. Note that if $S=\{f_{1},\ldots,f_{r}\}$, then clearly
	$\eta_\mvtex(S)=\min\{\eta_\mvtex(f_{i}): i=1,\ldots,r\}$. The next
	proposition, which is formulated in a broader context, shows in particular
	that $\eta_\mvtex(S)=\eta_\mvtex(R^{q}[S])$. This implies that
	$\eta_\mvtex(M)$, viewed as an invariant for $R^{q}$-submodules
	$M\subset R$, is also compatible with the two notions of equivalence for
	$R^{q}$-submodules.
\end{parrafo}
\begin{proposition}%
	\label{Diff(M)=Diff(O^q[M])}
	Fix a ring $R$ of characteristic $p$, a finite collection of ideals
	$\Lambda$, and $q$, a power of $p$. Then for any $i=1,\ldots,q-1$, and
	for any nonempty subset $S\subset R$,
	\begin{equation*}
		\operatorname{Diff}_{R,\Lambda,+}^{i}(S)=\operatorname
		{Diff}_{R,\Lambda,+}^{i}(R^{q}[S]).
	\end{equation*}
\end{proposition}
\begin{proof}
	We only need to show that
	$\operatorname{Diff}_{R,\Lambda,+}^{i}(R^{q}[S])\subseteq
	\operatorname{Diff}_{R,
		\Lambda,+}^{i}(S)$. This follows from the following argument. Given
	$r\in R$ and $D\in\operatorname{Diff}_{R,\Lambda}^{i}$, we denote by
	$D r$ the element in $\operatorname{Diff}_{R,\Lambda}^{i}$ obtained
	from the
	structure of right $R$-module, whereas $D(r)(\in R)$ denotes the evaluation.
	Given now elements $f,g\in R$ and
	$D\in\operatorname{Diff}_{R,\Lambda,+}^{i}$, we obtain
	$D(fg)=((Df-D(f)1_{R})(g)+gD(f)$. Note that $Df-D(f)1_{R}$ and $gD$ are
	clearly elements of $\operatorname{Diff}_{R,\Lambda,+}^{i}$; therefore
	$\operatorname{Diff}_{R,\Lambda,+}^{i}(fg)\subseteq\operatorname{Diff}_{R,
		\Lambda,+}^{i}(f)+\operatorname{Diff}_{R,\Lambda
		,+}^{i}(g)=:\operatorname{Diff}_{R,
		\lambda,+}^{i}(\{f,g\})$. An inductive argument then shows that
	$\operatorname{Diff}_{R,\Lambda,+}^{i}(f_{1}\cdots f_{r})\subseteq
	\operatorname{Diff}_{R,
		\Lambda,+}^{i}(\{f_{1},\ldots,f_{r}\})$. Since the elements of
	$\operatorname{Diff}_{R,\Lambda,+}^{i}$ are $R^{q}$-linear, we
	conclude that
	$\operatorname{Diff}_{R,\Lambda,+}^{i}(R^{q}[S])\subseteq
	\operatorname{Diff}_{R,
		\Lambda,+}^{i}(S)$.
\end{proof}

The following lemma shows the connection between the two numerical values
that we have introduced.
\begin{lemma}%
	\label{Alternative_Lemma}
	Fix $f\in R$, and fix an adapted $p$-basis
	$(x_{1},\ldots,x_{r},y_{1},\ldots,y_{s})$ for $(R,\mvtex)$. We consider
	the $q$-expansion of $f$, say
	$f=\sum_{\mathcal{A}_{q}}c_{\alpha,\beta}^{q}\xvtex^{\alpha
	}\yvtex^{\beta}$. Then we have the inequality
	$\nu_\mvtex^{(q)}(f)\leq\eta_\mvtex(f)$.
	
	Assume now that $\nu_\mvtex^{(q)}(f)<\infty$, say
	\begin{equation*}
		\nu_\mvtex^{(q)}(f)=qa+b,\quad  0\leq b<q.
	\end{equation*}
	Then the conclusions in each of the following three cases hold:
	\begin{enumerate}[Case 3:]
		\item[Case 1:] $b>0$. In this case,
		$\nu_\mvtex(\operatorname{Diff}_{R,+}^{b}(f))=qa$, and hence we have
		the equality
		$\nu_\mvtex^{(q)}(f)=\eta_\mvtex(f)$.
		\item[Case 2:] $b=0$, and there exists a term
		$c_{\alpha,\beta}^{q}\xvtex^{\alpha}\yvtex^{\beta}$ of order
		$\nu_\mvtex^{(q)}(f)$ with $\alpha\neq0$. In this case,
		$\nu_\mvtex(\operatorname{Diff}_{R,+}^{q-1}(f))<qa$, and we also
		have the equality
		$\nu_\mvtex^{(q)}(f)=\eta_\mvtex(f)$.
		\item[Case 3:] $b=0$, and the only terms of the $q$-expansion of $f$ different
		from $c_{0,0}^{q}$ and of order $\nu_\mvtex^{(q)}(f)$ are all of the
		form $c_{0,\beta}^{q}\yvtex^{\beta}$. In this case,
		$\nu_\mvtex(\operatorname{Diff}_{R,+}^{q-1}(f))\leq qa$, and
		$qa:=\nu_\mvtex^{(q)}(f)<\eta_\mvtex(f)<q(a+1)$.
	\end{enumerate}
	In particular,
	$qa+b:=\nu_\mvtex^{(q)}(f)\leq\eta_\mvtex(f)<q(a+1)$, and the inequality
	in the middle is an equality if $q\nmid\nu_\mvtex^{(q)}(f)$.
\end{lemma}
\begin{proof}
	If $f\in R^{q}+\mvtex^{n}$, then for any $i=1,\ldots,q-1$,
	$\operatorname{Diff}_{R,+}^{i}(f)\subseteq\operatorname
	{Diff}_{R,+}^{i}(\mvtex^{n})
	\subseteq \mvtex^{n-i}$ by (\ref{Diff(R^q)=0}) and (\ref
	{evaluating_differential_operators_at_powers_of_ideals}),
	whence $n\leq\nu_\mvtex(\operatorname{Diff}_{R,+}^{i}(f))+i$. This
	shows that
	$n\leq\eta_\mvtex(f)$ by the definition of $\eta_\mvtex(f)$. We conclude
	that $\nu_\mvtex^{(q)}(f)\leq\eta_\mvtex(f)$ by the definition of
	$\nu_\mvtex^{(q)}(f)$.
	
	Suppose now that $\nu_\mvtex^{(q)}(f)<\infty$, that is,
	$f\in R^{q}+\mvtex$ and $f\notin R^{q}$.
	
	Assume that we are in Case 1, that is, $b>0$. Then by (\ref
	{formula_for_the_q-order-equation})
	there exists a term
	$c_{\alpha,\beta}^{q}\xvtex^{\alpha}\yvtex^{\beta}$ with
	$(\alpha,\beta)\in\mathcal{A}_{q}^{+}$ such that
	$qa+b=\nu_\mvtex(c_{\alpha,\beta}^{q}\xvtex^{\alpha}\yvtex
	^{\beta})=q
	\nu_\mvtex(c_{\alpha,\beta})+|\alpha|$. Since
	$|\alpha|\equiv b\mod q$, not all entries of $\alpha$ are zero. Let
	$p^{t}<q$ be the highest power of $p$ that divides every
	$\alpha_{i}$. Note that $p^{t}$ divides $b$ since
	$\sum\alpha_{i}\equiv b\mod q$, whence $p^{t}\leq b<q$. Note also that
	by the maximal property of $p^{t}$ there is some index $i_{0}$ such that
	$p^{t}$ is the highest power of $p$ dividing $\alpha_{i_{0}}$. Let
	$\gamma^{(1)}\in\mbvtex{N}_{0}^{r}$ be the multiindex that has
	$p^{t}$ at its $i_{0}$-entry and zeros at the other entries. The term
	$D_{\gamma^{(1)},0}(c_{\alpha,\beta}^{q}\xvtex^{\alpha}\yvtex
	^{\beta})=\binom{\alpha_{i_{0}}}{p^{t}}c_{\alpha,\beta}^{q}\xvtex^{
		\alpha-\gamma^{(1)}}\yvtex^{\beta}$ is a nonzero term in the $q$-expansion
	of $D_{\gamma^{(1)},0}(f)$ by
	\ref{Taylor_operators_are_respectfull_with_the_p-basis} and since
	$\binom{\alpha_{i_{0}}}{p^{t}}\not\equiv0\mod p$ (see (\ref
	{Lucas_theorem})).
	The order of this term is $qa+b-p^{t}$. Therefore
	$D_{\gamma^{(1)},0}(f)$ has the order $qa+b-p^{t}$ by (\ref
	{evaluating_differential_operators_at_powers_of_ideals})
	and Proposition \ref{order_in_terms_of_q-atoms}(1). If $p^{t}=b$, then
	we are done. Otherwise, $0<p^{t}<b$, and we apply the same reasoning with
	$f$ replaced by $D_{\gamma^{(1)},0}(f)$,
	$c_{\alpha,\beta}^{q}\xvtex^{\alpha}\yvtex^{\beta}$ by
	$\binom{\alpha_{i_{0}}}{p^{t}}c_{\alpha,\beta}^{q}\xvtex^{\alpha-
		\gamma^{(1)}}\yvtex^{\beta}$, and $b$ by $b-p^{t}$, and we continue in
	this way. At the end, we will obtain a differential operator of the form
	$D_{\gamma^{(t)},0}\circ\cdots\circ D_{\gamma^{(1)},0}$ such that
	$\nu_\mvtex(D_{\gamma^{(t)},0}\cdots D_{\gamma^{(1)},0}(f))=qa$ and
	$|\gamma^{(t)}|+\cdots+|\gamma^{(1)}|=b$. Since this operator has order
	$b$, by (\ref{evaluating_differential_operators_at_powers_of_ideals}) we
	obtain $\nu_\mvtex(\operatorname{Diff}_{R,+}^{b}(f))=qa$, whence
	$\eta_\mvtex(f)\leq\nu_\mvtex(\operatorname
	{Diff}_{R,+}^{b}(f))+b=\nu_\mvtex^{(q)}(f)$. Combining this with the
	conclusion of the first paragraph,
	we finally obtain that $\eta_\mvtex(f)=\nu_\mvtex^{(q)}(f)$.
	
	Assume now that we are in Case 2, so $\nu_\mvtex^{(q)}(f)=qa$, and there
	exists a term $c_{\alpha,\beta}^{q}\xvtex^{\alpha}\yvtex^{\beta}$ with
	$(\alpha,\beta)\in\mathcal{A}_{q}^{+}$ such that
	$\nu_\mvtex(c_{\alpha,\beta}^{q}\xvtex^{\alpha}\yvtex^{\beta})=
	\nu_\mvtex^{(q)}(f)$ and $\alpha\neq0$. We may assume that
	$\alpha_{1}\neq0$. Set
	$\alpha':=(\alpha_{1},0,\ldots,0)\in\mbvtex{N}_{0}^{r}$. Note that
	$D_{\alpha',0}$ has order $0<|\alpha'|=\alpha_{1}<q$, and hence
	$D_{\alpha',0}(f)$ has the nonzero term
	$D_{\alpha',0}(c_{\alpha,\beta}^{q}\xvtex^{\alpha}\yvtex^{\beta})=c_{
		\alpha,\beta}^{q}x_{2}^{\alpha_{2}}\cdots x_{r}^{\alpha_{r}}
	\yvtex^{\beta}$ in its $q$-expansion by
	\ref{Taylor_operators_are_respectfull_with_the_p-basis}. Since this term
	has the order $\nu_\mvtex^{(q)}(f)-\alpha_{1}$, we obtain that
	$\nu_\mvtex(D_{\alpha',0}(f))=\nu_\mvtex^{(q)}(f)-|\alpha'|$ by
	(\ref{evaluating_differential_operators_at_powers_of_ideals})
	and Proposition \ref{order_in_terms_of_q-atoms}(1). It follows that
	$\nu_\mvtex(\operatorname{Diff}_{R,+}^{|\alpha'|}(f))+|\alpha
	'|=qa=\nu_\mvtex^{(q)}(f)$, hence
	$\eta_\mvtex(f)\leq\nu_\mvtex^{(q)}(f)$. This has to be an equality
	by the conclusion of the first paragraph.
	
	We finally assume that we are in Case 3, so again
	$\nu_\mvtex^{(q)}(f)=qa$, but now we can decomposes
	$f^{+}:=f-c_{0,0}^{q}$ as $f_{m}+f_{h}$, where $f_{m}\neq0$ is the sum
	of those terms $c_{0,\beta}^{q}\yvtex^{\beta}$ of order
	$\nu_\mvtex^{(q)}(f)=\nu_\mvtex(f^{+})$, and $f_{h}$ is the sum of
	higher-order terms. Any $D\in\operatorname{Diff}_{R}^{q-1}$ is $R^{q}$-linear,
	and hence $\nu_\mvtex(D(f_{m}))\geq\nu_\mvtex^{(q)}(f)$, as we easily
	check. We conclude that for any $i=1,\ldots,q-1$ and for any
	$D\in\operatorname{Diff}_{R,+}^{i}$, $D(f)$ is the sum of $D(f_{m})$, which
	is an element of order $\geq\nu_\mvtex^{(q)}(f)$, and $D(f_{h})$, which
	is an element of order
	$\geq\nu_\mvtex(f_{h})-i>\nu_\mvtex^{(q)}(f)-i$ by (\ref
	{evaluating_differential_operators_at_powers_of_ideals}).
	Therefore $D(f)$ has the order $>\nu_\mvtex^{(q)}(f)-i$. This proves
	that $\eta_\mvtex(f)>\nu_\mvtex^{(q)}(f)$. Finally, there exists at
	least one nonzero term $c_{0,\beta}^{q}\yvtex^{\beta}$ of $f_{m}$, and
	we may assume that $\beta_{1}\neq0$. We define
	$\beta'=(\beta_{1},0,\ldots,0)\in\mbvtex{N}_{0}^{s}$. Since
	$D_{0,\beta'}(f)$ has the term
	$D_{0,\beta'}(c_{0,\beta}^{q}\yvtex^{\beta})=c_{0,\beta}^{q}y_{2}^{
		\beta_{2}}\cdots y_{s}^{\beta_{s}}$ in its $q$-expansion (again by
	\ref{Taylor_operators_are_respectfull_with_the_p-basis}), we conclude by
Proposition	\ref{order_in_terms_of_q-atoms} that $D_{0,\beta'}(f)$ has
	the order $\leq\nu_\mvtex^{(q)}(f)$. Therefore
	$\nu_\mvtex(\operatorname{Diff}_{R,+}^{q-1}(f))\leq qa$, and so
	$\eta_\mvtex(f)\leq qa+(q-1)<q(a+1)$. This proves the conclusion in Case
	3.
\end{proof}

\subsection{Applications of Lemma \protect\ref{Alternative_Lemma}}
\label{Applications_of_Lemma_Alt}

We present some corollaries of Lemma \ref{Alternative_Lemma}. The first
corollary states a particular case of a more general and classical result
of commutative algebra, namely, which expresses the order of an element
at a regular local ring in terms of differential operators. In the second
corollary, we characterize the subring $R^{q}$ in terms of differential
operators. The last two corollaries are technical, and they will be used
in Section \ref{Applications_to_O_V^q-submodules_of_O_V} to give a
characterization
of permissible centers for a pair $(\mathscr{M},a)$. This characterization
will be useful for the study of transformation of
$\mathcal{O}_{V}^{q}$-modules.%
\begin{corollary}%
	\label{local_Jacobian_criterion}
	Let $(R,\mvtex)$ be an $F$-finite regular local ring, and let
	$f\in R$. Then for $n\geq1$, $f\in\mvtex^{n}$ if and only if
	$\operatorname{Diff}_{R}^{n-1}(f)\subseteq\mvtex$.
\end{corollary}
\begin{proof}
	If $f\in\mvtex^{n}$, then
	$\operatorname{Diff}_{R}^{n-1}(f)\subset\mvtex$ by (\ref
	{evaluating_differential_operators_at_powers_of_ideals}).
	Suppose now that $f\notin\mvtex^{n}$. If $f\notin\mvtex$, then
	$\operatorname{Diff}_{R}^{n-1}(f)\nsubseteq\mvtex$ since
	$\operatorname{Diff}_{R}^{n-1}$ contains the identity. If $f\in\mvtex
	$, then
	we choose $q$, a power of $p$, which is greater than $n$. Note that
	$0<\nu_\mvtex(f)<q$, whence
	$\nu_\mvtex^{(q)}(f)=\nu_\mvtex(f)$, as was noted in
	\ref{Definition_of_q-order}. Note that we are in Case 1 of Lemma \ref{Alternative_Lemma} with $a=0$ and $b:=\nu_\mvtex^{(q)}(f)>0$; therefore
	$\nu_\mvtex(\operatorname{Diff}_{R,+}^{b}(f))=0$. Since $b\leq n-1$,
	we conclude
	that $\operatorname{Diff}_{R}^{n-1}(f)\nsubseteq\mvtex$.
\end{proof}
\begin{corollary}%
	\label{characterization_of_R^q}
	Let $R$ be an $F$-finite regular local ring. Then
	\begin{equation*}
		R^{q}=\{f\in R:\operatorname{Diff}_{R,+}^{q-1}(f)=0
		\}.
	\end{equation*}
\end{corollary}
\begin{proof}
	The inclusion $\subseteq$ follows from (\ref{Diff(R^q)=0}). Conversely,
	if $\operatorname{Diff}_{R,+}^{q-1}(f)=0$, then
	$\operatorname{Diff}_{R,+}^{i}(f)=0$ for all $i=1,\ldots,q-1$, whence
	$\eta_\mvtex(f)=\infty$. Therefore by Lemma \ref{Alternative_Lemma} $\nu_\mvtex^{(q)}(f)=\infty$. This implies that
	$f\in R^{q}$, as was noted in \ref{formula_for_the_q-order}.
\end{proof}
\begin{corollary}%
	\label{characterization_of_R^q+m^qa}
	Let $(R,\mvtex)$ be an $F$-finite regular local ring. Given an integer
	$a\geq1$, $f\in R^{q}+\mvtex^{qa}$ if and only if
	$\operatorname{Diff}_{R,+}^{q-1}(f)\subseteq\mvtex^{q(a-1)+1}$.
\end{corollary}
\begin{proof}
	If $f\in R^{q}+\mvtex^{qa}$, then
	$qa\leq\nu_\mvtex^{(q)}(f)\leq\eta_\mvtex(f)\leq q-1+\nu_\mvtex
	(\operatorname{Diff}_{R,+}^{q-1}(f))$ by Lemma \ref{Alternative_Lemma}; therefore
	$\nu_\mvtex(\operatorname{Diff}_{R,+}^{q-1}(f))\geq
	qa-(q-1)=q(a-1)+1$, that
	is, $\operatorname{Diff}_{R,+}^{q-1}(f)\subseteq\mvtex^{q(a-1)+1}$. If
	$f\notin R^{q}+\mvtex^{qa}$, then $\nu_\mvtex^{(q)}(f)<qa$, say
	$\nu_\mvtex^{(q)}(f)=qa'+b$ with $a'<a$ and $0\leq b<q$. Thus by Lemma \ref{Alternative_Lemma} we have the inequality
	$\nu_\mvtex(\operatorname{Diff}_{R,+}^{q-1}(f))\leq qa'$, whence
	$\nu_\mvtex(\operatorname{Diff}_{R,+}^{q-1}(f))+q-1<qa$, that is,
	$\operatorname{Diff}_{R,+}^{q-1}(f)\nsubseteq\mvtex^{q(a-1)+1}$.
	This completes
	the proof.
\end{proof}
\begin{corollary}%
	\label{n=qa_in_the_local_case}
	Let $(R,\mvtex)$ be an $F$-finite regular local ring, and let
	$x\in R$ be a nonzero element such that $R/xR$ is regular. Given
	$f\in R\setminus R^{q}$, let $n$ denote the largest integer such that
	$\operatorname{Diff}_{R,+}^{q-1}(f)\subseteq x^{n}R$, and let $a$
	denote the
	largest integer such that $f\in R^{q}+x^{qa}R$. Then $n=qa$.
\end{corollary}
\begin{proof}
	We can write $f=\lambda^{q}+x^{qa}g$ for some $\lambda, g\in R$. Then
	$\operatorname{Diff}_{R,+}^{q-1}(f)=\operatorname
	{Diff}_{R,+}^{q-1}(x^{qa}g)=x^{qa}
	\operatorname{Diff}_{R,+}^{q-1}(g)\subseteq x^{qa}R$. This shows that
	$qa\leq n$. If $qa<n$, then
	$\operatorname{Diff}_{R,+}^{q-1}(f)\subseteq x^{qa+1}R$, and by
	Corollary \ref{characterization_of_R^q+m^qa}, using
	\ref{compatibility_with_localization}, we obtain
	$f\in R_{\pvtex}^{q}+x^{q(a+1)}R_{\pvtex}$ with $\pvtex=Rx$. It follows
	from Corollary \ref{R^q+p^n_and_localization} that
	$f\in R^{q}+x^{q(a+1)} R$, which contradicts the definition of $a$. We
	conclude that $qa=n$.
\end{proof}	

\section{Differential Operator Techniques Applied to the Study of
	$\mathcal{O}_{V}^{q}$-Modules and Morphisms}
\label{Section_on_Diff2}

Let $(R,\mvtex)$ be an $F$-finite regular local ring, and let
$\operatorname{Spec}(R) \leftarrow V_{1}\supset H_{1}$ be the blowup
at the closed
point, where $H_{1}$ denotes the exceptional hypersurface. Here
$V_{1}$ is regular, and if $\xi\in H_{1}$ denotes the generic point of
$H_{1}$, then $\mathcal{O}_{V_{1}, \xi}$ is a valuation ring. Given an
ideal $J$ in $R$, $J\subset\mvtex^{n}$ for a positive integer $n$ if
and only if
$J\cdot\mathcal{O}_{V_{1}}=\mathcal{I}(H_{1})^{n} \mathcal{J}_{1}$
for some
$\mathcal{O}_{V_{1}}$-ideal $\mathcal{J}_{1}$. In particular, $J$ has order
$n$ if $n$ is the valuation of $J$ at $\mathcal{O}_{V_{1}, \xi}$. So the
order of an ideal $J$ can be characterized either in terms of differential
operators at $(R,\mvtex)$ (see Corollary \ref{local_Jacobian_criterion}) or by the order of vanishing of the total
transform $J\cdot\mathcal{O}_{V_{1}}$ along the exceptional hypersurface
$H_{1}$.

Here $V$ is an irreducible $F$-finite regular scheme, and we aim to study
the $\mathcal{O}_{V}^{q}$-submodules of $\mathcal{O}_{V}$, a task initiated
in Section \ref{Section_on_modules}, always intending to draw parallels
with the realm of ideals on $\mathcal{O}_{V}$. The notion of order of an
ideal at a local regular ring $\mathcal{O}_{V,x}$ is replaced here by that
of $\eta_{x}(\mathscr{M})$ for a given $\mathcal{O}_{V}^{q}$-submodule
$\mathscr{M}$ (see \ref{par416}).

Once we fix a regular center $Z\subset V$ and an
$\mathcal{O}_{V}^{q}$-submodule $\mathscr{M}$ of $\mathcal{O}_{V}$,
we study
an analog to the order of vanishing for the total transform
$\mathscr{M}\mathcal{O}^{q}_{V_{1}}$ (Definition \ref{the_a-transform_of_an_O^q-module}) along the exceptional hypersurface
$H_{1}$ at the blowup, say $V\leftarrow V_{1}\subset H_{1}$. The function
$\eta_{x}(\mathscr{M})$ will ultimately enable us, at the end of this section,
to characterize $a$-transforms and the notion of permissible center for
$(\mathscr{M},a)$ in Definition \ref{permissible_center_for_(M,a)} (see
Proposition \ref{Describing_Sing(M,a)_with_eta(M)}).

The section begins by reviewing some results concerning the behavior of
differential operators when blowing up along suitable centers, results
which are essential in our forthcoming discussions.

\subsection{Behavior of Differential Operators with Morphisms}
\label{Behavior_of_differential_operators_with_morphisms}

We recall some facts about the behavior of differential operators with
respect to homomorphisms of rings, which include those that arise when
taking monoidal transformations along regular centers.
\begin{lemma}%
	\label{lemma_on_restriction_of_Diff}
	Let $R$ be an $F$-finite regular local ring, and let
	$\varphi:R\rightarrow R'$ be any ring homomorphism from $R$. Given
	$D'\in\operatorname{Diff}_{R'}^{i}$ (resp., $\operatorname
	{Diff}_{R',+}^{i}$), there
	are differential operators
	$D_{1},\ldots,D_{n}\in\operatorname{Diff}_{R}^{i}$ (resp.,
	$\operatorname{Diff}_{R,+}^{i}$) and scalars $r_{1}',\ldots,r_{n}'\in
	R'$ such
	that
	\begin{equation*}
		D'\circ\varphi=r_{1}'(\varphi\circ
		D_{1})+\cdots+r_{n}'(\varphi \circ
		D_{n}).
	\end{equation*}
\end{lemma}
\begin{proof}
	We view differential operators of order $\leq i$ from $R$ to an $R$-module
	$M$ as elements of $\operatorname{Hom}_{R}(\operatorname
	{P}_{R}^{i},M)$, where
	$\operatorname{P}_{R}^{i}$ denotes the $R$-module of principal parts
	of order
	$\leq i$. In this regard,
	$D'\circ\varphi\in\operatorname{Hom}_{R}(\operatorname
	{P}_{R}^{i},R')$. The existence
	of a finite $p$-basis for $R$ implies that $\operatorname{P}_{R}^{i}$
	is free
	of finite rank; therefore $D'\circ\varphi$ is a linear combination, say
	$D'\circ\varphi=\sum D'(\varphi(u_{k}))(\varphi\circ D_{k})$, where
	$\{u_{1},\ldots,u_{n}\}$ is a basis for $\operatorname{P}_{R}^{i}$, and
	$D_{1},\ldots, D_{n}$ is the dual basis. In the case that $D'(1)=0$, we
	have $0=D'(1)=\sum D'(\varphi(u_{k}))\varphi(D_{k}(1))$, whence
	$D'\circ\varphi=\sum D'(\varphi(u_{k})) (\varphi\circ D_{k})=
	\sum D'(\varphi(u_{k}))(\varphi\circ D_{k})-\sum D'(\varphi(u_{k}))
	\varphi(D_{k}(1)) \varphi=\sum D'(\varphi(u_{k}))(\varphi\circ(D_{k}-
	D_{k}(1)\operatorname{id}_{R}))$. Notice that
	$D_{k}-D_{k}(1)\operatorname{id}_{R}\in\operatorname
	{Diff}_{R,+}^{i}$. This completes
	the proof of the lemma.
\end{proof}

Given a morphism $\varphi:R\rightarrow R'$ of rings of characteristic
$p$, and given an ideal $I\subset R$, we denote by $IR'$, as usual, the
ideal of $R'$ generated by $\varphi(I)$. If $M$ is an $R^{q}$-submodule
of $R$, then $MR^{\prime q}$ denotes the ${R'}^{q}$-submodule
generated by
$\varphi(M)$.
\begin{corollary}%
	\label{behavoir_of_Diff_with_morphisms_of_local_rings}
	Let $R$ be an $F$-finite regular local ring, and let
	$\varphi:R\rightarrow R'$ be a homomorphism of rings. Let
	$M\subset R$ be an $R^{q}$-submodule, and let $I\subset R$ be an ideal.
	Then the following inclusions hold.
	\begin{enumerate}[(2)]
		\item[(1)] $\operatorname{Diff}_{R',+}^{i}(M{R'}^{q})\subseteq
		\operatorname{Diff}_{R,+}^{i}(M)
		R'$ for $i=1,\ldots, q-1$.
		\item[(2)] $\operatorname{Diff}_{R'}^{i}(I R')\subseteq\operatorname
		{Diff}_{R}^{i}(I)R'$,
		$\forall i\geq0$.
	\end{enumerate}
\end{corollary}
\begin{proof}
	(1) Since $i<q$, the elements of $\operatorname{Diff}_{R',+}^{i}$ are
	$R^{\prime q}$-linear, and hence
	$\operatorname{Diff}_{R',+}^{i}(M R^{\prime q})=\operatorname
	{Diff}_{R',+}^{i}(\varphi(M))$,
	and by the previous lemma this set is included in the $R'$-module generated
	by $\varphi(\operatorname{Diff}_{R,+}^{q-1}(M))$.
	
	(2) Using that
	$\operatorname{Diff}_{R'}^{i}(\varphi(I)R')=\operatorname{Diff}_{R'}^{i}(
	\varphi(I))$, which follows, for instance, from the statement right below
	(\ref{Decomposition_of_Diff=R+Diff_+}), the inclusions in (2) follow
	immediately
	from Lemma \ref{lemma_on_restriction_of_Diff}.
\end{proof}
\begin{parrafo}%
	\label{blow-up}
	Let $R$ be a domain, and let $P$ be any prime ideal, say generated by
	$x_{1},\ldots,x_{n}$. Then the blowup of $\operatorname{Spec}(R)$ along
	$\operatorname{Spec}(R/P)\subset\operatorname{Spec}(R)$ is obtained
	by gluing, in
	the natural way, the affine schemes $\operatorname{Spec}(R_{i})$,
	$i=1,\ldots,n$, where $R_{i}$ denotes
	$\{\frac{z}{x_{i}^{t}}\in R_{x_{i}}: z\in P^{t}, t\geq0\}$, which is an
	$R$-subalgebra of $R_{x_{i}}$. Note that for any integer $k\geq1$, we
	have the following equalities:
	\begin{align}
		\label{power_of_the_exceptional} P^{k} R_{i}=x_{i}^{k}
		R_{i}=\biggl\{\frac{z}{x_{i}^{t}}\in R_{x_{i}}: z\in
		P^{k+t}, t\geq0\biggr\}\subset R_{i}.
	\end{align}
	
	As was mentioned in \ref{compatibility_with_localization}, a differential
	operator $D:R\to R$ can be extended uniquely to a differential operator
	of $R_{x_{i}}$ (which we still denote $D$); however, this extension might
	not restrict to a differential operator of $R_{i}$, that is,
	$D(R_{i})$ might not be included in $R_{i}$. For example, take
	$R=k[x_{1},x_{2}]$ and $P=(x_{1},x_{2})$. Then
	$\frac{\partial}{\partial x_{1}}(\frac{x_{2}}{x_{1}})=-
	\frac{x_{2}}{x_{1}^{2}}$, which is not in
	$R_{1}:=k[x_{1},\frac{x_{2}}{x_{1}}]$. The following lemma shows that this
	does not happen when $D$ is $P$-logarithmic (in our example, we would take
	$x_{1}\frac{\partial}{\partial x_{1}}$).
\end{parrafo}
\begin{lemma}%
	\label{local_Girauds_lemma}
	Within the setting of \ref{blow-up}, given a $P$-logarithmic differential
	operator $D:R\to R$, its extension to a differential operator of
	$R_{x_{i}}$ restricts to $R_{i}$, that is, $D(R_{i})\subset R_{i}$. Moreover,
	$D$ is $x_{i} R_{i}$-logarithmic.
\end{lemma}
\begin{proof}
	We may assume that $i=1$. Fix
	$\frac{z}{x_{1}^{t}}\in x_{1}^{k}R_{1}$ with $z\in P^{t+k}$ (\ref
	{power_of_the_exceptional}).
	Let $q$ be a power of $p$ that is strictly greater than the order of
	$D$, so that $D$ is $R_{x_{1}}^{q}$-linear. Then
	$D(\frac{z}{x_{1}^{t}})=D(\frac{zx_{1}^{t(q-1)}}{x_{1}^{tq}})=
	\frac{D(zx_{1}^{t(q-1)})}{x_{1}^{tq}}$. Notice that
	$zx_{1}^{t(q-1)}\in P^{k+tq}$, and since $D$ is $P$-logarithmic, we have
	that $D(zx_{1}^{t(q-1)})\in P^{k+tq}$. In conclusion,
	$D(\frac{z}{x_{1}^{t}})\in x_{1}^{k}R_{1}$ by (\ref
	{power_of_the_exceptional}).
	
	We have proved that $D(x_{1}^{k} R_{1})\subseteq x_{1}^{k} R_{1}$ for all
	$k\in\mbvtex{N}_{0}$; in particular, $D(R_{1})\subset R_{1}$. Hence
	$D$ is a differential operator of $R_{1}$ that is $x_{1} R_{1}$-logarithmic
	(and of the same order as $D$).
\end{proof}
\begin{corollary}%
	\label{corollary_of_Giraud_lemma}
	Within the setting of \ref{blow-up}, if we identify differential operators
	of a domain with their extension to the fraction field, then for each
	$i\geq0$, there are inclusions
	\begin{align}
		\label{inclusions_in_Giraud_lemma} x_{1}^{i}\operatorname{Diff}_{R}^{i}
		\subseteq\operatorname {Diff}_{R_{1}}^{i} \quad \text{and}\quad
		x_{1}^{i}\operatorname {Diff}_{R,+}^{i}
		\subseteq\operatorname{Diff}_{R_{1},+}^{i}.
	\end{align}
	In particular, if $M\subseteq R$ is an $R^{q}$-submodule, and
	$I\subset R$ is an ideal, then there are inclusions
	\begin{align*}
		x_{1}^{i}\operatorname{Diff}_{R,+}^{i}(M)
		R_{1}&\subseteq \operatorname{Diff}_{R_{1},+}^{i}(M
		R_{1}^{q}),\quad \forall i=1,\ldots,q-1,
		\\
		x_{1}^{i}\operatorname{Diff}_{R}^{i}(I)
		R_{1}&\subseteq\operatorname {Diff}_{R_{1}}^{i}(IR_{1}),\quad
		\forall i\geq0.
	\end{align*}
\end{corollary}
\begin{proof}
Lemma	\ref{local_Girauds_lemma} shows that
	$\operatorname{Diff}_{R,P}^{i}\subseteq\operatorname
	{Diff}_{R_{1},x_{1} R_{1}}^{i}$,
	so the inclusions in (\ref{inclusions_in_Giraud_lemma}) are consequences
	of the inclusion
	$x_{1}^{i}\operatorname{Diff}_{R}^{i}\subset\operatorname
	{Diff}_{R,P}^{i}$ observed
	in (\ref{IDiff_is_logarithmic}).
	
	The last inclusions in the corollary follow from (\ref
	{inclusions_in_Giraud_lemma}).
	We only need to take into account that
	$\operatorname{Diff}_{R_{1},+}^{i}(M R_{1}^{q})=\operatorname
	{Diff}_{R_{1},+}^{i}(M)$
	if $i<q$ and
	$\operatorname{Diff}_{R_{1}}^{i}(IR_{1})=\operatorname
	{Diff}_{R_{1}}^{i}(I)$ for any
	$i\geq0$.
\end{proof}

\subsection{Applications to $\mathcal{O}_{V}^{q}$-Submodules of
	$\mathcal{O}_{V}$}
\label{Applications_to_O_V^q-submodules_of_O_V}

In what follows, $V$ denotes a connected $F$-finite regular scheme, and
$q=p^{e}$ is a fixed power of $p$. We make use of the properties of localization
of differential operators established in (\ref{S^-1Diff(R)=Diff(S^-1R)})
to translate the results from \S\ref{Applications_of_Lemma_Alt} and
\S\ref{Behavior_of_differential_operators_with_morphisms}, formulated
there at a local level, into the global setting of schemes and sheaves
of modules. We then discuss the role played by the function
$x\mapsto\eta_{x}(\mathscr{M})$ in the characterization of permissible
centers and in the definition of the ``right'' notion of transformation
of modules.

\begin{parrafo}
	Let $\mathcal{I}$ be an $\mathcal{O}_{V}$ ideal, and let $\mathscr
	{M}$ be
	an $\mathcal{O}_{V}^{q}$-submodule of $\mathcal{O}_{V}$. Then by
	(\ref{localization})
	there are $\mathcal{O}_{V}$-ideals $\Diff_{V}^{i}(\mathcal{I})$ for all
	$i\geq0$ and $\Diff_{V,+}^{i}(\mathscr{M})$ for all
	$i=1,\ldots,q-1$ such that for any $x\in V$,
	\begin{align}
		\begin{split}\label{good_localization}(\Diff_{V}^{i}(
			\mathcal{I}))_{x}&=\operatorname{Diff}_{\mathcal
				{O}_{V,x}}^{i}(
			\mathcal{I}_{x}), \quad \forall i\geq0,
			\\
			(\Diff_{V,+}^{i}(\mathscr{M}))_{x}&=
			\operatorname{Diff}_{\mathcal
				{O}_{V,x},+}^{i}( \mathscr{M}_{x}),
			\quad \forall i=1,\ldots,q-1.
	\end{split}\end{align}
	
	Propositions \ref{Jacobian_criterion:_global_version}--\ref
	{the_largest_integer_n_such_that_Diff(M)_is_included_in_I(H)^n}
	below are just translations of the results of \S\ref{Applications_of_Lemma_Alt} into the language of sheaves. This translation
	is possible thanks to (\ref{good_localization}). For instance, the following
	proposition, which we will call the \emph{absolute Jacobian criterion},
	is a consequence of its local version stated in Corollary \ref{local_Jacobian_criterion}. It is at the core of most arguments in
	our discussion.
\end{parrafo}
\begin{proposition}%
	\label{Jacobian_criterion:_global_version}
	Let $\mathcal{J}$ be an $\mathcal{O}_{V}$-ideal. Then for any integer
	$n\geq0$,
	\begin{equation*}
		\{x\in V:\nu_{x}(\mathcal{J})\geq n\}=\{x\in V:
		\nu_{x}(\Diff_{V}^{n-1}( \mathcal{J}))\geq1
		\}.
	\end{equation*}
	In particular, the assignment $x\mapsto\nu_{x}(\mathcal{J})$ defines
	an upper semicontinuous function $V\to\mbvtex{N}_{0}$.
\end{proposition}
We now characterize the $\mathcal{O}_{V}^{q}$-submodules of
$\subseteq\mathcal{O}_{V}^{q}$.
\begin{proposition}%
	\label{characterization_of_O^q}
	For an $\mathcal{O}_{V}^{q}$-module
	$\mathscr{M}\subseteq\mathcal{O}_{V}$, the following conditions are
	equivalent.
	\begin{enumerate}[(3)]
		\item[(1)] $\Diff_{V,+}^{q-1}(\mathscr{M})=0$.
		\item[(2)] $\operatorname{Diff}_{\mathcal{O}_{V,x},+}^{q-1}(\mathscr
		{M}_{x})=0$ for
		some $x\in V$.
		\item[(3)] $\mathscr{M}_{x}\subseteq\mathcal{O}_{V,x}^{q}$ for some point
		$x\in V$.
		\item[(4)] $\mathscr{M}\subseteq\mathcal{O}_{V}^{q}$.
	\end{enumerate}
\end{proposition}
\begin{proof}
	The implication (1) $\Rightarrow$ (2) is trivial, and (2) $\Rightarrow
	$ (3)
	follows from Corollary \ref{characterization_of_R^q}. If (3) holds, then
	certainly $\mathscr{M}_{\xi}\subseteq\mathcal{O}_{V,\xi
	}^{q}=K^{q}$, where
	$\xi\in V$ is the generic point, and $K$ denotes the function field of
	$V$. Next, given $y\in V$, clearly
	$\mathscr{M}_{y}\subseteq\mathscr{M}_{\xi}$, and hence
	$\mathscr{M}_{y}\subseteq K^{q}\cap\mathcal{O}_{V,y}=\mathcal{O}_{V,y}^{q}$,
	where the last equality holds since $\mathcal{O}_{V,y}$ is regular and hence
	normal. This proves (3) $\Rightarrow$ (4). Finally, (4) $\Rightarrow$ (1)
	follows from (\ref{Diff(R^q)=0}).
\end{proof}

We now describe
$\operatorname{Sing}(\mathscr{M},a)=\{x\in V:\mathscr
{M}_{x}\subseteq\mathcal{O}_{V,x}^{q}+m_{V,x}^{qa}
\}$ (Definition \ref{permissible_center_for_(M,a)}) as a closed subset
of $V$.
\begin{proposition}%
	\label{d(F(X))_is_closed}
	For an $\mathcal{O}_{V}^{q}$-module
	$\mathscr{M}\subseteq\mathcal{O}_{V}$ and an integer $a\geq1$, we have
	\begin{equation*}
		\operatorname{Sing}(\mathscr{M},a)=\{x\in V:\nu_{x}(
		\Diff_{V,+}^{q-1}( \mathscr{M}))\geq q(a-1)+1\}.
	\end{equation*}
	In particular, $\operatorname{Sing}(\mathscr{M},1)\subset V$ is the
	closed subset
	defined by $\Diff_{V,+}^{q-1}(\mathscr{M})$, and each
	$\operatorname{Sing}(\mathscr{M},a)$ is also closed.
\end{proposition}
\begin{proof}
	The equality follows from Corollary \ref{characterization_of_R^q+m^qa}, and the last statement follows from
	Proposition \ref{Jacobian_criterion:_global_version} when applied to
	$\mathcal{J}=\Diff_{V,+}^{q-1}(\mathscr{M})$ and $n=q(a-1)+1$.
\end{proof}
\begin{remark}%
	\label{final_part_of_main_theorem2}
	In view of Remark \ref{proof_of_the_first_version_of_main_theorem2},
	Proposition \ref{d(F(X))_is_closed} completes the proof of
	theorems \ref{finite_morphisms_and_closed_subsets} and \ref{main_theorem_2}.
\end{remark}

In Definition \ref{permissible_center_for_(M,a)}, we introduced the notion
of permissible center for $(\mathscr{M},a)$ with three conditions, which
were claimed to be equivalent. We proved (1) $\Leftrightarrow$ (2)
$\Rightarrow$ (3). We now complete the proof of that claim and add another
equivalent condition ((4) or (5) further) using differential operators.
\begin{proposition}%
	\label{characterization_of_permissible_centers_for_a_pair_(M,a)}
	Let $\mathscr{M}\subseteq\mathcal{O}_{V}$ be an $\mathcal
	{O}_{V}^{q}$-module,
	and let $a$ be a positive integer. Given an irreducible regular subscheme
	$Z\subset V$, the following conditions on $Z$ are equivalent.
	\begin{enumerate}[(3)]
		\item[(1)] $\mathscr{M}\subseteq\mathcal{O}_{V}^{q}+\mathcal{I}(Z)^{qa}$.
		\item[(2)] $Z$ is included in $\operatorname{Sing}(\mathscr{M},a)$.
		\item[(3)] $\xi\in\operatorname{Sing}(\mathscr{M},a)$, where $\xi
		$ denotes the generic
		point of $Z$.
		\item[(4)] $\nu_{\xi}(\Diff_{V,+}^{q-1}(\mathscr{M}))\geq q(a-1)+1$.
		\item[(5)] $\Diff_{V,+}^{q-1}(\mathscr{M})\subseteq\mathcal
		{I}(Z)^{q(a-1)+1}$.
	\end{enumerate}
\end{proposition}
\begin{proof}
	The implications (1) $\Rightarrow$ (2) and (2) $\Rightarrow$ (3) are clear,
	and (3) $\Rightarrow$ (1) is a consequence of Corollary \ref{R^q+p^n_and_localization} with $n=qa$. The equivalence of (3) and
	(4) follows from Proposition \ref{d(F(X))_is_closed}, and the equivalence
	of (4) and (5) holds because $Z$ is a regular subscheme in a regular scheme.
\end{proof}

The following proposition expresses a peculiarity of regular centers of
codimension one. It is a consequence of its local version stated in Corollary \ref{n=qa_in_the_local_case}.
\begin{proposition}%
	\label{the_largest_integer_n_such_that_Diff(M)_is_included_in_I(H)^n}
	Let $\mathscr{M}$ be a nontrivial $\mathcal{O}_{V}^{q}$-module, and let
	$H\subset V$ be an irreducible regular hypersurface. Let $n$ be the largest
	integer such that
	$\Diff_{V,+}^{q-1}(\mathscr{M})\subseteq\mathcal{I}(H)^{n}$, and let
	$a$ be the largest integer such that $H$ is permissible for
	$(\mathscr{M},a)$. Then $n=qa$. In particular, $n$ is a multiple of
	$q$.
\end{proposition}

The next two propositions present general properties of differential operators
to be used in our study of transformations of modules and of ideals under
morphisms. They are just consequences of Corollary \ref{behavoir_of_Diff_with_morphisms_of_local_rings} and Corollary \ref{corollary_of_Giraud_lemma}.
\begin{proposition}%
	\label{comparison_between_Diff_of_two_regular_schemes}
	Let $\pi:V_{1}\rightarrow V$ be a morphism between $F$-finite regular
	schemes, let $\mathcal{I}$ be an $\mathcal{O}_{V}$-ideal, and let
	$\mathscr{M}\subseteq\mathcal{O}_{V}$ be an $\mathcal{O}_{V}^{q}$-submodule.
	Then we have the following inclusions:
	\begin{enumerate}[(2)]
		\item[(1)] $\Diff_{V_{1},+}^{i}(\mathscr{M}\mathcal
		{O}_{V_{1}}^{q})\subseteq
		\Diff_{V,+}^{i}(\mathscr{M})\mathcal{O}_{V_{1}}$ for
		$i=1,\ldots,q-1$.
		\item[(2)] $\Diff_{V_{1}}^{i}(\mathcal{I}\mathcal
		{O}_{V_{1}})\subseteq\Diff_{V}^{i}(
		\mathcal{I})\mathcal{O}_{V_{1}}$, $\forall i\geq0$.
	\end{enumerate}
\end{proposition}
\begin{proposition}%
	\label{Giraud's_lemma:_Global_version}
	Let $V\xleftarrow{\pi} V_{1}\supset H_{1}$ be the blowup of $V$ along an
	irreducible regular center $Z\subset V$. Given an
	$\mathcal{O}_{V}^{q}$-module $\mathscr{M}$ and an $\mathcal{O}_{V}$-ideal
	$\mathcal{J}$, we have the following inclusions:
	\begin{enumerate}[(2)]
		\item[(1)] $\mathcal{I}(H_{1})^{i}(\Diff_{V,+}^{i}(\mathscr
		{M})\mathcal{O}_{V_{1}})
		\subseteq\Diff_{V_{1},+}^{i}(\mathscr{M}\mathcal{O}_{V_{1}}^{q})$ for
		$i=1,\ldots,q-1$.
		\item[(2)] $\mathcal{I}(H_{1})^{i}(\Diff_{V}^{i}(\mathcal
		{J})\mathcal{O}_{V_{1}})
		\subseteq\Diff_{V_{1}}^{i}(\mathcal{J}\mathcal{O}_{V_{1}})$,
		$\forall i\geq0$.
	\end{enumerate}
\end{proposition}

As a first application of the previous result, we show how they lead to
a proof of Theorem \ref{v_1(I)_is_at_most_v(I)} concerning the behavior
of the order of an ideal by blowups along regular centers.
\begin{proposition}%
	\label{the_order_function_of_the_transform_of_an_ideal}
	Let $V\xleftarrow{\pi} V_{1}\supset H_{1}$ be the blowup of $V$ along an
	irreducible regular center $Z\subset V$. Assume that the mapping
	$x\mapsto\nu_{x}(\mathcal{J})$ is constant along $Z$, say
	$\nu_{z}(\mathcal{J})=b$ for $z\in Z$. Let $\mathcal{J}_{1}$ be the
	$\mathcal{O}_{V_{1}}$-ideal such that
	$\mathcal{J}\mathcal{O}_{V_{1}}=\mathcal{I}(H_{1})^{b}\mathcal
	{J}_{1}$. Then
	for any $x_{1}\in V_{1}$,
	\begin{equation*}
		\nu_{x_{1}}(\mathcal{J}_{1})\leq\nu_{\pi(x_{1})}(
		\mathcal{J}).
	\end{equation*}
\end{proposition}
\begin{proof}
	We only need to prove the claim for points $x_{1}\in H_{1}$. In this case,
	we have to prove that $\nu_{x_{1}}(\mathcal{J}_{1})\leq b$ or, equivalently,
	$\nu_{x_{1}}(\mathcal{J}\mathcal{O}_{V_{1}})\leq2b$. To this end,
	by Proposition \ref{Jacobian_criterion:_global_version} it suffices to show that
	$\nu_{x_{1}}(\Diff_{V_{1}}^{2b}(\mathcal{J}\mathcal
	{O}_{V_{1}}))=0$. Observe
	that
	\begin{equation*}
		\Diff_{V_{1}}^{2b}(\mathcal{J}\mathcal{O}_{V_{1}})
		\supseteq\Diff _{V_{1}}^{b}( \Diff_{V_{1}}^{b}(
		\mathcal{J}\mathcal{O}_{V_{1}})) \stackrel{\text{(\ref{Giraud's_lemma:_Global_version})}} {\supseteq
		}\Diff_{V_{1}}^{b}( \mathcal{I}(H_{1})^{b}
		\Diff_{V}^{b}(\mathcal{J})\mathcal{O}_{V_{1}}).
	\end{equation*}
	Given that $\nu_{\pi(x_{1})}(\mathcal{J})=b$, Proposition \ref{Jacobian_criterion:_global_version} ensures that
	$(\Diff_{V}^{b}(\mathcal{J}))_{\pi(x_{1})}=\mathcal{O}_{V,\pi(x_{1})}$,
	whence
	\begin{equation*}
		(\Diff_{V_{1}}^{b}(\mathcal{I}(H)^{b}
		\Diff_{V}^{b}(\mathcal{J}) \mathcal{O}_{V_{1}}))_{x_{1}}=
		\operatorname{Diff}_{\mathcal
			{O}_{V_{1},x_{1}}}^{b}( \mathcal{I}(H_{1})_{x_{1}}^{b})
		\stackrel{\text{(\ref
				{Jacobian_criterion:_global_version})}} {=}\mathcal{O}_{V_{1},x_{1}}.
	\end{equation*}
	We conclude that
	$(\Diff_{V_{1}}^{2b}(\mathcal{J}\mathcal
	{O}_{V_{1}}))_{x_{1}}=\mathcal{O}_{V_{1},x_{1}}$,
	that is,
	$\nu_{x_{1}}(\Diff_{V_{1}}^{2b}(\mathcal{J}\mathcal
	{O}_{V_{1}}))=0$, as
	was to be proved.
\end{proof}

We add another application, which is of interest for the interpretation
of the transformations of modules as blowups stated in Corollary \ref{a-transform_as_sequence_of_blow-ups}(b).
\begin{proposition}%
	\label{the_invariant_a_in_the_blow-up}
	Let $\mathscr{M}\subseteq\mathcal{O}_{V}$ be an $\mathscr
	{O}_{V}^{q}$-module,
	and let $a$ be a positive integer. Let $Z\subset V$ be a closed irreducible
	and regular subscheme, and consider the blowup
	$V\leftarrow V_{1}\supset H_{1}$ of $V$ along $Z$, where $H_{1}$ denotes
	the exceptional divisor. Then $Z$ is permissible for
	$(\mathscr{M},a)$ if and only if $H_{1}$ is permissible for
	$(\mathscr{M}\mathcal{O}_{V_{1}}^{q},a)$.
\end{proposition}
\begin{proof}
	If $Z$ is permissible for $(\mathscr{M},a)$, then $H_{1}$ is permissible
	for $(\mathscr{M}\mathcal{O}_{V_{1}}^{q},a)$ by Proposition \ref{Well_definition_of_the_transforms}. Conversely, if $H_{1}$ is permissible
	for $(\mathscr{M}\mathcal{O}_{V_{1}}^{q},a)$, then by Proposition \ref{the_largest_integer_n_such_that_Diff(M)_is_included_in_I(H)^n}
	$\Diff_{V_{1},+}^{q-1}(\mathscr{M}\mathcal{O}_{V_{1}}^{q})\subseteq
	\mathcal{I}(H_{1})^{qa}$. This implies by Proposition \ref{Giraud's_lemma:_Global_version} that
	$\mathcal{I}(H_{1})^{q-1}(\Diff_{V,+}^{q-1}(\mathscr{M})\mathcal{O}_{V_{1}})
	\subseteq\mathcal{I}(H_{1})^{qa}$, that is,
	$\Diff_{V,+}^{q-1}(\mathscr{M})\mathcal{O}_{V_{1}}\subseteq \mathcal
	{I}(H_{1})^{q(a-1)+1}$.
	Therefore
	$\Diff_{V,+}^{q-1}(\mathscr{M})\subseteq\mathcal{I}(Z)^{q(a-1)+1}$, and
	hence by Proposition \ref{characterization_of_permissible_centers_for_a_pair_(M,a)} $Z$ is
	permissible
	for $(\mathscr{M},a)$.
\end{proof}
\begin{parrafo}%
	\label{par416}
	Fix an $\mathcal{O}_{V}^{q}$-submodule
	$\mathscr{M}\subseteq\mathcal{O}_{V}$. Then for each $x\in V$, we
	get a
	local ring $\mathcal{O}_{V,x}$ and an $\mathcal{O}_{V,x}^{q}$-submodule
	$\mathscr{M}_{x}\subset\mathcal{O}_{V,x}$. Therefore we have two numerical
	values
	$\nu_{x}^{(q)}(\mathscr{M}):=\nu_{m_{V,x}}^{(q)}(\mathscr{M}_{x})$ and
	$\eta_{x}(\mathscr{M}):=\eta_{m_{V,x}}(\mathscr{M}_{x})$; see
	Section \ref{On_the_definition_of_invariants_for_points_in_F(X)}. We
	get two numerical
	functions $V\to\mbvtex{N}_{0}$, namely
	$x\mapsto\nu_{x}^{(q)}(\mathscr{M})$ and
	$x\mapsto\eta_{x}(\mathscr{M})$. These functions are intrinsic to the
	$\mathcal{O}_{V}^{q}$-subalgebra generated by $\mathscr{M}$, as was observed
	in Lemma \ref{compatibility_of_the_q-order_with_the_weak_equivalence} and
	Proposition \ref{Diff(M)=Diff(O^q[M])}. Note that
	\begin{equation*}
		\eta_{x}(\mathscr{M})=\min\{\nu_{x}(
		\Diff_{V,+}^{i}(\mathscr{M}))+i: i=1,\ldots,q-1\},
	\end{equation*}
	and from this expression we see that the function
	$x\mapsto\eta_{x}(\mathscr{M})$ is upper-semicontinuous. Indeed, by
	Proposition \ref{Jacobian_criterion:_global_version} it is the minimum of
	upper semicontinuous
	functions.
	
	Lemma \ref{Alternative_Lemma} shows that
	$\nu_{x}^{(q)}(\mathscr{M})\leq\eta_{x}(\mathscr{M})$ for all
	$x\in V$, and the equality holds, for instance, when
	$\nu_{x}^{(q)}(\mathscr{M})$ is not divisible by $q$. The same lemma implies
	that if $V$ is a regular variety over a perfect field, then
	$\nu_{x}^{(q)}(\mathscr{M})= \eta_{x}(\mathscr{M})$ at any closed point
	$x\in V$. This shows that if for some nonclosed point $y$, it happens that
	$\nu_{y}^{(q)}(\mathscr{M})< \eta_{y}(\mathscr{M})$, then the function
	$x\mapsto\nu_{x}^{(q)}(\mathscr{M})$ cannot be upper-semicontinuous, since
	any two upper-semicontinuous functions on $V$ that coincide at closed points
	have to be the same. As an example of this phenomenon, we mention the Whitney
	umbrella. This is defined by the $(k[X,Y])^{q}$-submodule $M$ of
	$k[X,Y]$ generated by $X^{q}Y$ (here $k$ is a perfect field). We readily
	check that $\eta_{(X)}(M)=q+1$, whereas $\nu_{(X)}^{(q)}(M)=q$.
\end{parrafo}
The next proposition provides another description for
$\operatorname{Sing}(\mathscr{M},a)$ (see Proposition \ref{d(F(X))_is_closed}), this time by using $\eta_{x}(\mathscr{M})$.
\begin{proposition}%
	\label{Describing_Sing(M,a)_with_eta(M)}
	Let $\mathscr{M}\subset\mathcal{O}_{V}$ be an $\mathcal{O}_{V}^{q}$-module.
	Then for any integer $a\geq0$, we have that
	$\{x\in V:\eta_{x}(\mathscr{M})\geq qa\}=\operatorname
	{Sing}(\mathscr{M},a)$ (Definition \ref{permissible_center_for_(M,a)}). In particular, an irreducible regular
	subscheme $Z\subset V$ is permissible for $(\mathscr{M},a)$ if and only
	if $\eta_{x}(\mathscr{M})\geq qa$ for all $x\in Z$ (see Definition \ref{permissible_center_for_(M,a)}).
\end{proposition}
\begin{proof}
	It follows from Lemma \ref{Alternative_Lemma} that
	$\{x\in V:\eta_{x}(\mathscr{M})\geq qa\}=\{x\in V:\nu_{x}^{(q)}(
	\mathscr{M})\geq qa\}$, and the latter is clearly
	$\operatorname{Sing}(\mathscr{M},a)$ by the definition of
	$\eta_{x}(\mathscr{M})$. The last conclusion is clear from the definition.
\end{proof}
\begin{remark}
	This proposition, together with Proposition \ref{the_invariant_a_in_the_blow-up} and Corollary \ref{a-transform_as_sequence_of_blow-ups}, covers the proof of (1) and
	(2) of Theoerem \ref{invariants_1}.
\end{remark}
\begin{parrafo}%
	\label{The_roll_played_by_eta}%
	\emph{The role played by $\eta_{x}(\mathscr{M})$ in the description of
		$\delta(F_{d}(X))$ and in the definition of transformation of
		$\mathcal{O}_{V}^{q}$-modules.} Let
	$\mathscr{M}\subset\mathcal{O}_{V}$ be an $\mathcal
	{O}_{V}^{q}$-module, let
	$\delta:X\to V$ be the $V$-scheme attached to $\mathscr{M}$, and set
	$d:=[K(X):K(V)]$, the generic rank. Then
	$\delta(F_{d}(X))=\operatorname{Sing}(\mathscr{M},1)=\{x\in V:\eta_{x}(
	\mathscr{M})\geq q\}$, where the first equality was stated in Proposition \ref{0309}, and the second one follows from Proposition \ref{Describing_Sing(M,a)_with_eta(M)} with $a=1$. Therefore we can use
	the upper-semicontinuous function $x\mapsto\eta_{x}(\mathscr{M})$ to stratify
	$F_d(X)\cong\delta(F_{d}(X))$. In this regard, it is natural to
	blow up $V$ along a regular center, say $Z$, included in the closed set
	where the semi-continuous function $\eta_{x}(\mathscr{M})$ attains its
	maximum value, say $qa+b$ with $0\leq b<q$. We are only interested in the
	case where $a\geq1$ (which means that $F_{d}(X)\neq\emptyset$).
	
	Let $V\xleftarrow{\pi} V_{1}$ be the blowup along such center $Z$, and
	let $H_{1}\subset V_{1}$ be the exceptional divisor. Proposition \ref{Describing_Sing(M,a)_with_eta(M)} implies that $Z$ is permissible
	for $(\mathscr{M},a)$ but not for $(\mathscr{M},a+1)$. Then Proposition \ref{the_invariant_a_in_the_blow-up} implies that $H_{1}$ is permissible
	for $(\mathscr{M}\mathcal{O}_{V_{1}}^{q},a)$ but not for
	$(\mathscr{M}\mathcal{O}_{V_{1}}^{q},a+1)$. Finally, Corollary \ref{a-transform_as_sequence_of_blow-ups} shows that
	$\mathscr{M}_{1}^{(a)}:=(\mathcal{O}_{V_{1}}^{q}+\mathscr{M}\mathcal
	{O}_{V_{1}}^{q}:F^{e}(
	\mathcal{I}(H_{1})^{a}))$ is the ``right'' notion of transform of
	$\mathscr{M}$ in the sense that if $\delta_{a}:X_{a}\to V_{1}$ is the
	associated
	$V_{1}$-scheme in $\mathscr{C}_{q}(V_{1})$, then
	$\delta_a(F_{d}(X_{a}))$ does not include $H_{1}$. However, we should mention
	that with this notion of transformation, the function $\eta$ does not
	always satisfy the pointwise inequality:
	$\eta_{\pi(x_{1})}(\mathscr{M})\geq\eta_{x_{1}}(\mathscr{M}_{a})$ for
	$x_{1}\in V_{1}$, that is, there is no analog to Proposition \ref{the_order_function_of_the_transform_of_an_ideal}; see the example
	discussed in Remark \ref{eta_does_not_satisfy_point-wise_inequality}. This
	problem is due essentially to the fact that although we may factorize
	$F^{e}(\mathcal{I}(H_{1})^{a})$ from
	$\mathscr{M}\mathcal{O}_{V_{1}}^{q}$, we cannot factorize
	$\mathcal{I}(H_{1})^{b}$ in a natural way, as the latter is an ideal and
	$\mathscr{M}\mathcal{O}_{V_{1}}^{q}$ is an $\mathcal{O}_{V_{1}}^{q}$-module.
	However, we seem to recover $b$, as an invariant on the $q$-modules, later
	on, when we consider logarithmic differential operators.
\end{parrafo}

\section{On $q$-Differential Collections of Ideals}
\label{Section_on_q-Diff}

In this final section, we introduce in \S\ref{Definitions_and_basic_properties} the notion of $q$-differential collection
of ideals and their transformations by blowups. These collections have
attached numerical invariants satisfying the pointwise inequality. In
\S\ref{Logarithmic_differential_invariants}, we study
$q$-differential collections
that arise from evaluating logarithmic differential operators with respect
to a set of hypersurfaces with normal crossings on
$\mathcal{O}_{V}^{q}$-submodules $\mathscr{M}\subset\mathcal
{O}_{V}$. As
a result, we obtain invariants associated with $\mathcal{O}_{V}^{q}$-modules
and hypersurfaces with normal crossings, and these invariants satisfy the
pointwise inequality.

Along the section,  $V$ denotes, as usual in this paper, an irreducible
$F$-finite regular scheme, and $q=p^{e}$ a fixed power of $p$.

\subsection{Definitions and Basic Properties}
\label{Definitions_and_basic_properties}

We introduce the main algebraic object of the section.
\begin{definition}%
	\label{qdiffcoll}
	A collection $\mathcal{G}=(\mathcal{I}_{1},\ldots,\mathcal
	{I}_{q-1})$ of
	$q-1$ $\mathcal{O}_{V}$-ideals is said to be \emph{$q$-differential} if
	\begin{equation*}
		\Diff_{V}^{j}(\mathcal{I}_{i})\subseteq
		\mathcal{I}_{i+j} \quad \text
		{whenever } i+j\leq q-1.
	\end{equation*}
	In particular,
	$\mathcal{I}_{i}=\Diff_{V}^{0}(\mathcal{I}_{i})\subseteq\Diff_{V}^{1}(
	\mathcal{I}_{i})\subseteq\mathcal{I}_{i+1}$ for all $i=1,\ldots
	,q-2$. We
	say that $\mathcal{G}$ is nonzero and write $\mathcal{G}\neq0$ if
	$\mathcal{I}_{q-1}\neq0$. If
	$\mathcal{G}'=(\mathcal{I}_{1}',\ldots,\mathcal{I}_{q-1}')$ is a second
	$q$-differential collection of $\mathcal{O}_{V}$-ideals, then we write
	$\mathcal{G}\subseteq\mathcal{G'}$ whenever
	$\mathcal{I}_{i}\subseteq\mathcal{I}_{i}'$ for all $1\leq i\leq q-1$.
\end{definition}

As a main example (we will see others later, using logarithmic differential
operators), given an $\mathcal{O}_{V}^{q}$-module $\mathscr{M}$, we set
\begin{equation*}
	\mathcal{G}({\mathscr{M}}):=(\Diff_{V,+}^{1}(
	\mathscr{M}),\Diff_{V,+}^{2}( \mathscr{M}),\ldots,
	\Diff_{V,+}^{q-1}(\mathscr{M})).
\end{equation*}
This collection is $q$-differential, as the composition of a differential
operator of degree $i$ with one of degree $j$ is another one of degree
$i+j$. Note that by Proposition \ref{characterization_of_O^q}
$\mathcal{G}({\mathscr{M}})\neq0$ if and only if
$\mathscr{M}\nsubseteq\mathcal{O}_{V}^{q}$. Note also that
$\mathcal{G}(\mathscr{M})=\mathcal{G}(\mathcal{O}_{V}^{q}[\mathscr
{M}])$ by
Proposition \ref{Diff(M)=Diff(O^q[M])}.
\begin{example}
	Let $V=\Spec(\mathbb{F}_{3}[x_{1},x_{2},x_{3},x_{4},x_{5}])$ and
	$q=p=3$. We consider
	$\mathscr{M}:=\mathcal{O}_{V}^{3}\cdot x_{1}x_{2}x_{3}x_{4}x_{5}$. Then
	\begin{equation*}
		\mathcal{G}(\mathscr{M})= \biggl( \biggl\langle\frac
		{x_{1}x_{2}x_{3}x_{4}x_{5}}{x_{j}}: j=1,
		\ldots,5 \biggr\rangle, \biggl\langle\frac{x_{1}x_{2}x_{3}x_{4}x_{5}}{x_{i}x_{j}}: 1\leq i<j \leq5 \biggr
		\rangle \biggr)
	\end{equation*}
\end{example}

\begin{definition}%
	\label{eta}
	Given a $q$-differential collection
	$\mathcal{G}=(\mathcal{I}_{1},\ldots,\mathcal{I}_{q-1})$, for each
	$x\in V$, we define
	\begin{align*}
		\eta_{x}(\mathcal{G})&:=\min\{\nu_{x}(
		\mathcal{I}_{i})+i:1\leq i \leq q-1\}\in\mathbb{N}.
	\end{align*}
\end{definition}
Observe that if $\mathcal{G}\neq0$, then $\mathcal{I}_{q-1}\neq0$, and
so
$\eta_{x}(\mathcal{G})\leq\nu_{x}(\mathcal{I}_{q-1})+q-1<\infty$ for
all $x\in V$. Note also that $\eta_{x}(\mathcal{G})\geq1$ for all
$x\in V$. Finally, the function $x\mapsto\eta_{x}(\mathcal{G})$ from
$V$ to $\mbvtex{N}_{0}$ (with the usual order topology) is upper-semicontinuous
since it is the minimum of the upper-semicontinuous functions
$x\mapsto\nu_{x}(\mathcal{I}_{i})+i$, $i=1,\ldots,q-1$ (see Proposition \ref{Jacobian_criterion:_global_version}).

In the case $\mathcal{G}=\mathcal{G}(\mathscr{M})$, we obtain
$\eta_{x}(\mathcal{G}(\mathscr{M}))=\eta_{x}(\mathscr{M})$ (see
\ref{par416}).
\begin{example}
	In the previous example the maximum value of
	$\eta_{x}(\mathcal{G}(\mathscr{M}))=\eta_{x}(\mathscr{M})$ is $5$,
	and this
	maximum is attained only when $x$ is the origin.
\end{example}

We will use the next technical lemma to show that the function
$\eta_{x}(\mathcal{G})$ satisfies the pointwise inequality, with a suitable
definition of transformation of $q$-differential collection by blowups.
\begin{lemma}%
	\label{technical_lemma_1}
	Let $\mathcal{G}=(\mathcal{I}_{1},\ldots,\mathcal{I}_{q-1})$ be a nonzero
	$q$-differential collection of ideals on $V$. Given $x\in V$, we write
	$\eta_{x}(\mathcal{G})=aq+b$ with $0\leq b<q$. Then there is an index
	$i\in\{1,\ldots,q-1\}$ such that $i\geq b $ and
	$\nu_{x}(\mathcal{I}_{i})=\eta_{x}(\mathcal{G})-i$.
\end{lemma}
\begin{proof}
	By definition $\nu_{x}(\mathcal{I}_{i})\geq\eta_{x}(\mathcal
	{G})-i$ for
	all $i=1,\ldots,q-1$, and for some $i$, the equality holds. Let
	$i_{0}$ be the largest $i$ such that
	$\nu_{x}(\mathcal{I}_{i})= \eta_{x}(\mathcal{G})-i$. The claim of
	the lemma
	is that $i_{0}\geq b$. This is obvious if $b=0$, so we assume that
	$b>0$. Suppose on the contrary that $i_{0}<b$, so that
	$\nu_{x}(\mathcal{I}_{i_{0}})=qa+(b-i_{0})$ and $0< b-i_{0}<q$. We apply
	Lemma \ref{Alternative_Lemma} with an element
	$f\in(\mathcal{I}_{i_{0}})_{x}$ of order $qa+(b-i_{0})$. Note that we are
	in Case 1 of that lemma, and hence
	$\nu_{x}(\Diff_{V}^{b-i_{0}}(\mathcal{I}_{i_{0}}))=qa$. Finally,
	\begin{equation*}
		qa=\eta_{x}(\mathcal{G})-b\leq\nu_{x}(
		\mathcal{I}_{b})\leq\nu_{x}( \Diff_{V}^{b-i_{0}}(
		\mathcal{I}_{i_{0}}))=qa,
	\end{equation*}
	where the second inequality follows from the inclusion
	$\Diff_{V}^{b-i_{0}}(\mathcal{I}_{i_{0}})\subseteq\mathcal{I}_{b}$. This
	says that $\nu_{x}(\mathcal{I}_{b})=\eta_{x}(\mathcal{G})-b$ if
	$i_{0}<b$, which is in contradiction with the definition of $i_{0}$. Thus
	$i_{0}\geq b$, and the proof is complete.
\end{proof}
\begin{corollary}%
	\label{technical_lemma_2}
	Let $\mathcal{G}=(\mathcal{I}_{1},\ldots,\mathcal{I}_{q-1})$ be a
	$q$-differential
	collection of ideals on $V$. Fix $x\in V$ and write
	$\eta_{x}(\mathcal{G})=qa+b$ with $0\leq b<q$. Then
	\begin{equation*}
		qa+b=\eta_{x}(\mathcal{G})\leq\nu_{x}(
		\mathcal{I}_{q-1})+q-1<q(a+1).
	\end{equation*}
\end{corollary}
\begin{proof}
	The first inequality follows from the definition of
	$\eta_{x}(\mathcal{G})$. We now show the second (strict) inequality. By
	Lemma \ref{technical_lemma_1} there exists an index $i\geq b$ such that
	$\nu_{x}(\mathcal{I}_{i})=qa+b-i\leq qa$, whence
	$\nu_{x}(\mathcal{I}_{q-1})\leq qa$ since
	$\mathcal{I}_{i}\subseteq\mathcal{I}_{q-1}$. This implies that
	$\nu_{x}(\mathcal{I}_{q-1})+q-1\leq qa+q-1<q(a+1)$, as was to be shown.
\end{proof}

The next lemma provides a tool to construct examples of $q$-differential
collections of ideals from a given one. They will be used, for example,
in Definition \ref{the_transform_of_G} to give different notions of
transformation
of a $q$-differential collection of ideals under a blowup.
\begin{lemma-definition}
	Fix a $q$-differential collection of ideals
	$\mathcal{G}=(\mathcal{I}_{1},\ldots, \mathcal{I}_{q-1})$ on $V$.
	\begin{enumerate}[(2)]
		\item[(1)] If $\mathcal{L}\subset\mathcal{O}_{V}$ is an invertible
		ideal, then
		the collection of ideals
		\begin{equation*}
			\mathcal{G}_{\mathcal{L}}:=((\mathcal{I}_{1}:\mathcal
			{L}^{q}),\ldots,( \mathcal{I}_{q-1}:\mathcal{L}^{q}))
		\end{equation*}
		is also $q$-differential.
		\item[(2)] Given a second $F$-finite regular scheme $V_{1}$, and given
		a morphism
		$\pi:V_{1}\rightarrow V$, the collection
		\begin{equation*}
			\mathcal{G}\mathcal{O}_{V_{1}}:=(\mathcal{I}_{1}
			\mathcal{O}_{V_{1}}, \ldots,\mathcal{I}_{q-1}
			\mathcal{O}_{V_{1}})
		\end{equation*}
		is also $q$-differential.
	\end{enumerate}
\end{lemma-definition}
\begin{proof}
	(1) Given integers $i\geq1$ and $j\geq0$ with $i+j\leq q-1$, we have
	$\mathcal{L}^{q}\Diff_{V}^{j}((\mathcal{I}_{i}:\mathcal
	{L}^{q}))=\Diff_{V}^{j}(F^{e}
	\mathcal{L}(\mathcal{I}_{i}:\mathcal{L}^{q}))\subseteq\Diff_{V}^{j}(
	\mathcal{I}_{i})\subseteq\mathcal{I}_{i+j}$, and hence
	$\Diff_{V}^{j}((\mathcal{I}_{i}:\mathcal{L}^{q}))\subseteq(\mathcal
	{I}_{i+j}:
	\mathcal{L}^{q})$. This proves that $\mathcal{G}_\mathcal{L}$ is a
	$q$-differential
	collection.
	
	(2) For integers $i\geq1$ and $j\geq0$ with $i+j\leq q-1$, Proposition \ref{comparison_between_Diff_of_two_regular_schemes}(2) implies that
	$\Diff_{V_{1}}^{j}(\mathcal{I}_{i}\mathcal{O}_{V_{1}})\subseteq
	\Diff_{V}^{j}(
	\mathcal{I}_{i})\mathcal{O}_{V_{1}}\subseteq\mathcal
	{I}_{i+j}\mathcal{O}_{V_{1}}$.
	This proves that $\mathcal{G}\mathcal{O}_{V_{1}}$ is also a $q$-differential
	collection.
\end{proof}

We now formulate notions of transformation by blowups.
\begin{definition}%
	\label{the_transform_of_G}
	Let $\mathcal{G}=(\mathcal{I}_{1},\ldots,\mathcal{I}_{q-1})$ be a
	$q$-differential
	collection of $\mathcal{O}_{V}$-ideals, and let
	$V\xleftarrow{\pi} V_{1}\supset H_{1}$ be the blowup of $V$ along an
	irreducible
	regular center $Z$.
	\begin{enumerate}[(2)]
		\item[(1)] We call $\mathcal{G}\mathcal{O}_{V_{1}}$ the
		\emph{total transform} of $\mathcal{G}$ by the blowup.
		\item[(2)] For a positive integer $a$, we call
		$(\mathcal{G}\mathcal{O}_{V_{1}})_{\mathcal{I}(H_{1})^{a}}$ the
		\emph{$a$-transform} of $\mathcal{G}$.
	\end{enumerate}
\end{definition}
The following theorem, stated in Introduction as Theorem \ref{point-wise_inequality_for_q-diff}, establishes the fundamental pointwise
inequality for the invariant $\eta$ attached to $q$-differential collections.
\begin{theorem}%
	\label{fundamental_point-wise_inequality_for_eta(G)}
	Let $V\xleftarrow{\pi} V_{1}\supset H_{1}$ be a blowup along an irreducible
	regular center~$Z$. Assume, in addition, that
	$\eta_{x}(\mathcal{G})$ is constant along points in $Z$, say
	$\eta_{x}(\mathcal{G})=aq+b$ for all $x\in Z$, where $0\leq b<q$. If
	$\mathcal{G}_{1}^{(a)}$ denotes the $a$-transform of $\mathcal{G}$, then
	for any $x_{1}\in V_{1}$,
	\begin{equation*}
		\eta_{\pi(x_{1})}(\mathcal{G}) \geq\eta_{x_{1}}(
		\mathcal{G}_{1}^{(a)}).
	\end{equation*}
\end{theorem}
\begin{proof}
	We only need to prove this inequality for points in $H_{1}$. Fix
	$x_{1}\in H_{1}$ and set $x:=\pi(x_{1})\in Z$, so that
	$\eta_{x}(\mathcal{G})=aq+b$. According to Lemma \ref{technical_lemma_1}, there is an index $i$ with
	$b\leq i\leq q-1$ such that $\nu_{x}(\mathcal{I}_{i})=aq+b-i$. We fix such
	an index $i$. Since by hypothesis $\eta_{z}(\mathcal{G})=aq+b$ for all
	$z\in Z$, we have
	$\nu_{z}(\mathcal{I}_{i})\geq\eta_{z}(\mathcal{G})-i=qa+b-i$ for all
	$z\in Z$, and the equality holds for $z=x$. As the order function is
	upper-semicontinuous
	(Proposition \ref{Jacobian_criterion:_global_version}), after restricting
	to an open neighborhood of $x$, we may assume that
	$\nu_{z}(\mathcal{I}_{i})=qa+b-i$ for all $z\in Z$. So, for our fixed index
	$i$,
	$\mathcal{I}_{i}\mathcal{O}_{V_{1}}=\mathcal
	{I}(H_{1})^{aq+b-i}\mathcal{J}_{i}$
	for an $\mathcal{O}_{V_{1}}$-ideal $\mathcal{J}_{i}$. By Proposition \ref{the_order_function_of_the_transform_of_an_ideal},
	$\nu_{x_{1}}(\mathcal{J}_{i})\leq\nu_{x}(\mathcal{I}_{i})=qa+b-i$. Finally,
	\begin{equation*}
		(\mathcal{I}_{i}\mathcal{O}_{V_{1}}:\mathcal
		{I}(H_{1})^{aq})=(\mathcal{I}(H_{1})^{aq+b-i}
		\mathcal{J}_{i}:\mathcal{I}(H_{1})^{aq})
		\supseteq\mathcal{J}_{i},
	\end{equation*}
	where the last inclusion holds since $i\geq b$. We conclude that
	$\nu_{x_{1}}(\mathcal{I}_{i}\mathcal{O}_{V_{1}}:\mathcal{I}(H_{1})^{aq})
	\leq\nu_{x_{1}}(\mathcal{J}_{i})\leq aq+b-i$, and by definition this implies
	that $\eta_{x_{1}}(\mathcal{G}_{1}^{(a)})\leq aq+b$, as was to be proved.
\end{proof}

\begin{proposition}%
	\label{G(M_1)_and_(G_M)_1}
	Fix an $\mathcal{O}_{V}^{q}$-submodule
	$\mathscr{M}\subseteq\mathcal{O}_{V}$, an integer $a\geq1$, and an
	irreducible
	regular subscheme $Z\subset V$. Consider the blowup
	$V\leftarrow V_{1}\supset H_{1}$ of $V$ along $Z$. Let
	$\mathscr{M}_{1}^{(a)}$ and $(\mathcal{G}(\mathscr{M}))_{1}^{(a)}$ denote,
	respectively, the $a$-transform of $\mathscr{M}$ (Def.
	\ref{the_a-transform_of_an_O^q-module}) and the $a$-transform of
	$\mathcal{G}({\mathscr{M}})$ (Def. \ref{the_transform_of_G}). Then
	\begin{align*}
		\mathcal{G}({\mathscr{M}_{1}}^{(a)})\subseteq(
		\mathcal{G}(\mathscr {M}))_{1}^{(a)}.
	\end{align*}
	In particular,
	$\eta_{x_{1}}(\mathscr{M}_{1}^{(a)})\geq\eta_{x_{1}}((\mathcal{G}(
	\mathscr{M}))_{1}^{(a)})$ for all $x_{1}\in V_{1}$.
\end{proposition}
\begin{proof}
	By the definition of the $a$-transform of a module there is an inclusion
	$F^{e}(\mathcal{I}(H_{1})^{a})\mathscr{M}_{1}\subseteq\mathscr{M}
	\mathcal{O}_{V_{1}}^{q}+\mathcal{O}_{V_{1}}^{q}$, and hence for each
	$i=1,\ldots,q-1$,
	\begin{align*}
		\mathcal{I}(H_{1})^{qa}\Diff_{V_{1},+}^{i}(
		\mathscr{M}_{1})&=\Diff _{V_{1},+}^{i}(F^{e}(
		\mathcal{I}(H_{1})^{a})\mathscr{M}_{1})
		\\
		&\subseteq\Diff_{V_{1},+}^{i}( \mathscr{M}
		\mathcal{O}_{V_{1}}^{q})\subseteq\Diff _{V,+}^{i}(
		\mathscr{M}) \mathcal{O}_{V_{1}},
	\end{align*}
	{\spaceskip=0.2em plus 0.05em minus 0.04em where the last inclusion follows from} Proposition \ref{comparison_between_Diff_of_two_regular_schemes}. Therefore
	$ \Diff_{V_{1},+}^{i}(\mathscr{M}_{1})\subseteq(\Diff_{V,+}^{i}(
	\mathscr{M})\mathcal{O}_{V_{1}}: \mathcal{I}(H_{1})^{qa})$. This
	yields the
	desired inclusion of $q$-differential collections.
\end{proof}

\begin{remark}%
	\label{rk429}
	We discuss an outcome of interest for the study of transformations of
	singularities
	that follows from the previous results. We fix an
	$\mathcal{O}_{V}^{q}$-submodule $\mathscr{M}\subseteq\mathcal
	{O}_{V}$ We
	associate with $\mathscr{M}$ the $q$-differential collection
	$\mathcal{G}_{0}:=\mathcal{G}(\mathscr{M})$. Sequences of
	``permissible transformations''
	of this collection lead to ``permissible transformations'' of the module
	$\mathscr{M}$ by blowups in a way we now describe. We consider the sequences
	of blowups
	\begin{equation}
		\label{sequence_of_transformations_of_(G,a)} \xymatrix@R=0pt @C=30pt { V=V_{0} & \ar[l]_-{\pi_{1}} V_{1} & \ar[l]_-{
				\pi_{2}} \ldots& \ar[l]_-{\pi_{n}} V_{n}%
			\\
			\mathcal{G}_{0}=\mathcal{G}(\mathscr{M})& \mathcal{G}_{1} & \ldots&
			\mathcal{G}_{n} }
	\end{equation}
	that are constructed from the functions $\eta$ as follows: for each
	$i=0,1,\ldots,n-1$, $V_{i}\leftarrow V_{i+1}$ is the blowup of
	$V_{i}$ along an irreducible regular center $Z_{i}\subset V_{i}$ included
	in the set of points $x\in V_{i}$ at which
	$\eta_{x}(\mathcal{G}_{i})$ reaches its maximum value, say
	$a_{i}q+b_{i}$ ($0\leq b_{i}<q$), and $\mathcal{G}_{i+1}$ is the
	$a_{i}$-transform of $\mathcal{G}_{i}$. By Theorem \ref{fundamental_point-wise_inequality_for_eta(G)} we have
	\begin{equation*}
		a_{0}q+b_{0}\geq a_{1}q+b_{1}
		\geq\cdots\geq a_{n}q+b_{n}.
	\end{equation*}
	Assume that $a_{n-1}\geq1$. We now set $\mathscr{M}_{0}:=\mathscr
	{M}$, and
	for $i\geq0$, we define $\mathscr{M}_{i+1}$ as the $a_{i}$-transform of
	$\mathscr{M}_{i}$. It follows from Proposition \ref{G(M_1)_and_(G_M)_1} and induction on $i$ that there are inclusions
	\begin{equation*}
		\mathcal{G}(\mathscr{M}_{i})\subseteq\mathcal{G}_{i},\quad
		i=1,\ldots,n;
	\end{equation*}
	in particular,
	$\eta_{x}(\mathscr{M}_{i}):=\eta_{x}(\mathcal{G}(\mathscr{M}_{i}))
	\geq\eta_{x}(\mathcal{G}_{i})=:a_{i}q+b_{i}$ for all $x\in Z_{i}$. By
	Proposition \ref{Describing_Sing(M,a)_with_eta(M)} $Z_{i}$ is permissible
	for $(\mathscr{M}_{i},a_{i})$ for all $i=0,\ldots,n-1$, and hence the finite
	morphism $X_{i+1}\to V_{i+1}$ attached to $\mathscr{M}_{i+1}$ is obtained
	from the finite morphism $X_{i}\to V_{i}$ attached to
	$\mathscr{M}_{i}$ as $\delta_{a}$ is obtained from $\delta$ in Corollary \ref{a-transform_as_sequence_of_blow-ups}.
\end{remark}

\subsection{Logarithmic Differential Invariants}
\label{Logarithmic_differential_invariants}

In this final part of the section, we explore new invariants of singularities
that are constructed using logarithmic differential operators. These operators
arise naturally when we fix a regular scheme $V$ together with a sequence
of blowups along regular centers
$V \leftarrow\cdots\leftarrow V_{r}$. In this setting, $V_{r}$ is regular
and contains exceptional hypersurfaces, say $H_{1}, \dots, H_{r}$, and
it is useful to consider logarithmic differential operators with poles
along these hypersurfaces. Roughly speaking, we will construct $q$-differential
collections associated with modules and hypersurfaces with normal crossings
and define invariants from the function $\eta$ associated with $q$-differential
collections (Definition \ref{eta}). Our goal is to prove that these invariants
satisfy the fundamental pointwise inequality.

\begin{parrafo}
	Logarithmic differential operators were reviewed in
	\ref{logarithmic_differential_operators}. The discussion in
	\ref{compatibility_with_localization} enables us to extend everything to
	the setting of sheaves on schemes. So we fix an irreducible $F$-finite
	regular scheme $V$ and $q=p^{e}$, a power of $p$. Given an
	$\mathcal{O}_{V}^{q}$-module $\mathscr{M}\subset\mathcal{O}_{V}$, an
	$\mathcal{O}_{V}$-ideal $\mathcal{J}\subset\mathcal{O}_{V}$, and a
	collection
	of $\mathcal{O}_{V}$-ideals
	$\Lambda=\{\mathcal{I}_{1},\ldots,\mathcal{I}_{r}\}$, by (\ref
	{localization})
	it follows that there are $\mathcal{O}_{V}$-ideals
	\begin{align*}
		\Diff_{V,\Lambda}^{0}(\mathcal{J})&\subseteq
		\Diff_{V,\Lambda}^{1}( \mathcal{J})\subseteq
		\Diff_{V,\Lambda}^{2}(\mathcal{J})\subseteq \cdots,
		\\
		\Diff_{V,\Lambda,+}^{1}(\mathscr{M})&\subseteq
		\Diff_{V,\Lambda,+}^{2}( \mathscr{M})\subseteq\cdots\subseteq
		\Diff_{V,\Lambda,+}^{q-1}( \mathscr{M}),
	\end{align*}
	such that at every point $x\in V$, we have
	\begin{align*}
		(\Diff_{V,\Lambda}^{i}(\mathcal{J}))_{x}&=
		\operatorname {Diff}_{\mathcal{O}_{V,x},
			\Lambda_{x}}^{i}(\mathcal{J}_{x}),\quad
		i\geq0,
		\\
		(\Diff_{V,\Lambda,+}^{i}(\mathscr{M}))_{x}&=
		\operatorname {Diff}_{\mathcal{O}_{V,x},
			\Lambda_{x},+}^{i}(\mathscr{M}_{x}),\quad
		i=1,\ldots,q-1,
	\end{align*}
	where
	$\Lambda_{x}=\{(\mathcal{I}_{1})_{x},\ldots,(\mathcal{I}_{r})_{x}\}
	$. This
	enables us to reduce the verification of properties formulated for
	$\mathcal{O}_{V}^{q}$-modules, $\mathcal{O}_{V}$-ideals, and differential
	operators on $V$ at the local level.
\end{parrafo}

We are interested only in the case where $\Lambda$ consists of the ideals
of hypersurfaces with normal crossings. We review this definition.
\begin{definition}%
	\label{normal_crossings}
	A collection of hypersurfaces $\{H_{1},\ldots,H_{r}\}$ on $V$ is said
	to have only normal crossings in $V$ if for each point $z\in V$, there
	exists a regular system of parameters $(x_{1},\ldots,x_{n})$ for
	$\mathcal{O}_{V,z}$ such that for each $j=1,\ldots,r$, the ideal
	$\mathcal{I}(H_{j})_{z}\subseteq\mathcal{O}_{V,z}$ is either
	$\mathcal{O}_{V,z}$ or is $\langle x_{i_{j}}\rangle$ for some
	$i_{j}\in\{ 1,\ldots,n\}$. If in addition $Z\subset V$ is a closed regular
	subscheme, we say that $\{H_{1},\ldots,H_{r}\}$ has only normal crossings
	with $Z$ (or that $Z$ has only normal crossings with
	$\{H_{1},\ldots,H_{r}\}$) if for each $z\in Z$, there exists a regular
	system of parameters $(x_{1},\ldots,x_{n})$ for $\mathcal{O}_{V,z}$ satisfying
	the above condition and such that $\mathcal{I}(Z)_{z}$ is generated by a
	subset of $\{x_{1},\ldots,x_{n}\}$; see
	\cite[Definition 2, p. 141]{hironaka1964resolution}.
\end{definition}
We include the following proposition to show the role played by the ideal
$\Diff_{V,\Lambda,+}^{q-1}(\mathscr{M})$ in the characterization of the
``optimal monomial'' that can be extracted from the equivalence class of
$\mathscr{M}$ in the following sense.
\begin{proposition}
	Let $\{H_{1},\ldots,H_{r}\}$ be a collection of hypersurfaces with only
	normal crossings, and let $\Lambda$ be the collection of their ideals.
	Then for nonnegative integers $m_{1},\ldots,m_{r}$,
	\begin{eqnarray*}
		\Diff_{V,\Lambda,+}^{q-1}(\mathscr{M})&\subseteq&
		\mathcal{I}(H_{1})^{m_{1}} \cdots\mathcal{I}(H_{r})^{m_{r}}
		\quad \text{if and only if }
		\\
		\mathscr{M} &\subseteq&\mathcal{O}_{V}^{q}+
		\mathcal{I}(H_{1})^{m_{1}}\cdots \mathcal{I}(H_{r})^{m_{r}}.
	\end{eqnarray*}
\end{proposition}
\begin{proof}
	The ``if'' part is immediate from the definition of logarithmic differential
	operator. The other direction can be checked locally, and we only need
	to prove the following statement: \emph{Let $(R,\mvtex)$ be an $F$-finite
		regular local ring, let $(x_{1},\ldots,x_{n})$ be a regular system of
		parameters, and let $\Lambda=\{x_{1}R,\ldots,x_{r}R\}$ for some
		$r\leq n$. Given $f\in R$, assume that
		$\operatorname{Diff}_{R,\Lambda,+}^{q-1}(f)\in x_{1}^{m_{1}}\cdots
		x_{r}^{m_{r}}R$
		for integers $m_{1},\ldots,m_{r}\geq0$. Then
		$f\in R^{q}+x_{1}^{m_{1}}\cdots x_{r}^{m_{r}} R$.}
	
	We extend $(x_{1},\ldots,x_{n})$ to a $p$-basis
	$(x_{1},\ldots,x_{n},y_{1},\ldots,y_{n'})$ for $R$ and consider the
	$q$-expansion of $f$, say
	$f=\sum_{\mathcal{A}_{q}} c_{\alpha,\beta}^{q}\xvtex^{\alpha}
	\yvtex^{\beta}$. We will show that given
	$(\alpha,\beta)\neq(0,0)$ and $i\leq r$, the term
	$c_{\alpha,\beta}^{q}\xvtex^{\alpha}\yvtex^{\beta}$ is divisible
	by $x_{i}^{m_{i}}$.
	We can replace $R$ by its localization at the regular prime
	$\langle x_{i}\rangle$ since the $p$-basis and the $q$-expansion remain
	unaltered. Thus we may assume that $i=r=n=1$ Now
	$x_{1}^{\alpha_{1}}
	\frac{\partial^{\alpha_{1}}}{\partial x_{1}^{\alpha_{1}}}\in
	\operatorname{Diff}_{R,\Lambda,+}^{q-1}$, and
	$x_{1}^{\alpha_{1}}
	\frac{\partial^{\alpha_{1}}}{\partial x_{1}^{\alpha_{1}}}(f)$ has
	$c_{\alpha,\beta}^{q}\xvtex^{\alpha}\yvtex^{\beta}$ as a term in its
	$q$-expansion. Since
	$x_{1}^{\alpha_{1}}
	\frac{\partial^{\alpha_{1}}}{\partial x_{1}^{\alpha_{1}}}(f)$ is divisible
	by $x_{1}^{m_{1}}$ (this is the assumption), Proposition \ref{order_in_terms_of_q-atoms} implies that
	$c_{\alpha,\beta}^{q}\xvtex^{\alpha}\yvtex^{\beta}$ is divisible by
	$x_{1}^{m_{1}}$.
\end{proof}
We do not pursue in this line, namely the factorization of an ``optimal
monomial'' and the behavior under suitable blowups, and refer the reader
to \cite{moh1996newton} and \cite{HauPer2019}. We instead try to produce
a $q$-differential collection attached to $\mathscr{M}$ and
$\Lambda$ so that we can apply the results of the first part of the section
(mainly Theorem \ref{fundamental_point-wise_inequality_for_eta(G)}) and
obtain numeral invariants with a good behavior under blowups.

A natural step would be to work with the collection of
$\mathcal{O}_{V}$-ideals
\begin{equation*}
	(\Diff_{V,\Lambda,+}^{1}(\mathscr{M}),\ldots,
	\Diff_{V,\Lambda,+}^{q-1}( \mathscr{M})).
\end{equation*}
By Proposition \ref{Diff(M)=Diff(O^q[M])} this sequence of ideals is not
affected if $\mathscr{M}$ is replaced by
$\mathcal{O}_{V}^{q}[\mathscr{M}]$, so it is compatible with our
notion of
weak equivalence of $\mathcal{O}_{V}^{q}$-submodules. However, this sequence
is not (in general) a $q$-differential collection in the sense of Definition \ref{qdiffcoll}. This is due to the fact that there is no inclusion
$\Diff_{V}^{j} \Diff_{V,\Lambda,+}^{i} \subset\Diff_{V,\Lambda,+}^{i+j}$
in general. We will use the following lemmas to produce $q$-differential
collections from the previous collection (Proposition \ref{Definition_of_G_M(Lambda,L)}).
\begin{lemma}[{\cite[Lemma 1.5]{hironaka2005three}}]%
	\label{Hironaka's_lemma}%
	Let $R$ be a ring, and let
	$x\in R$ be an element that is not a zero divisor. Given integers
	$a,b\geq0$ and $D\in\operatorname{Diff}_{R}^{a}$, we have
	$x^{a-b} D x^{b}\in\sum_{k=0}^{a} x^{k}\operatorname{Diff}_{R}^{k}$.
	In particular,
	$x^{a-b}Dx^{b}$ is $(x)$-logarithmic. ($Dx^{b}$ denotes the operator
	$Dx^{b}(r)=D(x^{b} r)$, $r\in R$.)
\end{lemma}
\begin{proof}
	We proceed by induction on $ab$. If $ab=0$, then either $a=0$, in which
	case $D$ is $R$-linear and $x^{-b}Dx^{b}=D\in\operatorname
	{Diff}_{R}^{0}$, or
	else $b=0$, in which case the result is trivial. Assume now that
	$ab>0$, whence $a>0$ and $b>0$. Then
	\begin{eqnarray*}
		x^{a-b} D x^{b}&=& x^{a-b} (D
		x)x^{b-1}
		\\
		&=&x^{a-b}(Dx-xD+xD)x^{b-1}
		\\
		&=&x^{(a-1)-(b-1)}(Dx-xD)x^{b-1}+x^{a-(b-1)}
		D x^{b-1}.
	\end{eqnarray*}
	The first term in the last expression is in
	$\sum_{k=0}^{a-1}x^{k}\operatorname{Diff}_{R}^{k}$ by the inductive hypothesis
	since $(a-1)(b-1)<ab$ and $Dx-xD\in\operatorname{Diff}_{R}^{a-1}$. Similarly,
	the second term is in $\sum_{k=0}^{a}x^{k}\operatorname
	{Diff}_{R}^{k}$ since
	$a(b-1)<ab$. Therefore
	$x^{a-b} D x^{b}\in\sum_{k=0}^{a}x^{k}\operatorname{Diff}_{R}^{k}$.
	The induction
	is now complete.
\end{proof}
\begin{lemma}%
	\label{technical_lemma_3}
	Fix $\mathcal{O}_{V}$-ideals
	$\mathcal{J}_{1}\subseteq\cdots\subseteq\mathcal{J}_{q-1}$ and an
	invertible
	ideal $\mathcal{L}$, and assume that
	\begin{equation*}
		\mathcal{L}^{j}\Diff_{V}^{j}(
		\mathcal{J}_{i})\subseteq\mathcal{J}_{i+j}
		\quad \text{whenever } i+j\leq q-1, j\geq0, i\geq1.
	\end{equation*}
	Then the collection of $\mathcal{O}_{V}$-ideals
	$((\mathcal{J}_{1}:\mathcal{L}^{1}),\ldots,(\mathcal
	{J}_{q-1}:\mathcal{L}^{q-1}))$
	is $q$-differential (\ref[Definition]{qdiffcoll}).
\end{lemma}
\begin{proof}
	Observe first that
	\begin{align}
		\label{formula_log} \mathcal{L}^{i+j}\Diff_{V}^{j}((
		\mathcal{J}_{i}:\mathcal{L}^{i}))&= \mathcal{L}^{i+j-q}
		\mathcal{L}^{q}\Diff_{V}^{j}((
		\mathcal{J}_{i}: \mathcal{L}^{i}))\nonumber
		\\
		&=\mathcal{L}^{i+j-q}
		\Diff_{V}^{j}(\mathcal{L}^{q}(
		\mathcal{J}_{i}:\mathcal{L}^{i}))\nonumber
		\\
		&\subseteq
		\mathcal{L}^{i+j-q}\Diff_{V}^{j}(
		\mathcal{L}^{q-i}\mathcal{J}_{i}).
	\end{align}
	In the second equality, we use that
	$\mathcal{L}^{q}=F^{e}\mathcal{L}\cdot\mathcal{O}_{V}$ (since
	$\mathcal{L}$ is invertible) and that differential operators of order
	$\leq q-1$ are $\mathcal{O}_{V}^{q}$-linear. We now apply Lemma \ref{Hironaka's_lemma} with $a=j$ and $b=q-i$, and obtain that the expression
	on the right in (\ref{formula_log}) is included in
	$\sum_{k=0}^{j}\mathcal{L}^{k}\Diff_{V}^{k}(\mathcal{J}_{i})$,
	which is
	included in $\mathcal{J}_{i+j}$, by the hypothesis. Summarizing,
	$\mathcal{L}^{i+j}\Diff_{V}^{j}((\mathcal{J}_{i}:\mathcal{L}^{i}))
	\subseteq\mathcal{J}_{i+j}$, and hence
	$\Diff_{V}^{j}((\mathcal{J}_{i}:\mathcal{L}^{i}))\subseteq(\mathcal
	{J}_{i+j}:
	\mathcal{L}^{i+j})$, as was to be proved.
\end{proof}
\begin{proposition}%
	\label{Definition_of_G_M(Lambda,L)}
	Let $\Lambda=\{\mathcal{I}_{1},\ldots,\mathcal{I}_{r}\}$ be a
	family of
	nonzero $\mathcal{O}_{V}$-ideals, and let $\mathcal{L}$ be an invertible
	$\mathcal{O}_{V}$-ideal included in each $\mathcal{I}_{k}$,
	$k=1,\ldots,r$. Then for any $\mathcal{O}_{V}^{q}$-module
	$\mathscr{M}$, the collection of ideals
	\begin{align}
		\label{collection_Diff_(V,E_1,...,E_r)(M)_and_L}
		\mathcal{G}(\mathscr{M},\Lambda,\mathcal{L}):={}&((
		\Diff_{V,\Lambda,+}^{1}( \mathscr{M}):\mathcal{L}^{1}),(
		\Diff_{V,\Lambda,+}^{2}(\mathscr{M}): \mathcal{L}^{2}),
		\ldots,\nonumber
		\\&{}(\Diff_{V,\Lambda,+}^{q-1}(\mathscr{M}):
		\mathcal{L}^{q-1}))
	\end{align}
	is $q$-differential. In addition, there is a componentwise inclusion
	\begin{equation}
		\label{collection_Diff_(V,E_1,...,E_r)(M)_and_LXX} \mathcal{G}(\mathscr{M})\subseteq\mathcal{G}(\mathscr{M},
		\Lambda, \mathcal{L}).
	\end{equation}
\end{proposition}
\begin{proof}
	By (\ref{LDiff_is_Lambda_logarithmic}) there are inclusions
	$\mathcal{L}^{i}\Diff_{V}^{i}\subseteq\Diff_{V,\Lambda}^{i}$ and $\mathcal{L}^{i}\Diff_{V,+}^{i}\subseteq\Diff_{V,
		\Lambda,+}^{i}$. Thus
	\begin{equation*}
		\mathcal{L}^{j}\Diff_{V}^{j}(
		\Diff_{V,\Lambda,+}^{i}(\mathscr{M})) \subseteq
		\Diff_{V,\Lambda}^{j}(\Diff_{V,\Lambda,+}^{i}(
		\mathscr{M})) \subseteq\Diff_{V,\Lambda,+}^{i+j}(\mathscr{M}).
	\end{equation*}
	This shows that $\mathcal{L}$ and the collection
	$\Diff_{V,\Lambda,+}^{1}(\mathscr{M})\subseteq\cdots\subseteq
	\Diff_{V,\Lambda,+}^{q-1}(\mathscr{M})$ satisfy the hypothesis of Lemma \ref{technical_lemma_3}, whence the collection
	$\mathcal{G}(\mathscr{M},\Lambda,\mathcal{L})$ is $q$-differential.
	As for
	(\ref{collection_Diff_(V,E_1,...,E_r)(M)_and_LXX}), we note that for
	$i=1,\ldots,q-1$,
	\begin{equation*}
		\Diff_{V,+}^{i}(\mathscr{M})=(\mathcal{L}^{i}
		\Diff _{V,+}^{i}(\mathscr{M}): \mathcal{L}^{i})
		\subseteq(\Diff_{V,\Lambda,+}^{i}(\mathscr{M}):
		\mathcal{L}^{i}).
	\end{equation*}
\end{proof}
\begin{corollary}%
	\label{inequality}
	Let the setting and notation be as in Proposition \ref{Definition_of_G_M(Lambda,L)}. Then
	\begin{align*}
		\eta_{x}(\mathcal{G}(\mathscr{M},\Lambda,\mathcal{L}))\leq
		\eta_{x}( \mathcal{G}(\mathscr{M})),\quad  \forall x\in V.
	\end{align*}
	In particular, if $Z\subset V$ is a regular irreducible subscheme and
	$a$ is a positive integer such that
	$\eta_{x}(\mathcal{G}(\mathscr{M},\Lambda,\mathcal{L}))\geq qa$
	for all
	$x\in Z$, then $Z$ is permissible for $(\mathscr{M},a)$.
\end{corollary}
\begin{proof}
	The inequality follows from (\ref{collection_Diff_(V,E_1,...,E_r)(M)_and_LXX}),
	and the last assertion follows from Proposition \ref{Describing_Sing(M,a)_with_eta(M)}.
\end{proof}
\begin{remark}
	So far we have associated with a 3-uple
	$(\mathscr{M},\Lambda,\mathcal{L})$ as above an upper-semicontinuous function
	$V\to\mbvtex{N}$ given by
	$x\to\eta_{x}(\mathcal{G}(\mathscr{M},\Lambda,\mathcal{L}))$.
	This function
	coincides with $\eta_{x}(\mathscr{M})$ outside the support of the ideals
	of $\Lambda$ and $\mathcal{L}$. Now we want to formulate a notion of
	transformation
	of triples $(\mathscr{M},\Lambda,\mathcal{L})$ by blowups and to
	show that
	the fundamental pointwise inequality holds for the function
	$x\to\eta_{x}(\mathcal{G}(\mathscr{M},\Lambda,\mathcal{L}))$. To
	do this,
	we will restrict ourselves to the case where $\Lambda$ consists of ideals
	of hypersurfaces with only normal crossings. We also require a version
	of Corollary \ref{corollary_of_Giraud_lemma} with logarithmic differential
	operators, which is our main tool to compare invariants before and after
	a blowup. We begin with the formulation of the latter task.
\end{remark}

\begin{parrafo}
	Let $R$ be an $F$-finite regular local ring, let $P\subset R$ be a regular
	prime, and let $(x_{1},\ldots,x_{n})$ be a regular system of parameters
	for $R$ such that $P=\langle x_{1},\ldots,x_{s}\rangle$ for some index
	$s\leq n$. We refer to the discussion of the blowup of
	$\operatorname{Spec}(R)$ along $P$ in \ref{blow-up} and consider
	\begin{equation*}
		R_{1}:=\biggl\{\frac{z}{x_{1}^{t}}: z\in P^{t}, t
		\geq0\biggr\}=R\biggl[ \frac{x_{2}}{x_{1}},\ldots,\frac{x_{s}}{x_{1}}\biggr]
		\subset R_{x_{1}}.
	\end{equation*}
	For each index $j=1,\ldots, n$, we set
	\[
	x_{j}'=\begin{cases}
		\frac{x_{j}}{x_{1}}& \text{if } j\leq s,
		\\
		x_{j}&\text{if } j>s.
	\end{cases}
	\]
	Notice that $x_{j}'R_{1}$ is the restriction to
	$\operatorname{Spec}(R_{1})$ of the strict transform of $x_{j}R$ by
	the blowup
	along $P$.
	
	We fix a subset $\Phi\subseteq\{1,\ldots,n\}$ and set
	$\Lambda:=\{x_{j}R: j\in\Phi\}$, which is a collection of ideals of
	$R$, and
	$\Lambda_{1}:=\{x_{j}'R_{1}: j\in\Phi\}\cup\{x_{1} R_{1}\}$, which
	is a collection of ideals in $R_{1}$. In the formulation of the following
	lemma, we identify differential operators of $R$ or of $R_{1}$ with their
	extensions to the fraction field.
\end{parrafo}

\begin{lemma}
	There are inclusions
	\begin{align}
		\label{inclusion} \operatorname{Diff}_{R,\Lambda\cup\{P\}}^{i}
		\subseteq\operatorname {Diff}_{R_{1},
			\Lambda_{1}}^{i} \quad \text{and}\quad
		\operatorname{Diff}_{R,\Lambda\cup\{P\},+}^{i} \subseteq
		\operatorname{Diff}_{R_{1},\Lambda_{1},+}^{i}
	\end{align}
	for all $i\geq0$. In particular, given an $R^{q}$-submodule
	$M\subseteq R$, there are inclusions
	\begin{align}
		\label{inclusions2} x_{1}^{i}\operatorname{Diff}_{R,\Lambda,+}^{i}(M)
		R_{1} \subseteq \operatorname{Diff}_{R_{1},
			\Lambda_{1},+}^{i}(MR_{1}^{q})
	\end{align}
	for all $i=1,\ldots,q-1$. ($MR_{1}^{q}$ denotes the $R_{1}^{q}$-submodule
	of $R_{1}$ generated by $M$.)
\end{lemma}
\begin{proof}
	We prove the first inclusion in (\ref{inclusion}), the second being just
	a consequence. Fix $D\in\operatorname{Diff}_{R,\Lambda\cup\{P\}
	}^{i}$. By Lemma
	\ref{local_Girauds_lemma} $D$ extends to a differential operator
	of $R_{1}$, and this operator is $x_{1}R_{1}$-logarithmic. To show that
	$D\in\operatorname{Diff}_{R_{1},\Lambda_{1}}^{i}$, it is only left
	to show
	that $D({x_{j}'}^{k} R_{1})\subseteq{x_{j}'}^{k} R_{1}$ for all
	$k\geq0$ and $j\in\Phi$, $j\neq1$.
	
	Fix $j\in\Phi$. Assume first that $j\leq s$, so that
	$x_{j}'=\frac{x_{j}}{x_{1}}$ and the elements of $x_{j}^{\prime
		k}R_{1}$ are of
	the form $\frac{x_{j}^{k}z}{x_{1}^{k+t}}$ where $z\in P^{t}$. After
	multiplication
	and division by a power of $x_{1}$, we may assume that $k+t$ is a power
	of $p$ that is greater than the order of $D$. Hence
	$D(x_{j}^{k}z/x_{1}^{k+t})=D(x_{j}^{k}z)/x_{1}^{k+t}$. Since $D$ is
	$x_{j}R$-logarithmic and also $P$-logarithmic, and since
	$x_{j}^{k}z\in x_{j}^{k} R\cap P^{k+t}$, the numerator of the latter fraction
	belongs to $x_{j}^{k} R\cap P^{k+t}=x_{j}^{k} P^{t}$. Thus the fraction
	belongs to $x_{j}^{\prime k} R_{1}$. This proves that
	$D({x_{j}'}^{k} R_{1})\subseteq{x_{j}'}^{k} R_{1}$ for all
	$k\geq0$.
	
	Assume now that $j>s$, so that $x_{j}'=x_{j}$, and the elements of
	$x_{j}^{\prime k} R_{1}$ are of the form $x_{j}^{k}z/x_{1}^{t}$ with
	$z\in P^{t}$. As in the previous case, we may assume that $t$ is a power
	of $p$ that is greater than the order of $D$. Then
	$D(x_{j}^{k} z/x_{1}^{t})=D(x_{j}^{k}z)/x_{1}^{t}$. Since
	$x_{j}^{k} z\in x_{j}^{k}\cap P^{t}$, and since $D$ is $x_{j} R$-logarithmic
	and $P$-logarithmic, the numerator of the latter fraction is in
	$x_{j}^{k} R\cap P^{t}=x_{j}^{k} P^{t}$. Thus the fraction belongs to
	$x_{j}^{k} R_{1}=x_{j}^{\prime k} R_{1}$. This proves that
	$D({x_{j}'}^{k} R_{1})\subseteq{x_{j}'}^{k} R_{1}$ for all
	$k\geq0$.
	
	As for the inclusion in (\ref{inclusions2}), we just need to observe that
	$x_{1}^{i}\operatorname{Diff}_{R,\Lambda,+}^{i}\subset\operatorname
	{Diff}_{R,
		\Lambda\cup\{P\},+}^{i}$, which follows from (\ref{IDiff_is_logarithmic})
	and then use the second inclusion in (\ref{inclusion}).
\end{proof}

\begin{proposition}%
	\label{Giraud's_lemma,_logarithmic_version}
	Let $\Lambda$ be a finite collection of hypersurfaces on $V$ with only
	normal crossings, and let $Z\subset V$ be an irreducible regular closed
	subscheme such that $\Lambda$ has only normal crossings with $Z$ (Definition \ref{normal_crossings}). Let $V\leftarrow V_{1}\supset H_{1}$ be the blowup
	along $Z\subset V$, and let $\Lambda_{1}$ be the collection of the strict
	transforms of each $H\in\Lambda$ plus the exceptional hypersurface
	$H_{1}\subset V_{1}$. Then $\Lambda_{1}$ has only normal crossings on
	$V_{1}$, and for any $\mathcal{O}_{V}^{q}$-module $\mathscr{M}$, we have
	\begin{equation*}
		\mathcal{I}(H_{1})^{i}(\Diff_{V,\Lambda,+}^{i}(
		\mathscr{M})\mathcal {O}_{V_{1}}) \subseteq\Diff_{V_{1},\Lambda_{1},+}^{i}(
		\mathscr{M}\mathcal {O}_{V_{1}}^{q}) \quad \text{for } i=1,
		\ldots,q-1.
	\end{equation*}
\end{proposition}
(Here and further, logarithmic differential operators with respect to a
collection of hypersurfaces mean logarithmic differential operators with
respect to the ideals of these hypersurfaces.)
\begin{proof}
	The verification that $\Lambda_{1}$ has normal crossings is straightforward.
	The above inclusion is a consequence of (\ref{inclusions2}).
\end{proof}
\begin{parrafo}%
	\label{logarithmic_setting_and_blow-ups}
	In view of  Propositions \ref{Definition_of_G_M(Lambda,L)} and
	\ref{Giraud's_lemma,_logarithmic_version}, we now consider 3-tuples
	$(\mathscr{M},\Lambda,\mathcal{L})$, where $\mathscr{M}$ is an
	$\mathcal{O}_{V}^{q}$-module on a connected $F$-finite regular scheme
	$V$, $\Lambda$ is a finite collection of hypersurfaces on $V$ with only
	normal crossings, and $\mathcal{L}$ is an invertible $\mathcal{O}_{V}$-ideal
	included in $\mathcal{I}(H)$ for all $H\in\Lambda$. We attach to such
	a triple $(\mathscr{M},\Lambda,\mathcal{L})$ the $q$-differential collection
	$\mathcal{G}(\mathscr{M},\Lambda,\mathcal{L})$ obtained in Proposition \ref{Definition_of_G_M(Lambda,L)}. We set
	\begin{equation*}
		\operatorname{Sing}(\mathcal{G}(\mathscr{M},\Lambda,\mathcal {L})):=\{x\in V:
		\eta_{x}(\mathcal{G}(\mathscr{M},\Lambda,\mathcal{L}))\geq q\}
		\subset V.
	\end{equation*}
	A \emph{permissible center} for $(\mathscr{M},\Lambda,\mathcal{L})$
	is a
	closed irreducible regular subscheme $Z\subset V$ included in
	$\operatorname{Sing}(\mathcal{G}(\mathscr{M},\Lambda,\mathcal
	{L}))$ such that
	$\Lambda$ has only normal crossings with $Z$. Let
	$V\xleftarrow{\pi} V_{1}\supset H_{1}$ be the blow-up of $V$ along
	$Z$. Given a positive integer $a$, we define the \emph{$a$-transform of
		$(\mathscr{M},\Lambda,\mathcal{L})$} as the 3-tuple
	$(\mathscr{M}_{1}^{(a)},\Lambda_{1},\mathcal{L}_{1})$ on $V_{1}$, where
	\begin{enumerate}[(3)]
		\item[(1)] $\mathscr{M}_{1}^{(a)}:=(\mathscr{M}\mathcal
		{O}_{V_{1}}^{q})_{\mathcal{I}(H_{1})^{a}}$
		is the $a$-transform of $\mathscr{M}$ (Definition \ref{the_a-transform_of_an_O^q-module}),
		\item[(2)] $\Lambda_{1}$ is the collection of the strict transforms
		of the
		hypersurfaces in $\Lambda$ plus the exceptional hypersurface $H_{1}$,
		and
		\item[(3)] $\mathcal{L}_{1}:=(\mathcal{L}\mathcal
		{O}_{V_{1}})\mathcal{I}(H_{1})
		\subset\mathcal{O}_{V_{1}}$.
	\end{enumerate}
	Note that the family $\Lambda_{1}$ has only normal crossings and that
	$\mathcal{L}_{1}$ is included in $\mathcal{I}(H')$ for all
	$H'\in\Lambda_{1}$. Hence
	$(\mathscr{M}_{1},\Lambda_{1},\mathcal{L}_{1})$ is in the same situation
	as the original triple $(\mathscr{M},\Lambda,\mathcal{L})$. See Example \ref{final_example1}.
\end{parrafo}

\begin{proposition}%
	\label{inclusion_of_a-transforms_of_logarithmic_collections}
	Within the setting of \ref{logarithmic_setting_and_blow-ups}, assume in
	addition that
	$\eta_{x}(\mathcal{G}(\mathscr{M}, \Lambda,\mathcal{L}))\geq qa$
	for all
	$x\in Z$. Then $Z$ is permissible for $(\mathscr{M},a)$, and there is an
	inclusion
	\begin{align*}
		(\mathcal{G}(\mathscr{M},\Lambda,\mathcal{L})\mathcal{O}_{V_{1}})_{
			\mathcal{I}(H_{1})^{a}}&
		\subseteq\mathcal{G}(\mathscr {M}_{1},\Lambda_{1},
		\mathcal{L}_{1}),
	\end{align*}
	where the term on the left is the $a$-transform of
	$\mathcal{G}(\mathscr{M},\Lambda,\mathcal{L})$ (Definition \ref{the_transform_of_G}).
\end{proposition}
\begin{proof}
	Corollary \ref{inequality} implies that
	$\eta_{x}(\mathscr{M}):=\eta_{x}(\mathcal{G}(\mathscr{M}))\geq
	\eta_{x}(
	\mathcal{G}(\mathscr{M},\Lambda,\mathcal{L}))\geq qa$ for all $x\in
	Z$. Hence
	$Z$ is permissible for $(\mathscr{M},a)$ by Proposition  \ref{Describing_Sing(M,a)_with_eta(M)}. Next, we have to prove that
	\begin{equation*}
		((\Diff_{V,\Lambda,+}^{i}(\mathscr{M}):\mathcal{L}^{i})
		\mathcal{O}_{V_{1}}: \mathcal{I}(H_{1})^{qa})
		\subseteq(\Diff_{V_{1},\Lambda_{1},+}^{i}( \mathscr{M}_{1}):
		\mathcal{L}_{1}^{i})
	\end{equation*}
	for $i=1,\ldots,q-1$. Since $Z$ is permissible for $(\mathscr{M},a)$,
	Proposition \ref{Well_definition_of_the_transforms} shows that
	$\mathscr{M}\mathcal{O}_{V_{1}}^{q}\sim(\mathcal{I}(H_{1})^{(q)})^{a}
	\mathscr{M}_{1}$; therefore
	$\Diff_{V_{1},\Lambda_{1},+}^{i}(\mathscr{M}\mathcal{O}_{V_{1}}^{q})=
	\mathcal{I}(H_{1})^{qa}\Diff_{V_{1},\Lambda_{1},+}^{i}(\mathscr{M}_{1})$.
	We use this to prove the above inclusion:
	\begin{align*}
		&((\Diff_{V,\Lambda,+}^{i}(\mathscr{M}):\mathcal{L}^{i})
		\mathcal{O}_{V_{1}}: \mathcal{I}(H_{1})^{qa})
		\\
		&\quad
		\subseteq((\Diff_{V,\Lambda
			,+}^{i}(\mathscr{M})
		\mathcal{O}_{V_{1}}:\mathcal{L}^{i}\mathcal{O}_{V_{1}}):
		\mathcal {I}(H_{1})^{qa})
		\\
		&\quad =(\Diff_{V,\Lambda,+}^{i}(\mathscr{M})\mathcal
		{O}_{V_{1}}:(\mathcal{L}^{i} \mathcal{O}_{V_{1}})
		\mathcal{I}(H_{1})^{qa})
		\\
		&\quad =(\mathcal{I}(H_{1})^{i}(\Diff_{V,\Lambda,+}^{i}(
		\mathscr {M})\mathcal{O}_{V_{1}}):( \mathcal{L}\mathcal{O}_{V_{1}})^{i}
		\mathcal{I}(H_{1})^{qa}\mathcal {I}(H_{1})^{i})
		\\
		&\quad \stackrel{\text{(\ref
				{Giraud's_lemma,_logarithmic_version})}} {\subseteq}( \Diff_{V_{1},\Lambda_{1},+}^{i}(
		\mathscr{M}\mathcal{O}_{V_{1}}^{q}): \mathcal{L}_{1}^{i}
		\mathcal{I}(H_{1})^{qa})
		\\
		&\quad =(\mathcal{I}(H_{1})^{qa}\Diff_{V_{1},\Lambda_{1},+}^{i}(
		\mathscr{M}_{1}): \mathcal{L}_{1}^{i}
		\mathcal{I}(H_{1})^{qa})
		\\
		&\quad =(\Diff_{V_{1},\Lambda_{1},+}^{i}(\mathscr{M}_{1}):
		\mathcal{L}_{1}^{i}).
	\end{align*}
	\end{proof}
\begin{corollary}%
	\label{eta(G1)_logarithmic_is_at_most_eta(G)}
	Within the setting of \ref{logarithmic_setting_and_blow-ups}, assume in
	addition that $a\geq1$ and that there exist $0\leq b<q$ such that
	$\eta_{x}(\mathcal{G}(\mathscr{M},\Lambda,\mathcal{L}))=aq+b$ for all
	$x\in Z$. Then $Z$ is permissible for $(\mathscr{M},a)$, and
	\begin{equation*}
		\eta_{\pi(x_{1})}(\mathcal{G}(\mathscr{M},\Lambda,\mathcal {L}))\geq
		\eta_{x_{1}}(\mathcal{G}(\mathscr{M}_{1},
		\Lambda_{1},\mathcal{L}_{1})), \quad \forall
		x_{1}\in V_{1}.
	\end{equation*}
\end{corollary}
\begin{proof}
	The fact that $Z$ is permissible for $(\mathscr{M},a)$ was already noted
	in Proposition \ref{inclusion_of_a-transforms_of_logarithmic_collections}. Next, by Theorem \ref{fundamental_point-wise_inequality_for_eta(G)} we have that
	$\eta_{\pi(x_{1})}(\mathcal{G}(\mathscr{M},\Lambda,\mathcal
	{L}))\geq
	\eta_{x_{1}}((\mathcal{G}(\mathscr{M},\Lambda,\mathcal{L})\mathcal
	{O}_{V_{1}})_{
		\mathcal{I}(H_{1})^{a}})$ for $x_{1}\in V_{1}$, and from Proposition \ref{inclusion_of_a-transforms_of_logarithmic_collections} we deduce the
	inequality
	$\eta_{x_{1}}((\mathcal{G}(\mathscr{M},\Lambda,\mathcal
	{L})\mathcal{O}_{V_{1}})_{
		\mathcal{I}(H_{1})^{a}})\geq\eta_{x_{1}}(\mathcal{G}(\mathscr{M}_{1},
	\Lambda_{1},\mathcal{L}_{1}))$, $\forall x_{1}\in V_{1}$. These two
	inequalities
	imply the corollary.
\end{proof}

\begin{remark}%
	\label{setting_for_logarithmic}
	We use our previous results to define ``permissible'' sequences of
	transformations
	of a module by using functions satisfying the pointwise inequality:
	
	Fix an $\mathcal{O}_{V}^{q}$-submodule
	$\mathscr{M}\subseteq\mathcal{O}_{V}$. We attach to $\mathscr{M}$
	the 3-tuple
	$(\mathscr{M}_{0},\Lambda_{0},\mathcal{L}_{0}):=(\mathscr
	{M},\emptyset,
	\mathcal{O}_{V})$ with its associated $q$-differential collection
	$\mathcal{G}(\mathscr{M}_{0},\Lambda_{0},\mathcal{L}_{0})=\mathcal{G}(
	\mathscr{M})$. From this collection and its associated function $\eta$
	we can define sequences of transformations of the triple
	$(\mathscr{M}_{0},\Lambda_{0},\mathcal{L}_{0})$, say
	\begin{equation}
		\label{50342} \xymatrix@R=0pt @C=30pt { V_{0} & \ar[l]_-{\pi_{1}} V_{1} & \ar[l]_-{
				\pi_{2}} \ldots& \ar[l]_-{\pi_{r}} V_{r}%
			\\
			(\mathscr{M}_{0},\Lambda_{0},\mathcal{L}_{0})& (\mathscr{M}_{1},
			\Lambda_{1},\mathcal{L}_{1}) & \ldots& (\mathscr{M}_{r},\Lambda_{r},
			\mathcal{L}_{r}) }
	\end{equation}
	constructed as follows. For $i=0,\ldots, r-1$,
	$V_{i}\leftarrow V_{i+1}$ is a blowup with a permissible center
	$Z_{i}$ for the triple
	$(\mathscr{M}_{i},\Lambda_{i},\mathcal{L}_{i})$, along which the function
	$x\mapsto\eta_{x}(\mathcal{G}(\mathscr{M}_{i},\Lambda_{i},\mathcal
	{L}_{i}))$
	reaches its maximum value, say $a_{i}q+b_{i}$ (with $0\leq b_{i}<q$ and
	$a_{i}\geq1$). The triple
	$(\mathscr{M}_{i+1},\Lambda_{i+1},\mathcal{L}_{i+1})$ is the
	$a_{i}$-transform
	of $(\mathscr{M}_{i},\Lambda_{i},\mathcal{L}_{i})$.
	
	By Corollary \ref{eta(G1)_logarithmic_is_at_most_eta(G)} the functions
	$\eta$ defined from the $q$-differential collections
	$\mathcal{G}(\mathscr{M}_{i},\Lambda_{i},\mathcal{L}_{i})$ satisfy
	the fundamental
	pointwise inequality; in particular,
	\begin{align}
		\label{last_inequality} a_{1}q+b_{1}\geq
		a_{2}q+b_{2}\geq a_{3}q+b_{3}
		\geq\cdots.
	\end{align}
	By the same corollary, $Z_{i}$ is permissible for
	$(\mathscr{M}_{i},a_{i})$.
\end{remark}

\begin{example}%
	\label{final_example1}
	Let $V=\Spec(\mathbb{F}_{3}[x_{1},x_{2},x_{3},x_{4},x_{5}])$ and
	$q=p=3$. We consider
	$\mathscr{M}:=\mathcal{O}_{V}^{3}\cdot x_{1}x_{2}x_{3}x_{4}x_{5}$,
	$\Lambda:=\emptyset$, and $\mathcal{L}:=\mathcal{O}_{V}$. Then
	\begin{eqnarray*}
		\mathcal{G}(\mathscr{M},\Lambda,\mathcal{L})&=&\mathcal{G}(\mathscr{M})
		\\
		&=& \biggl(
		\biggl\langle\frac{x_{1}x_{2}x_{3}x_{4}x_{5}}{x_{j}}: j=1, \ldots,5 \biggr\rangle, \biggl\langle
		\frac
		{x_{1}x_{2}x_{3}x_{4}x_{5}}{x_{i}x_{j}}: 1\leq i<j\leq5 \biggr\rangle \biggr).
	\end{eqnarray*}
	The maximum value of
	$\eta_{x}(\mathcal{G}(\mathscr{M},\Lambda,\mathcal{L}))=\eta_{x}(
	\mathscr{M})$ is $5=3\cdot1+2$ (so that $a=1$), and this maximum is attained
	only when $x$ is the origin.
	
	We blowup $V$ at the origin and look at the affine chart
	$V_{1}:=\Spec(\mathbb{F}_{3}[x_{1},x_{2}', x_{3}',x_{4}',x_{5}'])$, where
	$x_{i}'=\frac{x_{i}}{x_{1}}$ for $i=2,\ldots,5$. The 1-transform of
	$\mathscr{M}$ (on $V_{1}$) is
	$\mathscr{M}_{1}:=\mathcal{O}_{V_{1}}^{3}\cdot
	x_{1}^{2}x_{2}'x_{3}'x_{4}'x_{5}'$.
	Note also that $\Lambda_{1}=\{\mathcal{O}_{V_{1}}\cdot x_{1}\}$ and
	$\mathcal{L}_{1}=\mathcal{O}_{V_{1}}x_{1}$. A~straightforward computation
	yields
	\begin{align*}
		\mathcal{G}(\mathscr{M}_{1},\Lambda_{1},
		\mathcal{L}_{1})={}& \biggl( \biggl( \biggl\langle x_{1}^{2}
		\frac{x'_{2}x'_{3}x'_{4}x'_{5}}{x'_{j}}: j=2,\ldots,5 \biggr\rangle:\mathcal{O}_{V_{1}}
		\cdot x_{1} \biggr),
		\\
		&{} \biggl( \biggl\langle x_{1}^{2}
		\frac{x'_{2}x'_{3}x'_{4}x'_{5}}{x'_{i}x'_{j}}: 2\leq i<j\leq5 \biggr\rangle:\mathcal{O}_{V_{1}}
		\cdot x_{1}^{2} \biggr) \biggr)
		\\
		={}& \biggl( \biggl\langle x_{1}\frac{x'_{2}x'_{3}x'_{4}x'_{5}}{x'_{j}}: j=2, \ldots,5
		\biggr\rangle, \biggl\langle\frac
		{x'_{2}x'_{3}x'_{4}x'_{5}}{x'_{i}x'_{j}}: 2\leq i<j\leq5 \biggr\rangle
		\biggr).
	\end{align*}
	The maximum value of
	$\eta_{x}(\mathcal{G}(\mathscr{M}_{1},\Lambda_{1},\mathcal
	{L}_{1}))$ along
	points of $V_{1}$ is $4=3\cdot1+1$ and is attained along the regular subscheme
	$Z_{1}$ defined by $\langle x_{2}',x_{3}',x_{4}',x_{5}'\rangle$, whereas
	the maximum value of
	$\eta_{x}(\mathcal{G}(\mathscr{M}_{1}))=\eta_{x}(\mathscr{M}_{1})$ is
	$6=3\cdot2$, and it is only attained at the origin.
	
	We blow up $V_{1}$ along $Z_{1}$ and look at the affine chart
	$V_{2}:=\Spec(\mathbb{F}_{3}[x_{1},x_{2}',x_{3}'', x_{4}'',x_{5}''])$,
	where $x_{i}'':=\frac{x_{i}'}{x_{2}'}$ for $i=3,4,5$. The 1-transform of
	$\mathscr{M}_{1}$ (on $V_{2}$) is
	$\mathscr{M}_{2}:=\mathcal{O}_{V_{2}}^{3}\cdot
	x_{1}^{2}x_{2}'x_{3}''x_{4}''x_{5}''$.
	Note also that
	$\Lambda_{2}=\{\mathcal{O}_{V_{2}}\cdot x_{1},\mathcal{O}_{V_{2}}
	\cdot x_{2}\}$ and that
	$\mathcal{L}_{2}=\mathcal{O}_{V_{2}}\cdot x_{1}x_{2}'$. A~straightforward
	computation yields
	\begin{align*}
		\mathcal{G}(\mathscr{M}_{2},\Lambda_{2},
		\mathcal{L}_{2})={}& \biggl( \biggl( \biggl\langle x_{1}^{2}x_{2}'
		\frac{x''_{3}x''_{4}x''_{5}}{x''_{j}}: j=3,4,5 \biggr\rangle:\mathcal {O}_{V_{2}} \cdot
		x_{1}x_{2}' \biggr),
		\\
		&{} ( \langle
		x_{1}^{2}x_{2}'x''_{j}:
		j=3,4, 5 \rangle:\mathcal{O}_{V_{2}}\cdot x_{1}^{2}x_{2}^{\prime 2}
		) \biggr)
		\\
		={}& \biggl( \biggl\langle x_{1}\frac{x''_{3}x''_{4}x''_{5}}{x''_{j}}: j=3,4,5 \biggr
		\rangle, \langle x_{3}'',x_{4}'',x_{5}''
		\rangle \biggr).
	\end{align*}
	The maximum value of
	$\eta_{x}(\mathcal{G}(\mathscr{M}_{2},\Lambda_{2},\mathcal
	{L}_{2}))$ is
	3, and it is attained along the regular subscheme $Z_{2}$ defined by the
	ideal $\langle x''_{3},x''_{4},x''_{5}\rangle$. On the other hand, the
	maximum of
	$\eta_{x}(\mathcal{G}(\mathscr{M}_{2}))=\eta_{x}(\mathscr{M}_{2})$ is
	$6=3\cdot2$, and it is only attained at the origin.
	
	We blow up $V_{2}$ along $Z_{2}$ and look at the affine chart
	$V_{3}=\operatorname{Spec}(\mathbb
	{F}_{3}[x_{1},x_{2}',x_{3}'', x_{4}''',x_{5}'''])$,
	where $x_{i}''':=\frac{x_{i}''}{x_{3}''}$ for $i=4,5$. The 1-transform
	of $\mathscr{M}_{2}$ is
	$\mathscr{M}_{3}:=\mathcal{O}_{V_{3}}^{3}\cdot
	x_{1}^{2}x_{2}'x_{4}'''x_{5}'''$.
	Note also that
	$\Lambda_{3}=\{\mathcal{O}_{V_{3}}\cdot x_{1},\mathcal{O}_{V_{3}}
	\cdot x_{2}',\mathcal{O}_{V_{3}}\cdot x_{3}''\}$ and that
	$\mathcal{L}_{3}=\mathcal{O}_{V_{3}}\cdot x_{1}x_{2}'x_{3}''$.
	A~straightforward
	computation yields
	\begin{align*}
		&\mathcal{G}(\mathscr{M}_{3},\Lambda_{3},
		\mathcal{L}_{3})
		\\
		&\quad = ( ( \langle x_{1}^{2}x_{2}'x'''_{j},
		j=4,5 \rangle: \mathcal{O}_{V_{2}}\cdot x_{1}x_{2}'
		x_{3}'' ), ( \langle
		x_{1}^{2}x_{2}' \rangle:
		\mathcal{O}_{V_{2}}\cdot x_{1}^{2}x_{2}^{\prime 2}x_{3}^{\prime\prime 2}
		) )
		\\
		&\quad = ( \langle x_{1}{x'''_{j}}:
		j=4,5 \rangle, \mathcal {O}_{V_{3}} ).
	\end{align*}
	
	The process stops here, at least over this chart, as now the maximum of
	$\eta_{x}(\mathcal{G}_{3}(\mathscr{M}_{3},\Lambda_{3},\mathcal{L}_{3}))$
	is $2<3$. Moreover, the maximum of $\eta_{x}(\mathscr{M}_{3})$ is still
	5, and hence the $V_{3}$-scheme, say
	\begin{equation*}
		X_{3}=\operatorname{Spec}((\mathcal{O}^{3}_{V_{3}}[
		\mathscr{M}_{3}])^{1/3}),
	\end{equation*}
	still has points of multiplicity 3. In fact, the dimension of the closed
	set of points of multiplicity 3 of $X_{3}$ is higher than the dimension
	of those of the original scheme, say
	$X=\Spec((\mathcal{O}^{3}_{V}[\mathscr{M}])^{1/3})$ (recall here that
	$X$ has highest multiplicity $q=3$ and that $X_{3}$ is obtained form
	$X$ by blowing up successively at regular equimultiple centers of multiplicity
	3).
	
	Nevertheless, at least in this example, points of multiplicity 3 of
	$X_{3}$ are easier to deal with than those of the original scheme
	$X$. In fact, we can eliminate points of multiplicity 3 of $X_{3}$ by simply
	blowing up at the higher-dimensional components of the closed set of points
	of multiplicity 3, namely by blowing up first at the components of codimension
	2, followed by a blowup at a component of codimension 3.
\end{example}

We hope that the invariants developed in this paper, defined in terms of
differentials and of logarithmic differentials, will eventually lead to
a simplification of the highest multiplicity locus and hence to a reduction
of points of multiplicity $q$.



\section*{Acknowledgement}
	We thank C. Abad, A. Benito, A. Bravo, and S. Encinas for stimulating
	discussions on these questions. We are very grateful to the referee, whose
	comments helped us to significantly improve the exposition. The first author
	gratefully acknowledges support by CONICET (Argentina) in the form of a
	Postdoctoral
	Fellowship and also Universidad Aut\'{o}noma de Madrid and ICMAT (Madrid)
	for the support and for excellent working conditions.

	\bibliography{References}
	\bibliographystyle{abbrv}
	
\end{document}